\PassOptionsToPackage{unicode}{hyperref}

\PassOptionsToPackage{hyphens}{url}
\PassOptionsToPackage{dvipsnames,svgnames,x11names}{xcolor}
\documentclass[12pt]{article}
\usepackage[hypertexnames=false]{hyperref}
\usepackage{amsmath,amssymb}
\usepackage{iftex}
\ifPDFTeX

  \usepackage[T1]{fontenc}
  \usepackage[utf8]{inputenc}
  \usepackage{textcomp} % provide euro and other symbols
\else % if luatex or xetex
  \usepackage{unicode-math}
  \defaultfontfeatures{Scale=MatchLowercase}
  \defaultfontfeatures[\rmfamily]{Ligatures=TeX,Scale=1}
\fi
\usepackage{lmodern}
\ifPDFTeX\else  
\fi
\IfFileExists{upquote.sty}{\usepackage{upquote}}{}
\IfFileExists{microtype.sty}{% use microtype if available
  \usepackage[]{microtype}
  \UseMicrotypeSet[protrusion]{basicmath} % disable protrusion for tt fonts
}{}
\makeatletter
\@ifundefined{KOMAClassName}{% if non-KOMA class
  \IfFileExists{parskip.sty}{%
    \usepackage{parskip}
  }{% else
    \setlength{\parindent}{0pt}
    \setlength{\parskip}{6pt plus 2pt minus 1pt}}
}{% if KOMA class
  \KOMAoptions{parskip=half}}
\makeatother
\usepackage{xcolor}
\makeatletter
\ifx\paragraph\undefined\else
  \let\oldparagraph\paragraph
  \renewcommand{\paragraph}{
    \@ifstar
      \xxxParagraphStar
      \xxxParagraphNoStar
  }
  \newcommand{\xxxParagraphStar}[1]{\oldparagraph*{#1}\mbox{}}
  \newcommand{\xxxParagraphNoStar}[1]{\oldparagraph{#1}\mbox{}}
\fi
\ifx\subparagraph\undefined\else
  \let\oldsubparagraph\subparagraph
  \renewcommand{\subparagraph}{
    \@ifstar
      \xxxSubParagraphStar
      \xxxSubParagraphNoStar
  }
  \newcommand{\xxxSubParagraphStar}[1]{\oldsubparagraph*{#1}\mbox{}}
  \newcommand{\xxxSubParagraphNoStar}[1]{\oldsubparagraph{#1}\mbox{}}
\fi
\makeatother

\usepackage{longtable,booktabs,array}
\usepackage{calc} % for calculating minipage widths
\usepackage{etoolbox}
\makeatletter
\patchcmd\longtable{\par}{\if@noskipsec\mbox{}\fi\par}{}{}
\makeatother
\IfFileExists{footnotehyper.sty}{\usepackage{footnotehyper}}{\usepackage{footnote}}
\makesavenoteenv{longtable}
\usepackage{graphicx}
\makeatletter
\def\maxwidth{\ifdim\Gin@nat@width>\linewidth\linewidth\else\Gin@nat@width\fi}
\def\maxheight{\ifdim\Gin@nat@height>\textheight\textheight\else\Gin@nat@height\fi}
\makeatother
\setkeys{Gin}{width=\maxwidth,height=\maxheight,keepaspectratio}
\makeatletter
\def\fps@figure{htbp}
\makeatother

\usepackage[left=1in,top=0.8in,right=1in,bottom=0.7in,head=.2in,nofoot]{geometry}

\usepackage[font=small,labelfont=bf]{caption}

\makeatletter
\@ifpackageloaded{caption}{}{\usepackage{caption}}
\AtBeginDocument{%
\ifdefined\contentsname
  \renewcommand*\contentsname{Table of contents}
\else
  \newcommand\contentsname{Table of contents}
\fi
\ifdefined\listfigurename
  \renewcommand*\listfigurename{List of Figures}
\else
  \newcommand\listfigurename{List of Figures}
\fi
\ifdefined\listtablename
  \renewcommand*\listtablename{List of Tables}
\else
  \newcommand\listtablename{List of Tables}
\fi
\ifdefined\figurename
  \renewcommand*\figurename{Figure}
\else
  \newcommand\figurename{Figure}
\fi
\ifdefined\tablename
  \renewcommand*\tablename{Table}
\else
  \newcommand\tablename{Table}
\fi
}
\@ifpackageloaded{float}{}{\usepackage{float}}
\floatstyle{ruled}
\@ifundefined{c@chapter}{\newfloat{codelisting}{h}{lop}}{\newfloat{codelisting}{h}{lop}[chapter]}
\floatname{codelisting}{Listing}

\makeatother
\makeatletter
\@ifpackageloaded{caption}{}{\usepackage{caption}}
\@ifpackageloaded{subcaption}{}{\usepackage{subcaption}}
\makeatother

\ifLuaTeX
  \usepackage{selnolig}  % disable illegal ligatures
\fi
\usepackage[]{natbib}
\usepackage{bookmark}

\IfFileExists{xurl.sty}{\usepackage{xurl}}{} % add URL line breaks if available
\hypersetup{
  pdftitle={Characterizing Identifiability in Boolean Factor Models},
  pdfauthor={Mengqi Lin and Gongjun Xu},
  pdfkeywords={Graphical structures, Cognitive diagnosis models, Boolean matrix decompositions, Hasse diagrams, Graph isomorphism, Boolean satisfiability (SAT)},
  colorlinks=true,
  allcolors=black,
  pdfcreator={LaTeX}}

\usepackage{xcolor}
\newcommand{\anon}{1}

\usepackage{mycommands}

\theoremstyle{plain}

\newtheorem{theorem}{Theorem}
\newtheorem{lemma}{Lemma}
\newtheorem{proposition}{Proposition}
\newtheorem{corollary}{Corollary}

\newtheorem{definition}{Definition}

\newtheorem{example}{Example}

\newtheorem{remark}{Remark}

\newtheorem{fact}{Fact}

\usepackage[nodisplayskipstretch]{setspace}

\hypersetup{colorlinks=true,allcolors=black}

\theoremstyle{remark}
\newtheorem{supplementfact}[fact]{Fact}
\theoremstyle{plain}

\begin{document}

\def\spacingset#1{\renewcommand{\baselinestretch}%
{#1}\small\normalsize} \spacingset{1}

%%%%%%%%%%%%%%%%%%%%%%%%%%%%%%%%%%%%%%%%%%%%%%%%%%%%%%%%%%%%%%%%%%%%%%%%%%%%%%

\if1\anon
{
  \title{\bf Characterizing Identifiability in Boolean Factor Models}
  \date{ }
   \author{Mengqi Lin and  Gongjun Xu\\
     \hspace{.2cm} \\
     Department of Statistics, University of Michigan}
  \maketitle
} \fi

\if0\anon
{
  \bigskip
  \bigskip
  \bigskip
  \begin{center}
    {\LARGE\bf Characterizing Identifiability in Boolean Factor Models}
\end{center}
  \medskip
} \fi

\bigskip
\begin{abstract}
Boolean factor models, including prominent subfamilies such as Boolean matrix decompositions and cognitive diagnosis models, find broad applications ranging from social sciences to engineering. Despite their flexibility, a key challenge lies in establishing the identifiability of their graphical structures, which specify how latent variables influence observed variables. Existing identifiability conditions typically rely on the strong assumption of pure nodes, which may be unrealistic in many applications. We develop a novel approach leveraging the Hasse diagram to represent the distribution of observed variables and transform identifiability into a graph isomorphism challenge. Based on this, we establish {\it sufficient and necessary} graphical identifiability conditions that do not require pure nodes. We further derive equivalent algebraic conditions and develop an efficient  Boolean satisfiability (SAT)-based verification algorithm. {\color{black} We extend the analysis to probabilistic Boolean factor models, establishing conditions for jointly identifying the graphical structure and additional model parameters without requiring pure nodes.} Our results substantially broaden the class of identifiable and interpretable Boolean factor models by removing the pure-node requirement, yielding new theoretical insights, while also providing practitioners with a concrete and easily implementable tool to assess model identifiability.
\end{abstract}

% Our results extend the applicability of Boolean factor models, offering new theoretical insights and practical tools. 

% This perspective leads to a sharp characterization: the graphical structure is identifiable precisely when its representation is unique, demonstrating that the traditional pure-node assumption is not necessary. From this, we derive both graphical and algebraic  sufficient and necessary conditions for identifiability. Exploiting the Boolean nature of the models, these conditions can be encoded entirely using logical operations, enabling an efficient Boolean satisfiability (SAT)–based algorithm for verification. 

\noindent%
\vfill

% Data Availability Statement: the code and data that support the findings of this study are available at \url{https://github.com/Mengqi-Lin/idQ}.	

%Competing Interests: No potential competing interest to report. 

%This research is partially supported by the National Science Foundation. 
%%(SES 1846747).
\newpage
\spacingset{1.8} % DON'T change the spacing!

\section{Introduction}
Boolean factor models are a powerful tool in statistics with various applications in engineering, machine learning, education, psychology, and biology, where they are used to uncover hidden latent structures from observed data. 
% ultimately enabling interpretable analysis of complex real-world phenomena.
On one hand, the use of Boolean factor models is particularly advantageous due to their natural alignment with many forms of data that are inherently \emph{discrete}.
On the other hand, 
the graphical structures of the models, which characterize the dependencies between different variables,  enhance both interpretability and flexibility of the models. 
One popularly used family of submodels is \emph{Boolean matrix decomposition} (or Boolean matrix factorization), which has broad applications including but not limited to role engineering \citep{vaidya2007role, zhang2007role}, text mining \citep{geerts2004tiling}, and machine learning \citep{wicker2012multi, rava2016boolean}.  
Another widely used family of submodels, the restricted latent attribute models, also known as \emph{cognitive diagnosis models}, has also attracted increasing interest in cognitive diagnosis in educational assessments \citep{junker2001cognitive,Templin,rupp2010diagnostic,de2011generalized}, psychiatric evaluation \citep{jaeger2006distinguishing, dela2018}, and
disease etiology diagnosis \citep{wu2017nested, wu2018}.
Beyond these, other notable submodels include Noisy-Or Bayesian networks \citep{halpern2013noisyor, shwe1991qmrdt} and restricted Boltzmann machines \citep{hinton2006fast, goodfellow2016deep}.
%in machine learning, both of which continue to attract considerable attention today.

{
\color{black}
Although these models are widely used, their identifiability, especially that of their graphical structures, remains an important and challenging problem} \citep{JLGXZY2011,chen2015statistical,xu2018, gu19, Culpepper2019,culpepper2023note,kong2023identification}.
In many applications, the graphical structures carry domain-specific scientific interpretations of the models, and uniquely identifying them is the prerequisite to ensure valid interpretations and guarantee that the inference made about the latent structures and relationships within the data is reliable.
For instance, in Boolean matrix decomposition, specifically its application in 
role engineering \citep{vaidya2007role}, where the observed variables and the latent variables represent a user's permissions and the roles of the user, the graphical structure encodes which permissions each role grants and is essential for accurately identifying the user’s roles.
% text mining \citep{geerts2004tiling}, where the observed variables and the latent variables represent words in a document and the topics the document is categorized into, the graphical structure captures the relationship between topics and the presence of words, which is crucial for accurately identifying which topics are associated with each document. 
In cognitive diagnosis, where the observed variables and the latent variables are commonly a subject's responses to test items (e.g., exam questions) and the attributes (e.g., the mastery or deficiency of some skills) to be inferred \citep{junker2001cognitive,davier2008general,henson2009defining}, the graphical structure in a model characterizes which attributes are required for each test item, and is vital for inferring individuals' fine-grained attributes.

% Notably, identifiability of the graphical structure, in the context of Boolean matrix decomposition, corresponds to the uniqueness of the decomposition \citep{desantis2022factorization}. \todo{remove this sentence?}
%\citep{mchugh1956efficient, teicher1967identifiability, goodman1974exploratory,DiBello,MarisBechger,TatsuokaC09, xu2017,gu19}.  
% When a model is identifiable, it ensures that there is a one-to-one correspondence between the parameter values and the probability distribution defined by the model. 

%The challenge of identifiability of the graphical structure arises from its nature as an graph isomorphism problem in combinatorics, a notoriously difficult issue in the literature.
%The challenge of identifiability of the graphical structure lies in the fact that the problem is essentially an isomorphism problem in combinactorics, which is known to be challenging in the literature.

The challenge of identifiability of the graphical structures mainly lies in the stringent assumptions often imposed by existing literature to achieve it, 
especially the requirement of \emph{pure nodes}.
Concretely, a pure node is an observed variable whose distribution depends directly on exactly one latent variable. In cognitive diagnosis models, a pure node is generally referred to as a diagnostic test item that measures exactly one latent attribute. 
The existence of pure nodes, or the stronger condition that requires at least one pure node for each latent variable, is commonly assumed in the literature to ensure identifiability
\citep{chen2015statistical, xu2016,xu2017, gu19, gu21, fang2019, Culpepper2019,lin2024polytomousDINA}. These assumptions, while mathematically convenient, may be unrealistic in practical applications, limiting the models' applicability and robustness. 
In practice, correlated latent variables frequently exist and are often jointly required by the same observed variable, making the pure nodes assumption difficult to satisfy. 
This challenge has been recognized by many practitioners \citep{de2004higher,davier2008general, TatsuokaC09, decarlo2011analysis, MarisBechger,davier2014dina}. Furthermore, beyond the scope of the Boolean factor models, similar pure node assumptions also appear commonly in the identifiability theories in many other latent variable models to ensure model identifiability. For instance, the \emph{pure node} assumption in mixed membership models \citep{jiang2024optimal,gu2023dimension}, the \emph{anchor word}  assumption in topic models \citep{arora2012learning, arora2013practical,Ke02012024}, and the  \emph{separability condition} in nonnegative matrix factorization literature \citep{papadimitriou2000latent,mcsherry2001spectral,donoho2003does, azar2001spectral}.
% removed citations: ,kleinberg2008using,kleinberg2003convergent

%As \citet{heinen1996latent} Since continuous distribution can be approximated by discrete distribution, the  we focus on the discrete distributions

% as powerful statistical and machine learning tools, are popularly used in many disciplines, such as education, psychology, sociology, and other fields like economics, medicine, and engineering. Generally, these models aim to infer individuals' latent attributes based on their observed responses to a set of items. 

% Popular SLAMs are, the identifiability of SLAMs has long been an issue in the literature. \citep[e.g.,][]{davier2008general, decarlo2011analysis, MarisBechger, davier2014dina}. 
%Recently, Boolean matrix decomposition also finds its application to address multi-label classification. The decomposition can serve as a useful mean to extract latent labels representing useful boolean combinations that capture the most important information of the original labels along with their dependencies of the original labels from the original label matrix \citep{wicker2012multi}.
{\color{black}
\vspace{-0.4cm}
}

{\color{black}
\subsection{Our Contributions}}
Our goal is to fill this gap by establishing practical and valid conditions for identifiability of the graphical structures {without} pure nodes.
Our main contributions are threefold.

\begin{enumerate}

\item \textbf{A graphical characterization under general Boolean
operators.}
We consider a broad class of Boolean factor models that allow highly
flexible interactions between observed and latent variables,
represented by general Boolean operators. We show that each such
operator induces a partial order on the latent space and use this
order to construct a Hasse-based distribution graph that represents
the distribution of the observed variables. This construction
transforms the identifiability of the graphical structure into the
\emph{uniqueness of the graphical representation}, allowing us to
formulate identifiability as a graph isomorphism problem. With this
tool, we establish a sufficient and necessary graphical condition
for identifying the graphical structure without requiring pure nodes.

\item \textbf{Algebraic characterization and SAT-based verification.}
We derive an algebraic condition equivalent to the graphical
characterization and establish structural reductions under commonly
used Boolean operators. Leveraging the Boolean structure of the
problem, we express the condition using logical operations and
develop a Boolean satisfiability (SAT)-based algorithm for its
verification, providing practitioners with a practical tool to
assess model identifiability. Simulation studies show that many
graphical structures without pure nodes are identifiable and
demonstrate the computational efficiency of our algorithm.

% Structural reductions
{\color{black}
\item \textbf{Joint identifiability for probabilistic models.}
We extend our results to probabilistic Boolean factor models with
additional parameters that account for uncertainty in the observed variables. We establish conditions for
identifiability of the graphical structure and for joint
identifiability of the graphical structure and other model
parameters. These conditions do not require pure nodes, substantially relaxing the requirements imposed by existing identifiability results.}
\end{enumerate}

The rest of the paper is organized as follows. In
Section~\ref{sec:setup}, we present the model setup and formulate
the identifiability problem. In Section~\ref{sec:graph}, we develop
a graphical representation of the distribution of the observed
variables and establish a sufficient and necessary graphical
condition for identifying the graphical structure. We further
derive an equivalent algebraic condition and show how it can be
verified efficiently through an SAT formulation. In
Section~\ref{sec:prob}, we extend the results to
more general models. 
%All the results do not require pure nodes, substantially relaxing the requirements commonly imposed in existing work. 
Future research directions are discussed in Section~\ref{sec: discussion}.
All proofs are provided in the Supplementary Material.
\\
{\color{black}
\textbf{Notation.}
}
Given a positive integer $K $, let $[K]$ denote the set of positive integers from 1 to $K$, i.e., $[K]=\set{1,\ldots,K}$. Let $\{0,1\}^K$ denote the binary vector space of $K$ dimensions. 
Denote the vector where only the $j$-th entry is 1 as $\ee_j = (0, \ldots, 1, 0, \ldots, 0)^\top$. The vector of all ones and the vector of all zeros are respectively denoted by $\one = (1,\ldots,1)^\top$ and $\zero = (0, \ldots, 0)^\top$. 
Given another positive integer $J$, let $\{0,1\}^{J \times K}$ denote the $J \times K$-dimensional binary matrix space, and we use $\cI_K$ to represent the $K$-dimensional identity matrix. Use \(\1\{\cdot\}\) to denote the indicator function.
\vspace{-0.4cm}
\section{Model Setup and Identifiability}\label{sec:setup}
\vspace{-0.2cm}
\subsection{Boolean Factor Models}\label{sec:BOM}

Consider a general Boolean factor model with \(J\) binary observed variables \(\RR = (R_1,\ldots,R_J)^\top\), each \(R_j\in\{0,1\}\) for \(j \in [J]\), and suppose that there are $K$ binary latent variables that are direct parents to $\RR$, which we denote by \(\aaa = (\alpha_1, \ldots, \alpha_K)^\top\),  each \(\alpha_k \in \{0,1\}\) for \(k \in [K]\). 
We allow a flexible joint distribution for these latent variables \(\aaa\). They may exhibit direct or indirect dependencies among themselves (e.g., $\aaa$ in Figure~\ref{fig:models}(a)) or be influenced by higher-level latent variables  (e.g., $\aaa$ in Figure~\ref{fig:models}(b)), that is, beyond \(\aaa\), additional latent variables (could be continuous or discrete) may generate or correlate these \(\alpha\)-variables, for instance, through a hierarchical structure 
(e.g., $\beta$-variables in Figure~\ref{fig:models}(b)).

\vspace{-0.1cm}
\begin{figure}[htbp]
\centering
\begin{minipage}{0.45\textwidth}
    \centering
    \begin{tikzpicture}[scale=.8]

  % Node Styles
  \tikzstyle{observed}=[circle, draw, fill=white, minimum size=20pt, inner sep=0pt]
  \tikzstyle{latent}  =[circle, draw, fill=white, minimum size=20pt, inner sep=0pt]
  \tikzstyle{qedge}   =[draw, thick, -latex]

  %--------------------------------------------------
  % 1) Observed variables R_1,..., R_6 (Bottom)
  %--------------------------------------------------
  \node[observed] (R1) at (0,   0) {$R_1$};
  \node[observed] (R2) at (1.2, 0) {$R_2$};
  \node[observed] (R3) at (2.4, 0) {$R_3$};
  \node[observed] (R4) at (3.6, 0) {$R_4$};
  \node[observed] (R5) at (4.8, 0) {$R_5$};
  \node[observed] (R6) at (6.0, 0) {$R_6$};

  %--------------------------------------------------
  % 2) Alpha layer (4 nodes) directly influencing R
  %--------------------------------------------------
  \node[latent] (a1) at (1.2, 1.5) {$\alpha_1$};
  \node[latent] (a2) at (2.4, 1.5) {$\alpha_2$};
  \node[latent] (a3) at (3.6, 1.5) {$\alpha_3$};
  \node[latent] (a4) at (4.8, 1.5) {$\alpha_4$};

  % Edges from alpha to R according to the QQ matrix
  \draw[qedge] (a1) -- (R1);
  \draw[qedge] (a2) -- (R1);
  \draw[qedge] (a1) -- (R2);
  \draw[qedge] (a3) -- (R2);
  \draw[qedge] (a1) -- (R3);
  \draw[qedge] (a4) -- (R3);
  \draw[qedge] (a2) -- (R4);
  \draw[qedge] (a3) -- (R4);
  \draw[qedge] (a2) -- (R5);
  \draw[qedge] (a4) -- (R5);
  \draw[qedge] (a3) -- (R6);
  \draw[qedge] (a4) -- (R6);

  \draw (a1) -- (a2);
  \draw (a2) -- (a3);
  \draw (a3) -- (a4);

  %--------------------------------------------------
  % Dashed box around only the alpha→R part
  %--------------------------------------------------
  \node[draw, dashed, thick, fit=(a1)(a2)(a3)(a4)(R1)(R2)(R3)(R4)(R5)(R6), inner sep=6pt] {};

  %--------------------------------------------------
  % High arcs connecting latent variables (drawn after the box)
  %--------------------------------------------------
  \draw[bend left=60, looseness=1.5] (a1) to (a3);
  \draw[bend left=60, looseness=1.5] (a2) to (a4);
  \draw[bend left=90, looseness=1.5] (a1) to (a4);

\end{tikzpicture}

\caption*{(a) {Model with internal dependencies among \\latent $\alpha$-variables, indicated by the edges connecting them.}}
\end{minipage}
\hfill
\begin{minipage}{0.45\textwidth}
    \begin{tikzpicture}[scale=.8]

% Node Styles
\tikzstyle{observed}=[circle, draw, fill=white, minimum size=20pt, inner sep=0pt]
\tikzstyle{latent}=[circle, draw, fill=white, minimum size=20pt, inner sep=0pt]
\tikzstyle{qedge}=[draw, thick, -latex]

%--------------------------------------------------
% 1) Observed variables R_1,..., R_6 (Bottom)
%--------------------------------------------------
\node[observed] (R1) at (0,   0) {$R_1$};
\node[observed] (R2) at (1.2, 0) {$R_2$};
\node[observed] (R3) at (2.4, 0) {$R_3$};
\node[observed] (R4) at (3.6, 0) {$R_4$};
\node[observed] (R5) at (4.8, 0) {$R_5$};
\node[observed] (R6) at (6.0, 0) {$R_6$};

%--------------------------------------------------
% 2) Alpha layer (4 nodes) directly influencing R
%--------------------------------------------------
\node[latent] (a1) at (1.2, 1.5) {$\alpha_1$};
\node[latent] (a2) at (2.4, 1.5) {$\alpha_2$};
\node[latent] (a3) at (3.6, 1.5) {$\alpha_3$};
\node[latent] (a4) at (4.8, 1.5) {$\alpha_4$};

% Edges from alpha to R according to the QQ matrix
% Row 1: R1 from alpha1, alpha2
\draw[qedge] (a1) -- (R1);
\draw[qedge] (a2) -- (R1);

% Row 2: R2 from alpha1, alpha3
\draw[qedge] (a1) -- (R2);
\draw[qedge] (a3) -- (R2);

% Row 3: R3 from alpha1, alpha4
\draw[qedge] (a1) -- (R3);
\draw[qedge] (a4) -- (R3);

% Row 4: R4 from alpha2, alpha3
\draw[qedge] (a2) -- (R4);
\draw[qedge] (a3) -- (R4);

% Row 5: R5 from alpha2, alpha4
\draw[qedge] (a2) -- (R5);
\draw[qedge] (a4) -- (R5);

% Row 6: R6 from alpha3, alpha4
\draw[qedge] (a3) -- (R6);
\draw[qedge] (a4) -- (R6);

%--------------------------------------------------

\draw (a1) -- (a2);
\draw (a2) -- (a3);
\draw (a3) -- (a4);

% 3) Beta^(1) layer with 3 nodes
%--------------------------------------------------
\node[latent] (b11) at (1.8, 3.0) {$\beta^{(1)}_1$};
\node[latent] (b12) at (3.0, 3.0) {$\beta^{(1)}_2$};
\node[latent] (b13) at (4.2, 3.0) {$\beta^{(1)}_3$};

% (Optionally, illustrate edges from \beta^(1) to alpha)
\draw[qedge] (b11) -- (a1);
\draw[qedge] (b12) -- (a2);
\draw[qedge] (b13) -- (a3);
\draw[qedge] (b13) -- (a4);

\node[draw, dashed, thick, fit=(a1)(a2)(a3)(a4)(R1)(R2)(R3)(R4)(R5)(R6), inner sep=6pt] {};

%--------------------------------------------------

% THIRD LAYER: A SINGLE NODE \beta^{(2)}_1
%------------------------------------------------------------------
\node[latent] (b21) at (3, 4.5) {$\beta^{(2)}_1$};

% Edges from \beta^{(2)}_1 to each \beta^{(1)} node
\draw[qedge] (b21) -- (b11);
\draw[qedge] (b21) -- (b12);
\draw[qedge] (b21) -- (b13);

\end{tikzpicture}\caption*{(b) \small{Model with additional latent $\beta$-variables, \\ which generate $\alpha$-variables.}}\label{model:hierarchical}
\end{minipage}\caption{{Boolean factor models.} The dashed boxes outline the  graphical structures $(\aaa \to \RR)$ of the two models, which are the focus of this study.} 
\label{fig:models}
\end{figure}
In this work, we focus on studying the \emph{graphical structure} of the Boolean factor model—the directed edges from \(\aaa\) to \(\RR\) (e.g., the dashed boxes in Figure~\ref{fig:models}), which characterizes the conditional distribution of $\RR$ given \(\aaa\). 
% Since edges exist only from $\aaa$ to $\RR$, $R_j$ for $j \in [J]$ are conditionally independent given \(\aaa\), i.e., \ \(\P(\RR \mid \aaa) = \prod_{j=1}^J \P(R_j \mid \QQ, \aaa)\). 
Such a bipartite graphical structure can be algebraically represented by a binary matrix, which we denote by $\QQ \in \{0,1\}^{J \times K}$. Specifically, let $\qq_j = (q_{j1}, \ldots, q_{jK}) \in \{0,1\}^K$ be the $j$-th row of $\QQ$. The presence or absence of an edge from $\alpha_k$ to $R_j$ is encoded by the $k$-th entry of $\qq_j$, for $k \in [K]$. Throughout, we impose the standard \emph{local independence} assumption:
\(\P(\RR \mid \aaa) = \prod_{j=1}^J \P(R_j \mid \QQ, \aaa)\).

Motivated by many popularly used Boolean factor models \citep{junker2001cognitive,Templin,geerts2004tiling}, we consider a generic Boolean operator $\ast: \{0,1\}^K \times \{0,1\}^K \;\to\; \{0,1\}$
to represent the dependency between $\aaa$ and $\RR$. 
%For each $j$, the output $\qq_j \ast \aaa \in \{0,1\}$ determines $\P(R_j \mid \QQ, \aaa)$. 
For instance, the basic and critical Boolean factor models, which have widespread applications in social sciences and engineering \citep{junker2001cognitive,vaidya2007role}, assume that the probability $\P(R_j  \mid \QQ, \aaa)$ is determined by the Boolean output $\qq_j \ast \aaa$, i.e., 
\begin{equation}\label{eq:Boolean}
    \P(R_j = 1 \mid \QQ, \aaa) = \qq_j \ast \aaa.
\end{equation}
Among the many possible Boolean operators, two are especially common in practice: the \emph{conjunctive operator} and the \emph{disjunctive operator} \citep{Templin,Junker07,rava2016boolean}. Specifically, the conjunctive operator is defined as
\vspace{-0.3cm}
\begin{equation}\label{eq:ope-conj}
   \qq_j \ast \aaa \;=\; \qq_j \preceq \aaa,
\end{equation}
where $\preceq$ is the \emph{product order}  on \(\{0,1\}^K\), which defines \(\vv \preceq \uu\) if \(v_{k} \leq u_{k}\) for all \(k \in [K]\), for any $\vv, \uu \in \{0,1\}^K$. The disjunctive operator is defined as
\begin{equation}\label{eq:ope-disj2}
 \qq_j \ast \aaa = \qq_j \odot \aaa =  \bigvee_{k=1}^K q_{jk} \wedge \alpha_k,  
\end{equation}
where $\vee$ denotes the logical OR (maximum), and $\wedge$ denotes the logical AND (minimum).
{
\color{black}
When applied to binary vectors of the same dimension, these operations are
defined componentwise. 
}
% It also appears equivalently as \(\qq_j \ast \aaa = \1\{\exists\,k: q_{jk} = \alpha_k = 1\}\)).
We refer to the Boolean factor models employing these operators as \emph{conjunctive models} 
and \emph{disjunctive models} respectively.
In the following, we present two applications of the Boolean factor models to further illustrate their practical usefulness.

% As mentioned, this modeling framework naturally extends to multi-layer settings by introducing additional layers of latent variables and specifying inter-layer graphical structures \citep[cf.][]{gu2023bayesian}. In such multi-layer Bayesian networks, each observed or latent variable depends directly on a subset of latent variables in the layer directly above it, and the overall dependence structure is encoded by these inter-layer graphical structures. Our identifiability results, which focus on the matrix representing the graphical structure linking one latent layer to the observed layer, naturally extend to these more complex models by treating each pair of consecutive layers in a similar manner.

%For instance, the widely used conjunctive operator \eqref{eq:ope-conj} and disjunctive operator \eqref{eq:ope-disj} we presented in Section~\ref{sec:BOM} both satisfy this assumption. 

% \begin{remark}
%     We point out that relaxation of the assumption $p_{\aaa}>0$ for all $\aaa \in \{0,1\}^K$ is possible with careful examination; see Section~\ref{sec: discussion} for further discussion. \todo{equivalence class positive}
% \end{remark}

\begin{example}[Boolean Matrix Decomposition]\label{ex:app1}
{\it    In data analysis and machine learning, Boolean factor models play a significant role in \emph{Boolean matrix decomposition} \citep{geerts2004tiling, zhang2007role, miettinen2008discrete, Vaidya2012}.
Consider $N$ individuals, whose response vectors $\RR_i$ for $i \in [N]$ are generated from model \eqref{eq:Boolean}, and we use $\Re = (\RR_1, \ldots, \RR_N) \in \{0,1\}^{J \times N}$ to denote the observed response matrix and $\AA = (\aaa_1, \ldots, \aaa_N) \in \{0,1\}^{K \times N}$ to denote the unobserved latent matrix in which  $\aaa_i$ is the latent vector/latent configuration of individual $i$.
Boolean matrix decomposition corresponds to the decomposition of 
\(\Re\) into some product of $\QQ$ and $\AA$: $\Re = \QQ \ast \AA$,
where \(\Re_{j,i} = \bigl(\QQ\ast \AA\bigr)_{j,i} := \qq_j \ast \aaa_i\).
In the literature, the disjunctive operator \eqref{eq:ope-disj2} is commonly used for its high interpretability and computational efficiency \citep{rava2016boolean}.
Such decompositions appear in a range of applications:
\begin{itemize}
    \item {\emph{Role Engineering.} $\RR$ represents a user's permissions, $\aaa$ represents  the roles assigned to the user \citep{vaidya2007role, zhang2007role}, and the $\QQ$-matrix specifies which permissions are included in each role. Under the operator \eqref{eq:ope-disj2}, a user has permission $j$ if the user is assigned \emph{any} role that has permission $j$ indicated by $\qq_j$.} 
    \item {\emph{Text Mining.} $\RR$ represents the presence or absence of specific words in a document, while $\aaa$ represents the topics the document is categorized into \citep{geerts2004tiling}, and the $\QQ$-matrix indicates which words are characteristic of each topic. Under \eqref{eq:ope-disj2}, a document contains word $j$ if the document belongs to \emph{any} topics that are relevant to word $j$ specified by $\qq_j$.}
\end{itemize}
Other applications include multi-label classification \citep{wicker2012multi, siblini2019review},  identifying functional interactions in brain
networks~\citep{haddad2018identifying}, and discovery of patterns in network
data~\citep{kocayusufoglu2018summarizing}, etc.
}
\end{example}

% ; in multi-label classification problems \citep{wicker2012multi, rava2016boolean}, the observed data $\RR$ represents the set of labels assigned to each instance, and the latent factors $\aaa$ capture underlying label patterns or dependencies. The $\QQ$-matrix specifies how these latent patterns contribute to the observed labels.

\begin{example}[Personalized Diagnosis and Clustering]\label{ex:app2}
{\it    In social and biological sciences, Boolean factor models could be used to make classification-based decisions on individuals' latent attributes for diagnostic or clustering purposes, based on their responses to a set of items, and are known as cognitive diagnosis models. Specific applications include:
\begin{itemize}
    \item {\emph{Cognitive Diagnosis in Educational Assessments.}
    \(\RR\) represents a student's correct/incorrect answers to test items, while 
    \(\aaa\) denotes the presence/absence of the underlying skills \citep{davier2008general, henson2009defining,rupp2010diagnostic}. The \(\QQ\)-matrix 
    indicates which skills each item measures \citep{Tatsuoka1990}. Often, a 
    conjunctive operator \eqref{eq:ope-conj} is used, implying that a student must possess \emph{all} required skills in $\qq_j$ to answer item~\(j\) correctly.}

    \item {\emph{Psychiatric Evaluation of Mental Disorders.}
    \(\RR\) represents a person's manifested symptoms, and \(\aaa\) denotes the presence/absence 
    of various mental disorders \citep{Templin,dela2018}. The \(\QQ\)-matrix indicates which disorders are 
    associated with which symptoms. 
    The often used disjunctive operator \eqref{eq:ope-disj2} means that having 
    \emph{any} relevant disorder in $\qq_j$ leads to symptom~\(j\).}
\end{itemize}
Other applications, such as medical measurement studies \citep{wu2017nested,o2019causes}, are also receiving increasing attention. 
}
\end{example}

Beyond the Boolean factor models \eqref{eq:Boolean}, researchers also introduce additional parameters to further accommodate the uncertainty and noise in the observed data. 
{\color{black}
We refer to such models as \emph{probabilistic Boolean factor models} and consider three popular models below.
}
%uncertainty in various applications.
%there exist many applications in practice, requiring to account for uncertainty, utilizes models with additional parameters to account for such uncertainty.
%We consider \emph{probabilistic Boolean factor models} which introduce additional parameters into \(\P(R_j \mid \QQ, \aaa)\) to quantify the noise levels that deviate from the scenario in \eqref{eq:Boolean}.  

%Specifically, they allow for the possibility that a subject in the capable class may incorrectly answer an item, and a subject in the incapable class may correctly answer an item, due to factors such as guessing or slipping. 

\textbf{One-parameter Boolean factor models.}
By introducing an additional parameter \(g\), the models assume that for each $j$,
\vspace{-0.2cm}
\begin{equation}\label{eq:one-param-model}
    \P(R_j = 1 \mid \QQ, \aaa) = {\renewcommand{\arraystretch}{0.5}
    \begin{cases}
1-g,\;&\text{if}\;\qq_j \ast \aaa = 1,\\
g,\;&\text{if}\;\qq_j \ast \aaa = 0.
    \end{cases}
    }
    \end{equation}
In Boolean matrix decomposition, \eqref{eq:one-param-model} is viewed as a probabilistic generative modeling approach to Boolean matrix decomposition \citep{rava2016boolean}; with the disjunctive operator it appears as the \emph{OrMachine} \citep{ormachine2017}, which has been applied to single-cell gene-expression analysis and to controlling false positives in collaborative filtering \citep{ormachine2017}.
In cognitive diagnosis, the same specification underlies the so-called \emph{non-parametric classification (NPC) method} \citep{Chiu2013}.  
% Given a $\QQ$-matrix and a response vector $\RR$ from an individual, NPC estimates the individual's latent attributes $ \hat{\aaa} = \argmin_{\aaa \in \{0,1\}^K} \sum_{j=1}^J |R_j - \qq_j \ast \aaa|,$ i.e., the latent vector whose ideal responses is closest (in Hamming distance) to the observed responses, which is exactly the maximum-likelihood estimator for model~\eqref{eq:one-param-model} \citep{restrictedDINA}. 
Due to the method's advantages such as parsimony, computational efficiency and
robustness, it has attracted growing attention in recent work \citep{Chiu2013,chiu2018cognitive,ma2023bridging}.

%referring to the one-parameter conjunctive model as the \emph{restricted DINA model}. 

%When the conjunctive operator in Equation~\eqref{eq:ope-conj} is used in the one-parameter Boolean factor model, the resulting model yields the same estimations of individuals' latent vectors as the nonparametric method \citep{restrictedDINA}. Therefore, the one-parameter Boolean factor models, which admit a generic Boolean operator $\ast$, can be considered a generalization of the classic nonparametric method, allowing for more flexibility in the relationship between observed variables and latent variables.
\noindent
\textbf{Two-parameter Boolean factor models.}
By introducing an additional parameter \(s\), the models assume that
for any $j$,
\vspace{-0.4cm}
\begin{equation}\label{eq:two-param-model}
    \P(R_j = 1 \mid \QQ, \aaa) = 
{\renewcommand{\arraystretch}{0.5}
    \begin{cases}
1 - s, & \text{if } \qq_j \ast \aaa = 1, \\
g,     & \text{if } \qq_j \ast \aaa = 0.
    \end{cases}}
\end{equation}
In educational assessments, the parameter $s$ is known as the \emph{slipping parameter}, representing the probability that an individual who could respond correctly in an ideal case provides an incorrect response due to carelessness.  
Conversely, the parameter $g$, referred to as the \emph{guessing parameter}, captures the probability that an individual who lacks the necessary skills  guesses the correct answer.
In cognitive diagnosis models, these models correspond to the \emph{weighted nonparametric classification method} \citep{Chiu2013}, a variation of the aforementioned NPC that assigns different weights to slips and guesses. 
% Specifically, this method defines the objective as the weighted Hamming distance, assigning a weight \( w_s \) to slips and a weight \( w_g \) to guesses. Consequently, the estimator of the latent attribute \( \aaa \) for an individual with observed responses \( \RR \) is given by: $\hat{\aaa} = \argmin_{\aaa \in \{0,1\}^K} \sum_{j=1}^J w_g \1\,(R_j = 1)\,|R_j - \qq_j \ast \aaa| + \sum_{j=1}^J w_s \1\,(R_j = 0)\,|R_j - \qq_j \ast \aaa|,$
% which is also the maximum likelihood estimator of the model in Equation~\eqref{eq:two-param-model}. 

{\color{black}
\noindent
\textbf{Multi-parameter Boolean factor models.}
Allowing the slipping and guessing parameters to vary across observed variables, the model assumes that, for each \(j\in[J]\),
\begin{equation}\label{eq:multi-model}
    \P(R_j = 1 \mid \QQ, \aaa) = {\renewcommand{\arraystretch}{0.5}
    \begin{cases}
1-s_j,\;&\text{if}\;\qq_j \ast \aaa = 1,\\
g_j,\;&\text{if}\;\qq_j \ast \aaa = 0.
    \end{cases}
    }
    \end{equation}
    When \(\ast\) is the conjunctive or disjunctive operator, \eqref{eq:multi-model} becomes the well-known DINA model \citep{junker2001cognitive} or DINO model \citep{Templin}, respectively. 
}
\begin{remark}\label{rm:generalmodels}
The probabilistic Boolean factor models above also serve as foundations for more complex latent variable models. In cognitive diagnosis, based on the DINA model, more general models can be obtained, such as the generalized DINA model \citep[GDINA;][]{de2011generalized}, the general diagnostic model \citep[GDM;][]{davier2008general}, the log-linear cognitive diagnosis model \citep[LCDM;][]{henson2009defining}.
\end{remark}

\subsection{Identifiability Problem}\label{sec:identifiability}
% In many applications, the \(\QQ\)-matrix must be estimated rather than specified in advance.
As with most latent variable models, achieving reliable and valid estimation
and inference requires addressing the identifiability of the model parameters.
In particular, the graphical structure/\(\QQ\)-matrix of the models we consider plays a pivotal role: it encodes the relationships between
the latent and observed variables and is crucial for model interpretation, as
illustrated in Examples~\ref{ex:app1} and~\ref{ex:app2}. 
Consequently, its identifiability is the main focus
of this work.

The primary challenge in identifying the $\QQ$-matrix arises from its discrete nature and the non-linear relationships it encodes between latent and observed variables. While previous works \citep{JLGXZY2011,chen2015statistical,xu2018,gu21,MIRON2021107809,desantis2022factorization} have studied conditions for
identifying the \(\QQ\)-matrix, the conditions proposed in these works are often too stringent for practical applications. 
In particular, they typically require the existence of at least one \emph{pure node} for \emph{each} latent variable. Formally, an observed variable $R_j$  is called a pure node if it has exactly one parent in the graphical structure, equivalently, $\norm{\qq_j} = 1$, where $\norm{\Cdot}$ denotes the $\ell_0$ norm. 
For example, consider the graphical structure and its corresponding $\QQ$-matrix in Figure~\ref{model:pure}. Here, observed variables $R_j$ for $j \in \{1,2,3,4\}$ are pure nodes, and each latent variable has one associated pure node. 
When pure nodes exist for every latent variable, it implies that, after a suitable column permutation, the corresponding $\QQ$-matrix contains an identity submatrix. Such a $\QQ$-matrix is also said to be \emph{complete} \citep{chiu2009,xu2017}, a property that is commonly assumed in the literature to ensure identifiability. However, such a requirement is often too stringent and hard to meet in practice.
%, especially when latent variables are correlated, especially in cognitive diagnosis applications. 
For instance, in educational and psychological assessments, where diagnostic test items are designed to measure latent cognitive skills, a practical $\QQ$-matrix rarely satisfies such a requirement \citep[e.g.,][]{jaeger2006distinguishing, de2011generalized, lee2011cognitive}.
\vspace{-0.5cm}
\begin{figure}[H]
\centering
\begin{minipage}{0.57\textwidth}
\centering
\begin{tikzpicture}[scale=0.9]

  \tikzstyle{obs}=[circle, draw, minimum size=20pt, inner sep=0pt, fill=white]
  \tikzstyle{lat}=[circle, draw, minimum size=20pt, inner sep=0pt, fill=white]
  \tikzstyle{qedge}=[-latex, thick]

  \node (R1)[obs] at (1.2, 0) {$R_1$};
  \node (R2)[obs] at (2.4, 0) {$R_2$};
  \node (R3)[obs] at (3.6, 0) {$R_3$};
  \node (R4)[obs] at (4.8, 0) {$R_4$};
  \node (R5)[obs] at (6.0, 0) {$R_5$};
  \node (R6)[obs] at (7.2, 0) {$R_6$};

  \node (a1)[lat] at (2.4, 1.5) {$\alpha_1$};
  \node (a2)[lat] at (3.6, 1.5) {$\alpha_2$};
  \node (a3)[lat] at (4.8, 1.5) {$\alpha_3$};
  \node (a4)[lat] at (6.0, 1.5) {$\alpha_4$};

  \draw[qedge] (a1) -- (R1);
  \draw[qedge] (a2) -- (R2);
  \draw[qedge] (a3) -- (R3);
  \draw[qedge] (a4) -- (R4);

  \draw[qedge] (a1) -- (R5);
  \draw[qedge] (a2) -- (R5);
  \draw[qedge] (a3) -- (R6);
  \draw[qedge] (a4) -- (R6);

  \draw (a1) -- (a2);
  \draw (a2) -- (a3);
  \draw (a3) -- (a4);
  \draw[bend left=60] (a1) to (a3);
  \draw[bend left=60] (a2) to (a4);
  \draw[bend left=90] (a1) to (a4);

  \node[draw, dashed, thick,
        fit=(a1)(a4)(R1)(R6), inner sep=6pt] {};

\end{tikzpicture}
\caption*{(a) Graphical structure}
\end{minipage}\hfill
\begin{minipage}{0.36\textwidth}
\centering
\begin{equation*}
\QQ=
\begin{pmatrix}
1 & 0 & 0 & 0 \\
0 & 1 & 0 & 0 \\
0 & 0 & 1 & 0 \\
0 & 0 & 0 & 1 \\
1 & 1 & 0 & 0 \\
0 & 0 & 1 & 1
\end{pmatrix}
\end{equation*}
\caption*{(b) \(\QQ\)-matrix}
\end{minipage}
\caption{A graphical structure and its $\QQ$-matrix.
(a) Each \(R_j\), for \(j=1,\ldots,4\), is a pure node with exactly one
parent. (b) The \(\QQ\)-matrix contains an identity submatrix.}
\label{model:pure}
\end{figure}
\vspace{-0.5cm}

Our goal is to address this limitation by establishing sufficient and
necessary conditions for identifying the \(\QQ\)-matrix without requiring
pure nodes. For instance, the graphical structures in
Figure~\ref{fig:models} share the same \(\QQ\)-matrix,
which contains no pure nodes yet will be shown to be identifiable under
commonly used Boolean factor models in Example~\ref{ex2} of
Section~\ref{sec:graphcond}.

Before formally defining the identifiability of the \(\QQ\)-matrix, we note that many commonly used Boolean operators, including the conjunctive and disjunctive operators, are invariant to permutations of the latent coordinates. Throughout this work, we restrict attention to operators satisfying this property. Specifically, for any permutation \(\pi\) of \([K]\), $\qq \ast \aaa = (\pi \qq)\ast(\pi \aaa),$ where \(\pi\qq\) and \(\pi\aaa\) denote the vectors obtained by applying the same coordinate permutation \(\pi\). Thus, simultaneously permuting the columns of \(\QQ\) and relabeling the corresponding latent variables does not change the model for the observed variables. Since the latent variables have no intrinsic ordering, we can therefore only expect to identify the \(\QQ\)-matrix up to column permutation. Specifically, we write \(\bar{\QQ}\sim\QQ\) if \(\bar{\QQ}\) can be obtained by permuting the columns of \(\QQ\); otherwise, we write $\bar{\QQ} \nsim \QQ$.

{\color{black}
In addition to identifying $\QQ$, we also study identifiability
of the remaining model parameters. We note that, without pure nodes, the full latent proportion vector
$\pp=(p_{\vv}:\vv\in\{0,1\}^K)$, where
$p_{\vv}=\P(\aaa=\vv)$, is generally not identifiable.
For instance, under the conjunctive and disjunctive operators,
completeness is necessary for identifying $\pp$
\citep{xu2017,gu19}.
When completeness is not imposed, distinct latent vectors may
induce the same conditional distribution of $\RR$ and therefore
cannot be distinguished from the observed data.
For a fixed $\QQ$, let $[\aaa]$ denote the equivalence class of
latent vectors that induce the same conditional distribution as
$\aaa$, and let
$p_{[\aaa]}=\sum_{\vv\in[\aaa]}p_{\vv}$ denote its grouped
proportion. We write $\mathcal E(\QQ)$ for the collection of
these equivalence classes and
$p_{[\Cdot]}=(p_{[\aaa]})_{[\aaa]\in\mathcal E(\QQ)}$
for the vector of grouped proportions; see
Section~\ref{sec:G} for the formal definitions.
The observed distribution depends on the proportions within
each class only through their sum. Thus, proportions of
indistinguishable latent vectors cannot be identified separately,
and $p_{[\Cdot]}$ represents the finest possible resolution
for identifying $\pp$ \citep{gu2020partial}. 
% Importantly, this partial identification of \(\pp\) does not prevent the identification of the \(\QQ\)-matrix, which is our primary identifiability target. 
We therefore formulate joint identifiability in terms of these
grouped proportions. Let $\TTT$ denote the model parameters
other than $\QQ$. Specifically, for Boolean
factor models in \eqref{eq:Boolean}, \(\TTT=p_{[\Cdot]}\); for probabilistic Boolean factor models in \eqref{eq:multi-model},
\(\TTT=(\mathbf s,\mathbf g,p_{[\Cdot]})\), where
\(\mathbf s=(s_1,\ldots,s_J)^\top\) and
\(\mathbf g=(g_1,\ldots,g_J)^\top\). 
Correspondingly, we write \((\bar{\QQ},\bar{\TTT})\sim(\QQ,\TTT)\) if
\(\bar{\QQ}\sim\QQ\) and \(\bar{\TTT}=\TTT\) after a common
relabeling of the latent variables that aligns
\(\bar{\QQ}\) with \(\QQ\), where grouped proportions are compared
using the corresponding equivalence classes. Otherwise, we write
\((\bar{\QQ},\bar{\TTT})\nsim(\QQ,\TTT)\).

For the $\QQ$-matrix under consideration, we assume that
$p_{[\aaa]}>0$ for every $[\aaa]\in\mathcal E(\QQ)$.
This positivity condition is imposed at the level of equivalence
classes and does not require $p_{\aaa}>0$ for every individual
latent vector. This condition excludes degenerate cases in which
the support of $\pp$ precludes identification of $\QQ$. For
instance, under the conjunctive model, if all subjects have the
all-one latent vector ($p_{\one}=1$), then $\QQ$ is not
identifiable, since the resulting response distribution cannot
distinguish different $\QQ$-matrices. Note that the positivity
condition is imposed only on the data-generating representation
$(\QQ,\TTT)$, while a competing representation
$(\bar\QQ,\bar\TTT)$ may assign zero probability to some of its
equivalence classes. For probabilistic Boolean factor models, we
also assume $0\leq g_j<1-s_j\leq1$ for $j\in[J]$, as in the
literature \citep{xu2016,xu2017}.

% With these preliminaries in place, we now give the formal definition of
% identifiability for the \(\QQ\)-matrix and the model parameter.}

\begin{definition}\label{def:identifiability}
For given \(J\) and \(K\), let
\(\QQ\in\{0,1\}^{J\times K}\) with \(p_{[\aaa]}>0\) for every
\([\aaa]\in\mathcal E(\QQ)\).
We say that \(\QQ\) is identifiable if for every
\(\TTT\)  there does not exist \((\bar{\QQ},\bar{\TTT})\), with
\(\bar{\QQ}\nsim\QQ\), such that
\begin{equation}
\label{eq:equal-response-distributions}
\P(\RR=\rr\mid\QQ,\TTT)
=
\P(\RR=\rr\mid\bar{\QQ},\bar{\TTT}),
\qquad \rr\in\{0,1\}^J.
\end{equation}
We say that \((\QQ,\TTT)\)
is jointly identifiable if there does not exist \((\bar{\QQ},\bar{\TTT})\), with
\((\bar{\QQ},\bar{\TTT})\nsim(\QQ,\TTT)\), such that
\eqref{eq:equal-response-distributions} holds. 
% In particular, joint identifiability of \((\QQ,\TTT)\) for every \(\TTT\) implies identifiability of \(\QQ\).
\end{definition}

% \begin{definition}
% \label{def:uniform-q-identifiability}
% For given \(J\) and \(K\), we say that a given $\QQ\in\{0,1\}^{J\times K}$ is identifiable if for every $\TTT$ with  \(p_{[\aaa]}>0\) for every 
% \([\aaa] \in \mathcal{E}(\QQ)\) there does not exist
% $(\bar\QQ,\bar\TTT)$ where $\bQQ \nsim \QQ$, such that
% \[
% \P(\RR=\rr\mid\QQ,\TTT)
% =
% \P(\RR=\rr\mid\bar\QQ,\bar\TTT),
% \qquad
% \rr\in\{0,1\}^J.
% \]
% \end{definition}

% \begin{definition}\label{def:joint-identifiability}
% For given \(J\) and \(K\), and a given $\QQ \in \{0,1\}^{J \times K}$, if for every
% $(\bar\QQ,\bar\TTT)$,
% \[
% \P(\RR=\rr\mid\QQ,\TTT)
% =
% \P(\RR=\rr\mid\bar\QQ,\bar\TTT),
% \qquad
% \rr\in\{0,1\}^J,
% \]
% implies that $\QQ \sim \bQQ$ for any $\TTT$  with \(p_{[\aaa]}>0\) for every 
% \([\aaa] \in \mathcal{E}(\QQ)\), we say that $\QQ$ is identifiable. If it implies that $\QQ \sim \bQQ$ and $\TTT = \bar \TTT$, then we say that $(\QQ, \TTT)$ are jointly identifiable. 
% \end{definition}
}
Finally, to ensure that our study of the identifiability of the
\(\QQ\)-matrix is practically meaningful, we restrict attention to operators
\(\ast\) that are injective in both \(\qq\) and \(\aaa\). Formally, \(\ast\)
is injective in \(\qq\) if, for any
\(\qq_1,\qq_2\in\{0,1\}^K\), the condition
\(\qq_1\ast\aaa=\qq_2\ast\aaa\) for all
\(\aaa\in\{0,1\}^K\) implies that \(\qq_1=\qq_2\). Similarly, \(\ast\) is
injective in \(\aaa\) if, for any
\(\aaa_1,\aaa_2\in\{0,1\}^K\), the condition
\(\qq\ast\aaa_1=\qq\ast\aaa_2\) for all
\(\qq\in\{0,1\}^K\) implies that \(\aaa_1=\aaa_2\). 
Injectivity in both $\qq$ and $\aaa$ is crucial for identifiability.
For instance, if $\ast$ is \emph{not} injective in $\qq$, i.e., there exist distinct $\qq_1 \neq \qq_2$ such that $\qq_1 \ast \aaa = \qq_2 \ast \aaa$ for all $\aaa \in \{0,1\}^K$, then the observed responses generated by $\qq_1$ and $\qq_2$ are indistinguishable regardless of the latent vectors $\aaa$. This means that replacing $\qq_1$ with $\qq_2$ in a $\QQ$-matrix does not change the distribution of $\RR$, making it impossible to identify $\qq_1$ and $\qq_2$ based on the data.
We therefore focus on operators that are injective in both
arguments. Many commonly used operators satisfy these injectivity properties;
see Fact~\ref{fact:injective} in the Supplementary Material. Furthermore, we
assume that \(\qq\ast\zero=0\) for every \(\qq\neq\zero\). 
% which aligns with the practical interpretation in cognitive diagnosis that a zero latent skill vector $\aaa = \zero$ cannot produce a positive response to any item with nonzero $\qq$.
\vspace{-0.3cm}
\section{Graphical Representation and  Identifiability Results}\label{sec:graph}
This section develops graphical and algebraic characterizations of the
identifiability of the $\QQ$-matrix under Boolean factor models. Our key observation is that an operator $\ast$ induces a natural partial order on the latent space, which can be visualized through a Hasse diagram. For a given $\QQ$, embedding its $\qq$-vectors into this diagram yields a graphical representation of the distribution of $\RR$.
Based on this, we characterize identifiability through the uniqueness of this graphical representation and demonstrate that pure nodes are not necessary. 
To verify this condition, we leverage the Boolean structure of the problem to express it through logical operations and subsequently develop an efficient SAT-based algorithm.
\subsection{Graphical Representation}\label{sec:graph-representative}

\subsubsection{\texorpdfstring{$``\boldsymbol{\ast}"$}{``*''}-Induced Partial Order and Hasse Diagram}
%The graphical representation is called the \emph{Hasse-based distribution graph} of $\QQ$, denoted $G(\QQ)$. 
Our results leverage a novel graphical tool. We call this graph the \emph{Hasse-based distribution graph of $\QQ$}, and denote it by $G(\QQ)$, to be introduced in the following.
  The construction of $G(\QQ)$ starts from our key observation in the Boolean factor models that, under any Boolean operator \(\ast\), a natural hierarchy arises among the latent vectors in \(\{0,1\}^K\), based on their interactions with all \(\qq \in \{0,1\}^K\). 
Specifically, we define a relation \(\preceq^\ast\) on \(\{0,1\}^K\) as follows: for any \(\aaa_1, \aaa_2 \in \{0,1\}^K\),
\begin{equation}\label{eq:partial_order_def}
\aaa_1 \preceq^\ast \aaa_2 \quad \text{if} \quad \qq \ast \aaa_1 \leq \qq \ast \aaa_2 \quad \text{for all} \; \qq \in \{0,1\}^K.
\end{equation}
The next result shows that this relation indeed forms a \emph{non-strict partial order}, as it satisfies the three defining properties of reflexivity, antisymmetry, and transitivity \citep{davey2002introduction}. 
\begin{proposition}\label{prop:partial}
    The relation \(\preceq^\ast\) defined in ~\eqref{eq:partial_order_def} is a non-strict partial order on \(\{0,1\}^K\). 
\end{proposition}
\noindent
Following the literature \citep{davey2002introduction}, we define
the strict partial order $\prec^\ast$ corresponding to
$\preceq^\ast$ by $\aaa_1\prec^\ast\aaa_2$ if and only if
$\aaa_1\preceq^\ast\aaa_2$ and $\aaa_1\neq\aaa_2$.
The reverse relations $\succeq^\ast$ and $\succ^\ast$
are defined analogously.
Both the conjunctive operator~\eqref{eq:ope-conj} and the
disjunctive operator~\eqref{eq:ope-disj2} induce the product
order~$\preceq$; see Supplementary
Section~\ref{app:secgraph} for a proof. 

%\medskip
%\noindent
%\textbf{Hasse diagram for the induced partial order.}
Given the induced partial order \(\preceq^\ast\) on \(\{0,1\}^K\), we proceed to represent it using a Hasse diagram \citep{davey2002introduction}, which we denote by \(H(\{0,1\}^K, \, \preceq^\ast)\). Let \(v(\aaa)\) represent the node corresponding to the latent vector \(\aaa \in \{0,1\}^K\). The diagram is constructed as follows:
draw a directed edge from node \(v(\aaa_1)\) to node \(v(\aaa_2)\) if and only if \(\aaa_1 \prec^\ast \aaa_2\), and there does not exist any other \(\aaa_3\), such that \(\aaa_1 \prec^\ast \aaa_3 \prec^\ast \aaa_2\). 
The structure of the Hasse diagram implies that a directed path from \(v(\aaa_1)\) to \(v(\aaa_2)\) exists if and only if \(\aaa_1 \prec^\ast \aaa_2\) holds.
Therefore, the Hasse diagram fully characterizes the partial orders between all the latent vectors.

Different Boolean operators may induce different partial orders, resulting in varying Hasse diagrams on the latent vector space $\{0,1\}^K$. Among these, the product order, which is the induced partial order of both of the conjunctive operator and the disjunctive operator, is commonly considered in practice. To illustrate, 
%consider the following example with \(K = 3\).
Figure~\ref{fig:exhasse} gives an example of the Hasse diagram for the product order \(\preceq\) on \(\{0,1\}^3\) with \(K = 3\).
%Therefore, unless otherwise specified, the partial order referred to throughout the rest of this paper will be the product order $\preceq$. 
 
% \begin{example}
% The Hasse diagram for the product order \(\preceq\) on \(\{0,1\}^3\) is shown in Figure~\ref{fig:exhasse}. Edges from \(v(110), v(101), v(011)\) to \(v(111)\) indicate that latent vector \((111)\) covers each of these corresponding latent vectors. Although there are no direct edges from \(v(100), v(010), v(001)\) to \(v(111)\), the existence of directed paths between them indicates that \((100) \prec (111)\), \((010) \prec (111)\), and \((001) \prec (111)\).
% \end{example}

\begin{figure}
    \centering
    \begin{tikzpicture}[baseline=(current bounding box.center)]
    % Adjusted node positions for more compact layout
    \node[circle, draw, inner sep=2pt, font=\normalfont, fill=red!20] (000) at (0, 0) {000};
    \node[circle, draw, inner sep=2pt, font=\normalfont, fill=red!20] (100) at (-1.5, 1.5) {100};
    \node[circle, draw, inner sep=2pt, font=\normalfont, fill=red!20] (010) at (0, 1.5) {010};
    \node[circle, draw, inner sep=2pt, font=\normalfont, fill=red!20] (001) at (1.5, 1.5) {001};
    \node[circle, draw, inner sep=2pt, font=\normalfont, fill=red!20] (110) at (-1.5, 3) {110};
    \node[circle, draw, inner sep=2pt, font=\normalfont, fill=red!20] (101) at (0, 3) {101};
    \node[circle, draw, inner sep=2pt, font=\normalfont, fill=red!20] (011) at (1.5, 3) {011};
    \node[circle, draw, inner sep=2pt, font=\normalfont, fill=red!20] (111) at (0, 4.5) {111};
    % Edges with arrows
    \draw[-{Latex[length=2.5mm]}] (000) -- (100);
    \draw[-{Latex[length=2.5mm]}] (000) -- (010);
    \draw[-{Latex[length=2.5mm]}] (000) -- (001);
    \draw[-{Latex[length=2.5mm]}] (100) -- (110);
    \draw[-{Latex[length=2.5mm]}] (100) -- (101);
    \draw[-{Latex[length=2.5mm]}] (010) -- (110);
    \draw[-{Latex[length=2.5mm]}] (010) -- (011);
    \draw[-{Latex[length=2.5mm]}] (001) -- (101);
    \draw[-{Latex[length=2.5mm]}] (001) -- (011);
    \draw[-{Latex[length=2.5mm]}] (110) -- (111);
    \draw[-{Latex[length=2.5mm]}] (101) -- (111);
    \draw[-{Latex[length=2.5mm]}] (011) -- (111);
\end{tikzpicture}\caption{Hasse diagram $H(\{0,1\}^3, \, \preceq)$ for the product order $\preceq$ on $\{0,1\}^3$}
     \label{fig:exhasse}
\end{figure}
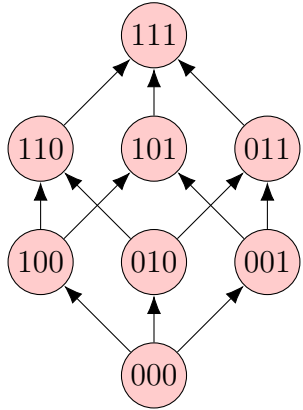

%The Hasse diagram $H(\{0,1\}^K, \, \preceq^\ast)$, which visually characterizes the induced partial orders among all latent vectors, serves as the foundation for our core tool to study the identifiability of the $\QQ$-matrix—the graphical representation of the distribution of $\RR$ to be introduced in the following subsection.
 
\vspace{-0.3cm}
{
\color{black}
\subsubsection{Extended Hasse Diagram of \texorpdfstring{$\QQ$}{Q}}\label{subsec:extended_hasse}
Based on the induced Hasse diagram, to graphically represent the distribution of $\RR$ under $\QQ$, 
% which is determined by the interactions between the \(\qq\)-vectors and the latent $\aaa$-vectors, we need to incorporate nodes representing the \(\qq\)-vectors into the Hasse diagram.
we incorporate nodes representing the \(\qq\)-vectors
into the Hasse diagram.

\begin{definition}[Extended Hasse diagram of \(\QQ\)]\label{def:extended_hasse}
{\it
The \emph{extended Hasse diagram of \(\QQ\)}, denoted by \(\mch(\QQ)\), is
obtained from the Hasse diagram \(H(\{0,1\}^K,\preceq^\ast)\) by adding one
node \(v_j^*(\qq_j)\) for each \(\qq_j\)-vector, and drawing a directed
edge from \(v_j^*(\qq_j)\) to \(v(\aaa)\) whenever
\begin{equation}\label{eq:minimal_capable}
        \aaa \in \mca_j^+ 
        \quad\text{and}\quad
        \nexists\ \aaa' \in \mca_j^+ \text{ such that } \aaa' \prec^\ast \aaa,
\end{equation}
where $\mca_j^+ = \{\aaa \in \{0,1\}^K : \qq_j \ast \aaa = 1\}.$
}
\end{definition}

That is, \(\aaa\) is a minimal element of \(\mca_j^+\) under \(\preceq^\ast\). 
% That is, $\aaa$ is the ``minimal'' element in $\mca_j^+$, meaning that it is not greater than any other latent vectors  in $\mca_j^+$ under the partial order relation $\prec^\ast$. 
In the applications of educational assessments,
\(\mca_j^+\) is often called the \emph{capable class} for item \(j\)
\citep{xu2017,slam}. Adopting this terminology, we call the latent vectors
satisfying~\eqref{eq:minimal_capable} the \emph{minimal capable vectors for \(\qq_j\)}. 
We use \(\mcm_j^+\) to denote the set of these vectors and refer
to \(\mcm_j^+\) as the \emph{minimal capable class}.
}
% In the applications of educational assessments, $\mca_j^+$ represents the set of the latent vectors $\aaa$'s that are capable of correctly answering item $j$, and is often called as the \emph{capable class} for item $j$ \citep{xu2017,slam}. Latent vectors  $\aaa$'s satisfying Equation~\eqref{eq:minimal_capable} thus represent the minimal $\aaa$-vectors that are capable of correctly answering item $j$. Adopting this terminology, we call the latent vectors $\aaa$'s satisfying~\eqref{eq:minimal_capable} the \emph{minimal capable vectors for $\qq_j$}. Note that, for each $j$,  {\color{black} there} might exist multiple minimal capable vectors for $\qq_j$ (see Figure~\ref{fig:hassewithitem}(b)), we use $\mcm_j^+$ to denote the set of these vectors, and refer to $\mcm_j^+$ as the \emph{minimal capable class}.

\begin{remark}
    For the conjunctive model \eqref{eq:ope-conj}, where $\mca_j^+ = \{ \aaa \in \{0,1\}^K \;|\; \qq_j \preceq \aaa \}$, we have $\aaa = \qq_j$ being the unique minimal element in $\mca_j^+$; therefore, $\mcm_j^+ = \left\{ \qq_j \right\}$. For the disjunctive model \eqref{eq:ope-disj2}, where $\mca_j^+ = \{ \aaa \in \{0,1\}^K \;|\; \exists\, k: q_{jk} = \alpha_k = 1 \}$,   \(\mcm_j^+\) consists of those \(\aaa\) with exactly one coordinate \(k\) set to 1 where \(q_{jk}=1\).
That is, \(\mcm_j^+\) takes the form: $\mcm_j^+ = \set{\ee_k \mid k \in \mcs(\qq_j)}$,
where $\mcs(\qq_j) = \{k \mid q_{jk} = 1\}$ denotes the support of $\qq_j$, i.e., the set of indices where $\qq_j$ has non-zero entries.

\end{remark}

Intuitively, given the minimality of $\mcm_j^+$,  we can determine whether a latent vector $\aaa$ satisfies $\qq_j \ast \aaa = 1$ by checking if there exists at least one element  $\aaa_j^{\min}\in \mcm_j^+$, such that, \(\aaa \succeq^* \aaa_j^{\min}\). Graphically, this means that either there exists a directed \emph{path} from $v(\aaa_j^{\min})$ to $v(\aaa)$ or $v(\aaa_j^{\min}) = v(\aaa)$.
Since \(v_j^*(\qq_j)\) is connected to all elements in \(\mcm_j^+\), whether
\(\qq_j\ast\aaa=1\) holds can be checked by whether there exists a directed
path from \(v_j^*(\qq_j)\) to \(v(\aaa)\).
In this way, $\qq_j \ast \aaa$ for any $\aaa$ is characterized by $\mch(\QQ)$. 
We state this result in the following.
\begin{proposition}\label{prop:path-response}
    For any Boolean operator $\ast$ and for all $\aaa \in \{0,1\}^K$, $\qq_j \ast \aaa = 1$ if and only if there exists at least one $\aaa_j^{\min} \in \mcm_j^+$ such that $\aaa_j^{\min} \preceq^\ast \aaa$, or graphically, if and only if there exists a directed path from $v_j^*(\qq_j)$ to $v(\aaa)$ in $\mch(\QQ)$.
\end{proposition}

{
\color{black}
Proposition~\ref{prop:path-response} shows that \(\qq_j\ast\aaa\) can be read from \(\mch(\QQ)\), motivating our following definition.
\begin{definition}[Path status vector]\label{def:path_status}
{\it
For each latent vector node \(v(\aaa)\) in \(\mch(\QQ)\), its
\emph{path status vector}, denoted by \(\Psi(v(\aaa))\in\{0,1\}^J\), is the
binary vector whose \(j\)-th entry equals one if and only if there exists a
directed path from \(v_j^*(\qq_j)\) to \(v(\aaa)\) in \(\mch(\QQ)\).
}
\end{definition}
By Proposition~\ref{prop:path-response},
\(\Psi(v(\aaa))=(\qq_1\ast\aaa,\ldots,\qq_J\ast\aaa)^\top\). Thus,
\(\Psi(v(\aaa))\) corresponds exactly to the conditional distribution of \(\RR\) given
\(\aaa\).
}
As a bridge between the graphical representation and the conditional distribution of $\RR$, the path status vectors serve as a useful tool in the graphical approach to study the identifiability of the $\QQ$-matrix. 

To further illustrate the construction of the extended Hasse diagram and the concept of path status vectors, consider the following example.
\begin{example}\label{ex1}
 {\it   Consider conjunctive~\eqref{eq:ope-conj} and disjunctive~\eqref{eq:ope-disj2} models with the same $\QQ$:
    \begin{equation}\label{ex1-Q}
    {\renewcommand{\arraystretch}{0.6}
            \QQ = 
    \begin{pmatrix}
        1 & 1 & 0\\
        1 & 0 & 1\\
        0 & 1 & 1
    \end{pmatrix}.}
    \end{equation}
Under both models, the induced partial order is $\preceq$, and the corresponding Hasse diagram is shown in Figure~\ref{fig:exhasse}. 
The $\qq$-vector nodes $v_1^*(110)$, $v_2^*(101)$, and $v_3^*(011)$ are represented by red rectangles in Figure~\ref{fig:hassewithitem}, and are connected to the latent vector nodes via their minimal capable classes $\mcm_j^+$. 
For visual simplicity, we suppress the item indices. 
% as all item nodes can be distinguished from their positions.
For the conjunctive model each $\mcm_j^+$ contains a single element, whereas for the disjunctive model each contains two.
The same $\QQ$ produces distinct extended Hasse diagrams under the two models, reflecting the different interactions between $\qq$ and $\aaa$.  
In both cases, the existence of a directed path from $v_j^*(\qq_j)$ to $v(\aaa)$ matches the condition $\qq_j \ast \aaa = 1$ in Proposition~\ref{prop:path-response}.
Additionally, note that multiple latent $\aaa$-vector nodes may share identical path status vectors; examples of such nodes are grouped with dashed boxes.
}
\begin{figure}[htbp]
    \centering
    \begin{minipage}{0.45\textwidth}
    \centering
    \begin{tikzpicture}[baseline=(current bounding box.center)]
        % Adjusted node positions for more compact layout
        \node[circle, draw, inner sep=2pt, font=\normalfont, fill=red!20] (000) at (1.5, 0) {000};
        \node[circle, draw, inner sep=2pt, font=\normalfont, fill=red!20] (100) at (0, 1.5) {100};
        \node[circle, draw, inner sep=2pt, font=\normalfont, fill=red!20] (010) at (1.5, 1.5) {010};
        \node[circle, draw, inner sep=2pt, font=\normalfont, fill=red!20] (001) at (3, 1.5) {001};
        \node[circle, draw, inner sep=2pt, font=\normalfont, fill=red!20] (110) at (-1.5, 3) {110};
        \node[circle, draw, inner sep=2pt, font=\normalfont, fill=red!20] (101) at (0, 3) {101};
        \node[circle, draw, inner sep=2pt, font=\normalfont, fill=red!20] (011) at (1.5, 3) {011};
        \node[circle, draw, inner sep=2pt, font=\normalfont, fill=red!20] (111) at (0, 4.5) {111};
    
        % $\qq$-vector nodes
        \node[rectangle, draw, inner sep=5pt, fill=red!50, font=\normalfont] (110d) at (-3.6, 1.5) {110};
        \node[rectangle, draw, inner sep=5pt, fill=red!50, font=\normalfont] (101d) at (-2.4, 1.5) {101};
        \node[rectangle, draw, inner sep=5pt, fill=red!50, font=\normalfont] (011d) at (-1.2, 1.5) {011};
        % Edges with arrows
        \draw[-{Latex[length=2.5mm]}] (000) -- (100);
        \draw[-{Latex[length=2.5mm]}] (000) -- (010);
        \draw[-{Latex[length=2.5mm]}] (000) -- (001);
        \draw[-{Latex[length=2.5mm]}] (100) -- (110);
        \draw[-{Latex[length=2.5mm]}] (100) -- (101);
        \draw[-{Latex[length=2.5mm]}] (010) -- (110);
        \draw[-{Latex[length=2.5mm]}] (010) -- (011);
        \draw[-{Latex[length=2.5mm]}] (001) -- (101);
        \draw[-{Latex[length=2.5mm]}] (001) -- (011);
        \draw[-{Latex[length=2.5mm]}] (110) -- (111);
        \draw[-{Latex[length=2.5mm]}] (101) -- (111);
        \draw[-{Latex[length=2.5mm]}] (011) -- (111);
    
        \draw[-{Latex[length=2.5mm]}] (110d) -- (110);
        \draw[-{Latex[length=2.5mm]}] (101d) -- (101);
        \draw[-{Latex[length=2.5mm]}] (011d) -- (011);
    
        % Dashed rectangle
        \draw[dashed, thick] (-0.5,-0.5) rectangle (3.5,2);
    \end{tikzpicture}\caption*{(a) Conjunctive model}
    \end{minipage}%
    \hfill
    \begin{minipage}{0.45\textwidth}
    \centering
    \begin{tikzpicture}[baseline=(current bounding box.center)]
        % Adjusted node positions for more compact layout
        \node[circle, draw, inner sep=2pt, font=\normalfont, fill=red!20] (000) at (3.5, 0) {000};
        \node[circle, draw, inner sep=2pt, font=\normalfont, fill=red!20] (100) at (-1.5, 1.5) {100};
        \node[circle, draw, inner sep=2pt, font=\normalfont, fill=red!20] (010) at (0, 1.5) {010};
        \node[circle, draw, inner sep=2pt, font=\normalfont, fill=red!20] (001) at (1.5, 1.5) {001};
        \node[circle, draw, inner sep=2pt, font=\normalfont, fill=red!20] (110) at (-1.5, 3) {110};
        \node[circle, draw, inner sep=2pt, font=\normalfont, fill=red!20] (101) at (0, 3) {101};
        \node[circle, draw, inner sep=2pt, font=\normalfont, fill=red!20] (011) at (1.5, 3) {011};
        \node[circle, draw, inner sep=2pt, font=\normalfont, fill=red!20] (111) at (0, 4.5) {111};

        % $\qq$-vector nodes
        \node[rectangle, draw, inner sep=5pt, fill=red!50, font=\normalfont] (110d) at (-1.5, 0) {110};
        \node[rectangle, draw, inner sep=5pt, fill=red!50, font=\normalfont] (101d) at (0, 0) {101};
        \node[rectangle, draw, inner sep=5pt, fill=red!50, font=\normalfont] (011d) at (1.5, 0) {011};

        % Edges with arrows
        \draw[-{Latex[length=2.5mm]}] (000) -- (100);
        \draw[-{Latex[length=2.5mm]}] (000) -- (010);
        \draw[-{Latex[length=2.5mm]}] (000) -- (001);
        \draw[-{Latex[length=2.5mm]}] (100) -- (110);
        \draw[-{Latex[length=2.5mm]}] (100) -- (101);
        \draw[-{Latex[length=2.5mm]}] (010) -- (110);
        \draw[-{Latex[length=2.5mm]}] (010) -- (011);
        \draw[-{Latex[length=2.5mm]}] (001) -- (101);
        \draw[-{Latex[length=2.5mm]}] (001) -- (011);
        \draw[-{Latex[length=2.5mm]}] (110) -- (111);
        \draw[-{Latex[length=2.5mm]}] (101) -- (111);
        \draw[-{Latex[length=2.5mm]}] (011) -- (111);

        %edges between items nodes and latent vector nodes.
        \draw[-{Latex[length=2.5mm]}] (110d) -- (100);
        \draw[-{Latex[length=2.5mm]}] (110d) -- (010);
        \draw[-{Latex[length=2.5mm]}] (101d) -- (100);
        \draw[-{Latex[length=2.5mm]}] (101d) -- (001);
        \draw[-{Latex[length=2.5mm]}] (011d) -- (001);
        \draw[-{Latex[length=2.5mm]}] (011d) -- (010);

        % Dashed rectangle
        \draw[dashed, rounded corners] (-2.5, 2.5) rectangle (2.5, 5);
    \end{tikzpicture}
    \caption*{(b) Disjunctive model}\end{minipage}\caption{Extended Hasse diagrams of $\QQ$ in Equation~\eqref{ex1-Q}.}\label{fig:hassewithitem}
\end{figure}
\end{example}

\subsubsection{Hasse-Based Distribution Graph of \texorpdfstring{$\QQ$}{Q}}\label{sec:G}
{
As discussed in Example~\ref{ex1}, multiple latent vector nodes in
\(\mch(\QQ)\) may have the same path status vector. Since each latent vector
\(\aaa\) is represented by a unique node \(v(\aaa)\), the path status vectors
induce an equivalence relation on the latent vectors. Specifically, we say that
two latent vectors \(\aaa\) and \(\aaa'\) are equivalent, denoted by
\(\aaa\sim\aaa'\), if \(\Psi(v(\aaa))=\Psi(v(\aaa'))\). Let
\([\aaa]=\{\aaa'\in\{0,1\}^K:\aaa'\sim\aaa\}\) denote the equivalence class of
\(\aaa\), and let \(\mathcal E(\QQ)=\{[\aaa]:\aaa\in\{0,1\}^K\}\) be the
collection of all equivalence classes. Since \(\Psi(v(\aaa))\) determines the
conditional distribution of \(\RR\) given \(\aaa\), latent vectors in the same
equivalence class have the same conditional distribution of \(\RR\), and therefore
are indistinguishable in their contributions to the marginal distribution of
\(\RR\).
For each equivalence class \(\mcc\in\mathcal E(\QQ)\), choose one representative
latent vector \(\aaa_{\mcc}\in\mcc\). The corresponding node
\(v(\aaa_{\mcc})\) is called the representative node of \(\mcc\). 
Denote the collection of all such nodes by $R(\QQ)=\{v(\aaa_{\mcc}):\mcc\in\mathcal E(\QQ)\}.$ Together with the $\qq$-vector nodes, they form the node set of \(G(\QQ)\): $V(\QQ)=$ $R(\QQ)\cup\{v_j^*(\qq_j):j\in[J]\}$. Using this node set, we now formally define $G(\QQ)$ analogously to the Hasse-diagram construction.

\color{black}
\begin{definition}[Hasse-based distribution graph of \(\QQ\)]
\label{def:GQ}
{\it
For any distinct
\(\mcc_1,\mcc_2\in\mathcal E(\QQ)\), we draw a directed edge from
\(v(\aaa_{\mcc_1})\) to \(v(\aaa_{\mcc_2})\) if and only if $
\Psi(v(\aaa_{\mcc_1}))
\prec
\Psi(v(\aaa_{\mcc_2}))$
and there does not exist any
\(\mcc_3\in\mathcal E(\QQ)\) such that $\Psi(v(\aaa_{\mcc_1}))
\prec
\Psi(v(\aaa_{\mcc_3}))
\prec
\Psi(v(\aaa_{\mcc_2}))$.
For every \(j\in[J]\) and \(\mcc\in\mathcal E(\QQ)\), we draw a
directed edge from \(v_j^*(\qq_j)\) to
\(v(\aaa_{\mcc})\) if and only if $\Psi_j(v(\aaa_{\mcc}))=1$
and there does not exist any
\(\mcc'\in\mathcal E(\QQ)\) such that $\Psi_j(v(\aaa_{\mcc'}))=1$ and $\Psi(v(\aaa_{\mcc'}))
\prec
\Psi(v(\aaa_{\mcc})).$
}
\end{definition}
% Since the construction is based on $\Psi$, \(G(\QQ)\) does not depend on the
% choice of representatives. 
}

\begin{remark}
{\it
    In each equivalence class, any of the elements could serve as a representative for the class. For the conjunctive and disjunctive models, $R(\QQ)$ can be chosen separately as:
\begin{equation}\label{eq:RQ_conj}
R(\QQ) = \left\{ v\paren{\underset{j \in \mcs }{\vee} \qq_j} \;\Big|\; \mcs \subseteq [J] \right\}\quad\text{and}\quad
R(\QQ) = \left\{ v\paren{\underset{j \in \mcs }{\wedge} \tilde{\qq}_j} \;\Big|\; \mcs \subseteq [J] \right\},
\end{equation}
where $\tilde{\qq}_j$ is obtained by flipping the zeros and ones in $\qq_j$. 
% The empty disjunction is \(\zero\), and the empty conjunction is \(\one\).
}
\end{remark}

% Our final $G(\QQ)$ is exactly this reduced version of the extended Hasse diagram of $\QQ$. Specifically, let $V(\QQ)$ and $E(\QQ)$ denote the node set and the edge set of $G(\QQ)$.  Then we have $V(\QQ) = R(\QQ) \bigcup I(\QQ)$, where   $R(\QQ)$ is as in \eqref{eq:R(QQ)} and $I(\QQ) = \left\{ v_j^*(\qq_j) \;\middle|\; j \in [J] \right\}$ denotes the set of $\qq$-vector nodes. The edge set $E(\QQ) = E_I(\QQ) \bigcup E_R(\QQ)$ comprises two subsets: the edge set which connects the $\qq$-vector nodes and the latent vector nodes:
% \begin{equation}\label{eq:E_I}
%     E_I(\QQ) = \bigcup_{j=1}^J \set{\paren{v_j^*(\qq_j), v(\aaa)} \mid \aaa \in \mcm_j^+},
% \end{equation}
% and the edge set which connects the latent vector nodes in $R(\QQ)$:
% \begin{equation}\label{eq:E_R}
%     E_R(\QQ) = \set{\paren{v(\aaa_1), v(\aaa_2)} \;\big| \;\aaa_2 \;\text{covers} \;\aaa_1, \;\text{for}\;v(\aaa_1), v(\aaa_2) \in R(\QQ)}.
% \end{equation}

The graphical construction involves several  intermediate objects. To help readers keep track of them, we provide a table summarizing their notation and definitions, together with a schematic figure illustrating the construction of
\(G(\QQ)\) in Section~\ref{app:graphical-guide} of the Supplementary Material.
 We next return to Example~\ref{ex1} to provide a concrete
illustration. 
\addtocounter{example}{-1}
\begin{example}[Continued]
{\it Using Equation~\eqref{eq:RQ_conj} to choose representatives among the equivalent nodes  grouped with dashed lines in Figure~\ref{fig:hassewithitem}, and, for visualization, suppressing the $\zero$-vector node (which has no edges from any nonzero $\qq$-vector), we obtain the simplified Hasse-based distribution graph $G(\QQ)$ shown in Figure~\ref{fig:G(Q)ex1}.}
\begin{figure}[htbp]
    \centering
    \begin{minipage}{0.5\textwidth}
    \centering
\begin{tikzpicture}[baseline=(current bounding box.center)]
    % Conjunctive model (simplified)
    \node[rectangle, draw, inner sep=5pt, fill=red!50, font=\normalfont] (110d) at (-1.5, 1) {110};
    \node[rectangle, draw, inner sep=5pt, fill=red!50, font=\normalfont] (101d) at (0, 1) {101};
    \node[rectangle, draw, inner sep=5pt, fill=red!50, font=\normalfont] (011d) at (1.5, 1) {011};
    \node[circle, draw, inner sep=2pt, fill=red!20, font=\normalfont] (110) at (-1.5, 2.5) {110};
    \node[circle, draw, inner sep=2pt, fill=red!20, font=\normalfont] (101) at (0, 2.5) {101};
    \node[circle, draw, inner sep=2pt, fill=red!20, font=\normalfont] (011) at (1.5, 2.5) {011};
    \node[circle, draw, inner sep=2pt, fill=red!20, font=\normalfont] (111) at (0, 4) {111};
    \draw[-{Latex[length=2.5mm]}] (110d) -- (110);
    \draw[-{Latex[length=2.5mm]}] (101d) -- (101);
    \draw[-{Latex[length=2.5mm]}] (011d) -- (011);
    \draw[-{Latex[length=2.5mm]}] (110) -- (111);
    \draw[-{Latex[length=2.5mm]}] (101) -- (111);
    \draw[-{Latex[length=2.5mm]}] (011) -- (111);
\end{tikzpicture}\caption*{(a) Conjunctive model}
    \end{minipage}%
    \hfill
    \begin{minipage}{0.5\textwidth}
    \centering
\begin{tikzpicture}[baseline=(current bounding box.center)]
    % Disjunctive model (simplified)
    \node[circle, draw, inner sep=2pt, font=\normalfont, fill=red!20] (100) at (-1.5, 1.5) {100};
    \node[circle, draw, inner sep=2pt, font=\normalfont, fill=red!20] (010) at (0, 1.5) {010};
    \node[circle, draw, inner sep=2pt, font=\normalfont, fill=red!20] (001) at (1.5, 1.5) {001};
    \node[circle, draw, inner sep=2pt, font=\normalfont, fill=red!20] (111) at (0, 3) {111};

    % $\qq$-vector nodes
    \node[rectangle, draw, inner sep=5pt, fill=red!50, font=\normalfont] (110d) at (-1.5, 0) {110};
    \node[rectangle, draw, inner sep=5pt, fill=red!50, font=\normalfont] (101d) at (0, 0) {101};
    \node[rectangle, draw, inner sep=5pt, fill=red!50, font=\normalfont] (011d) at (1.5, 0) {011};

    % Edges
    \draw[-{Latex[length=2.5mm]}] (100) -- (111);
    \draw[-{Latex[length=2.5mm]}] (010) -- (111);
    \draw[-{Latex[length=2.5mm]}] (001) -- (111);
    \draw[-{Latex[length=2.5mm]}] (110d) -- (100);
    \draw[-{Latex[length=2.5mm]}] (110d) -- (010);
    \draw[-{Latex[length=2.5mm]}] (101d) -- (100);
    \draw[-{Latex[length=2.5mm]}] (101d) -- (001);
    \draw[-{Latex[length=2.5mm]}] (011d) -- (001);
    \draw[-{Latex[length=2.5mm]}] (011d) -- (010);

    % Dashed rectangle around node 111
    \draw[dashed, thick] (-0.5, 2.5) rectangle (0.5, 3.5);
\end{tikzpicture}\caption*{(b) Disjunctive model}
\end{minipage}
\caption{Hasse-based distribution graph $G(\QQ)$ for the $\QQ$ in Equation~\eqref{ex1-Q}.}
\label{fig:G(Q)ex1}
\end{figure}
\end{example}

%\begin{remark}
    %For all the examples in this work, we only consider $\qq \neq \zero$. 
%    Given that we also assume that the latent $\zero$-vector can not produce a positive response to any nonzero $\qq$-vector, there are no edges from all the non-zero $\qq$-vector nodes to the latent $\zero$-vector node in the graph.

% For simplicity of visualization, we can further remove the $\zero$-vector node in the graph, since there are no edges from all the non-zero $\qq$-vector nodes to the latent $\zero$-vector node. After removing the $\zero$-vector node in Figures~\ref{fig:hasse_merged}(a) and~\ref{fig:hasse_merged}(b), the resulting graphs are presented in Figures~\ref{fig:G(Q)ex1}(a) and~\ref{fig:G(Q)ex1}(b), respectively.
%\end{remark}

\vspace{-0.4cm}
\subsection{ Sufficient and Necessary Graphical Condition}\label{sec:graphcond}
A $\QQ$-matrix is \emph{not} identifiable when there exists another matrix $\bQQ \nsim \QQ$ that could yield the same distribution of $\RR$. Since each \(\QQ\) defines a graph \(G(\QQ)\) that encodes the distribution of \(\RR\), 
as we will show, a matrix $\QQ$ is identifiable when $G(\QQ)$ is unique in the sense that there exists no other \(\bQQ\nsim\QQ\) such that \(G(\QQ)\) can be obtained from \(G(\bQQ)\).
Before stating this result,  we first introduce the following node deletion operation.
For \(A\subseteq R(\QQ)\), define \(G(\QQ)-A\) by removing the nodes
in \(A\) and their incident edges, and then redrawing the edges
according to Definition~\ref{def:GQ}, restricted to the remaining
representative nodes.
We now formalize when \(G(\QQ)\) can be obtained from \(G(\bQQ)\) using the notion of \emph{graph isomorphism}. Intuitively, two
graphs are isomorphic if one can be obtained from the other by
relabeling the nodes while preserving the edge structure. Since
the $\qq$-vector nodes $v_1^*(\qq_1), \ldots, v_J^*(\qq_J)$ correspond to the observed variables $R_1,\ldots,R_J$, we regard them as labeled by their indices and require these labels to be preserved under graph isomorphism.
% Intuitively, after node deletion, two graphs are isomorphic if their latent  nodes can be relabeled, while the item labels are preserved,  so that every directed edge in one graph corresponds to a directed edge in the other, and vice versa.

{\color{black}
% \begin{definition}\label{def:iso}
% {\it
% Let \(G(\QQ)\) and \(G(\bQQ)\) be two Hasse-based distribution graphs.
% We say that \(G(\QQ)\) can be obtained from \(G(\bQQ)\), when $G(\QQ)$ is isomorphic to $G(\bQQ)-A$ for some $A \subseteq R(\bQQ)$, denoted by $G(\QQ)\sim G(\bQQ)-A$. That is, when there exist \(A\subseteq R(\bQQ)\) and a bijection
% \(f:V(\QQ)\to V(\bQQ)\setminus A\) such that $f\bigl(v_j^*(\qq_j)\bigr)=v_j^*(\bqq_j)$ for every $j\in[J],$
% and, for every \(u,v\in V(\QQ)\), $(u,v)$ is an edge in $G(\QQ)$ if and only if $\bigl(f(u),f(v)\bigr)$ is an edge in $G(\bQQ)-A.$
% % \[
% % (u,v)\text{ is an edge in }G(\QQ)
% % \quad\Longleftrightarrow\quad
% % \bigl(f(u),f(v)\bigr)\text{ is an edge in }G(\bQQ)-A.
% % \]
% }
% \end{definition}
\begin{definition}\label{def:iso}
{\it
Let \(G(\QQ)\) and \(G(\bQQ)\) be two Hasse-based distribution graphs.
We say that \(G(\QQ)\) can be obtained from \(G(\bQQ)\) if \(G(\QQ)\) is isomorphic to \(G(\bQQ)-A\) for some \(A\subseteq R(\bQQ)\). We denote this relation by \(G(\QQ)\sim G(\bQQ)-A\). 
}
\end{definition}
}
% \begin{definition}[Graph isomorphism]
% {\it
% Let \(G(\QQ)\) and \(G(\bQQ)\) be two Hasse-based distribution graphs.
% We say that \(G(\QQ)\) can be obtained from \(G(\bQQ)\), up to graph
% isomorphism, if there exist \(A\subseteq R(\bQQ)\) and a bijection
% \(f:V(\QQ)\to V(\bQQ)\setminus A\) such that $f\bigl(v_j^*(\qq_j)\bigr)=v_j^*(\bqq_j)$ for every $j\in[J],$
% and, for every \(u,v\in V(\QQ)\),
% \[
% (u,v)\text{ is an edge in }G(\QQ)
% \quad\Longleftrightarrow\quad
% \bigl(f(u),f(v)\bigr)\text{ is an edge in }G(\bQQ)-A.
% \]
% Here \(G(\bQQ)-A\) is obtained by successively deleting the nodes in
% \(A\). Specifically, when deleting \(v(\aaa)\in A\), remove the node and all its connected edges, and then draw an edge from each parent \(u\)
% of \(v(\aaa)\) to each child \(w\) of \(v(\aaa)\) if no directed path from \(u\) to \(w\) remains. 
% % This operation leaves the path status vectors of all remaining representative nodes unchanged.

% For such an \(A\), we write \(G(\QQ)\sim G(\bQQ)-A\) and say that
% \(G(\QQ)\) is isomorphic to \(G(\bQQ)-A\).
% }
% \end{definition}
 
We can now state our sufficient and necessary graphical condition for identifying~\(\QQ\).
Recall that $p_{[\aaa]} \;=\; \sum_{\vv \in [\aaa]} p_{\vv}$ denotes the grouped proportion for the equivalence class \([\aaa]\in \mathcal{E}(\QQ)\) under $\QQ$, 
{
\color{black}
and $\TTT = p_{[\Cdot]} = (p_{[\aaa]})_{[\aaa]\in \mathcal{E}(\QQ)}$ under the Boolean factor model in \eqref{eq:Boolean}.
}
\begin{theorem}\label{thm:graph}
Consider a Boolean factor model in \eqref{eq:Boolean} with \(p_{[\aaa]} > 0\) for 
\([\aaa] \in \mathcal{E}(\QQ)\). 
Then \(\QQ\) is identifiable if and only if
{\color{black}\((\QQ,p_{[\Cdot]})\) is jointly identifiable,
which holds if and only if}
the following graphical condition holds:
\begin{equation}\label{eq:g_condition}
\text{There does not exist } \bQQ \nsim \QQ \text{ and } A \subseteq R(\bQQ) \text{ such that } G(\bQQ) - A \sim G(\QQ).
\end{equation}
\end{theorem}
\noindent
{\color{black}
In other words, \(\QQ\) is identifiable when there is no \(\bQQ \nsim \QQ\) whose $G(\bQQ)$, after node deletion, is isomorphic to~\(G(\QQ)\). That is, the graphical structure of $G(\QQ)$ is unique on its own. In Section~\ref{app:theory_intuition} of the Supplementary Material, we provide further intuition for Theorem~\ref{thm:graph}.}

To illustrate the application of Theorem~\ref{thm:graph}, we return to Example~\ref{ex1}.

\addtocounter{example}{-1}
\begin{example}[Continued]
{\it Consider the conjunctive and disjunctive models in Example~\ref{ex1}, where their $G(\QQ)$ are illustrated in Figure~\ref{fig:G(Q)ex1}. 
Consider \(\bQQ = \cI_3\). Under both of the conjunctive and disjunctive models, \(G(\cI_3)\) is identical, as shown in Figure~\ref{fig:thmex}\textnormal{(a)}.
Define subsets \(A_1 = \{v(110), v(101), v(011)\}\) and \(A_2 = \{v(100), v(010), v(001)\}\). Deleting nodes in these subsets yields \(G(\cI_3) - A_1\) and \(G(\cI_3) - A_2\), respectively, as shown in Figure~\ref{fig:thmex}\textnormal{(b)} and Figure~\ref{fig:thmex}\textnormal{(c)}.
Under the conjunctive model, \(G(\cI_3) - A_1 \sim G(\QQ)\), whereas under the disjunctive model, \(G(\cI_3) - A_2 \sim G(\QQ)\). By Theorem~\ref{thm:graph}, these findings imply that the \(\QQ\)-matrix \eqref{ex1-Q} in Example~\ref{ex1} is not identifiable under either the conjunctive or the disjunctive model.}

\begin{figure}[htbp]
    \centering
    \begin{minipage}{0.3\textwidth}
    \centering
\begin{tikzpicture}[baseline=(current bounding box.center), scale=0.8]
    % Circle Nodes
    \node[circle, draw, inner sep=2pt, fill=red!20, font=\normalfont] (100) at (-1.5, 2.5) {100};
    \node[circle, draw, inner sep=2pt, fill=red!20, font=\normalfont] (010) at (0, 2.5) {010};
    \node[circle, draw, inner sep=2pt, fill=red!20, font=\normalfont] (001) at (1.5, 2.5) {001};
    \node[circle, draw, inner sep=2pt, fill=red!20, font=\normalfont] (110) at (-1.5, 4) {110};
    \node[circle, draw, inner sep=2pt, fill=red!20, font=\normalfont] (101) at (0, 4) {101};
    \node[circle, draw, inner sep=2pt, fill=red!20, font=\normalfont] (011) at (1.5, 4) {011};
    \node[circle, draw, inner sep=2pt, fill=red!20, font=\normalfont] (111) at (0, 5.5) {111};
    % Lower Nodes (replicated)
    \node[rectangle, draw, inner sep=5pt, fill=red!50, font=\normalfont] (100d) at (-1.5, 1) {100};
    \node[rectangle, draw, inner sep=5pt, fill=red!50, font=\normalfont] (010d) at (0, 1) {010};
    \node[rectangle, draw, inner sep=5pt, fill=red!50, font=\normalfont] (001d) at (1.5, 1) {001};
    % Edges with arrows
    \draw[-{Latex[length=2.5mm]}] (100) -- (110);
    \draw[-{Latex[length=2.5mm]}] (100) -- (101);
    \draw[-{Latex[length=2.5mm]}] (010) -- (110);
    \draw[-{Latex[length=2.5mm]}] (010) -- (011);
    \draw[-{Latex[length=2.5mm]}] (001) -- (101);
    \draw[-{Latex[length=2.5mm]}] (001) -- (011);
    \draw[-{Latex[length=2.5mm]}] (110) -- (111);
    \draw[-{Latex[length=2.5mm]}] (101) -- (111);
    \draw[-{Latex[length=2.5mm]}] (011) -- (111);
    \draw[-{Latex[length=2.5mm]}] (100d) -- (100);
    \draw[-{Latex[length=2.5mm]}] (010d) -- (010);
    \draw[-{Latex[length=2.5mm]}] (001d) -- (001);    \end{tikzpicture}\captionsetup{labelformat=empty}\caption*{(a) $G(\cI_3)$}\label{fig:G(I_3)}
    \end{minipage}%
    \hfill
    \begin{minipage}{0.3\textwidth}
    \centering
\begin{tikzpicture}[baseline=(current bounding box.center), scale=0.8]
    % Circle Nodes
    \node[circle, draw, inner sep=2pt, fill=red!20, font=\normalfont] (100) at (-1.5, 2.5) {100};
    \node[circle, draw, inner sep=2pt, fill=red!20, font=\normalfont] (010) at (0, 2.5) {010};
    \node[circle, draw, inner sep=2pt, fill=red!20, font=\normalfont] (001) at (1.5, 2.5) {001};
    \node[circle, draw, inner sep=2pt, fill=red!20, font=\normalfont] (111) at (0, 5.5) {111};
    % Lower Nodes (replicated)
    \node[rectangle, draw, inner sep=5pt, fill=red!50, font=\normalfont] (100d) at (-1.5, 1) {100};
    \node[rectangle, draw, inner sep=5pt, fill=red!50, font=\normalfont] (010d) at (0, 1) {010};
    \node[rectangle, draw, inner sep=5pt, fill=red!50, font=\normalfont] (001d) at (1.5, 1) {001};
    % Edges
    \draw[-{Latex[length=2.5mm]}] (100) -- (111);
    \draw[-{Latex[length=2.5mm]}] (010) -- (111);
    \draw[-{Latex[length=2.5mm]}] (001) -- (111);
    \draw[-{Latex[length=2.5mm]}] (100d) -- (100);
    \draw[-{Latex[length=2.5mm]}] (010d) -- (010);
    \draw[-{Latex[length=2.5mm]}] (001d) -- (001);
\end{tikzpicture}\captionsetup{labelformat=empty}\caption*{(b) $G(\cI_3) - A_1$}
\end{minipage}
\hfill
    \begin{minipage}{0.3\textwidth}
    \centering
\begin{tikzpicture}[baseline=(current bounding box.center), scale=0.8]
    % Circle Nodes
    \node[circle, draw, inner sep=2pt, fill=red!20, font=\normalfont] (110) at (-1.5, 4) {110};
    \node[circle, draw, inner sep=2pt, fill=red!20, font=\normalfont] (101) at (0, 4) {101};
    \node[circle, draw, inner sep=2pt, fill=red!20, font=\normalfont] (011) at (1.5, 4) {011};
    \node[circle, draw, inner sep=2pt, fill=red!20, font=\normalfont] (111) at (0, 5.5) {111};
    % Lower Nodes (replicated)
    \node[rectangle, draw, inner sep=5pt, fill=red!50, font=\normalfont] (100d) at (-1.5, 1) {100};
    \node[rectangle, draw, inner sep=5pt, fill=red!50, font=\normalfont] (010d) at (0, 1) {010};
    \node[rectangle, draw, inner sep=5pt, fill=red!50, font=\normalfont] (001d) at (1.5, 1) {001};
    % Edges with arrows
    \draw[-{Latex[length=2.5mm]}] (110) -- (111);
    \draw[-{Latex[length=2.5mm]}] (101) -- (111);
    \draw[-{Latex[length=2.5mm]}] (011) -- (111);
    \draw[-{Latex[length=2.5mm]}] (100d) -- (110);
    \draw[-{Latex[length=2.5mm]}] (100d) -- (101);
    \draw[-{Latex[length=2.5mm]}] (010d) -- (110);
    \draw[-{Latex[length=2.5mm]}] (010d) -- (011);
    \draw[-{Latex[length=2.5mm]}] (001d) -- (101);
    \draw[-{Latex[length=2.5mm]}] (001d) -- (011);
    \end{tikzpicture}\captionsetup{labelformat=empty}\caption*{(c) $G(\cI_3) - A_2$}
\end{minipage}\caption{$G(\cI_3)$ and two graphs after deleting different nodes from $G(\cI_3)$.}\label{fig:thmex}
\end{figure}
\end{example}

%The equivalence of the identifiability of the $\QQ$-matrix under both the conjunctive and disjunctive models is not limited to this specific example but is, in fact, a general result. This equivalence follows from the duality between the two models, as established in prior work \citep{chen2015statistical}. In other words, a $\QQ$-matrix is identifiable in one of the model if and only if it is identifiable in the other.
The following example provides two identifiable $\QQ$-matrices under both the conjunctive and disjunctive models. For simplicity, here we focus on the conjunctive model.

\begin{example}\label{ex2}
{\it Consider two conjunctive models with $\QQ$-matrices given in Equation~\eqref{eq:ex_two_Q}, along with their corresponding $G(\QQ)$ in Figure~\ref{eq:ex2-G}. For each of these matrices $\QQ^{(i)}$, for $i = 1,2$, it can be verified that for any \(\bQQ \nsim \QQ^{(i)}\), no matter which subset \(A \subseteq R(\bQQ)\) is chosen, removing the nodes in \(A\) from \(G(\bQQ)\) does not produce a graph isomorphic to \(G(\QQ^{(i)})\). By Theorem~\ref{thm:graph}, we conclude that both $\QQ$-matrices are identifiable. Notably, neither of these identifiable $\QQ$-matrices is complete, and the second $\QQ$-matrix does not even contain any pure nodes. The identifiability of these matrices comes solely from the uniqueness of their graphs $G(\QQ)$.}

\begin{equation}\label{eq:ex_two_Q}
{\renewcommand{\arraystretch}{0.5}
\QQ^{(1)} = 
        \begin{pmatrix}
        1 & 1 & 0 \\
        1 & 0 & 1\\
        0 & 1 & 0 \\
        0 & 0 & 1 
        \end{pmatrix}
        \qquad
    \QQ^{(2)} = \begin{pmatrix}
    1 & 1 & 0 & 0 \\
    1 & 0 & 1 & 0 \\
    1 & 0 & 0 & 1 \\
    0 & 1 & 1 & 0 \\
    0 & 1 & 0 & 1 \\
    0 & 0 & 1 & 1 \\
\end{pmatrix}
}
\end{equation}
\begin{figure}[h!]
    \centering
    \begin{minipage}[t]{0.3\textwidth}
        \centering
        \begin{tikzpicture}[baseline=(current bounding box.center), scale=0.8]
            % Circle Nodes
            \node[circle, draw, inner sep=2pt, fill=red!20, font=\normalfont] (010) at (0, 2.5) {010};
            \node[circle, draw, inner sep=2pt, fill=red!20, font=\normalfont] (001) at (1.5, 2.5) {001};
            \node[circle, draw, inner sep=2pt, fill=red!20, font=\normalfont] (110) at (-1.5, 4) {110};
            \node[circle, draw, inner sep=2pt, fill=red!20, font=\normalfont] (101) at (0, 4) {101};
            \node[circle, draw, inner sep=2pt, fill=red!20, font=\normalfont] (011) at (1.5, 4) {011};
            \node[circle, draw, inner sep=2pt, fill=red!20, font=\normalfont] (111) at (0, 5.5) {111};
            % Lower Nodes (replicated)
            \node[rectangle, draw, inner sep=5pt, fill=red!50, font=\normalfont] (110d) at (-3, 2.5) {110};
            \node[rectangle, draw, inner sep=5pt, fill=red!50, font=\normalfont] (101d) at (-1.5, 2.5) {101};
            \node[rectangle, draw, inner sep=5pt, fill=red!50, font=\normalfont] (010d) at (0, 1) {010};
            \node[rectangle, draw, inner sep=5pt, fill=red!50, font=\normalfont] (001d) at (1.5, 1) {001};
            
            % Edges with arrows
            \draw[-{Latex[length=2.5mm]}] (010) -- (110);
            \draw[-{Latex[length=2.5mm]}] (010) -- (011);
            \draw[-{Latex[length=2.5mm]}] (001) -- (101);
            \draw[-{Latex[length=2.5mm]}] (001) -- (011);
            \draw[-{Latex[length=2.5mm]}] (110) -- (111);
            \draw[-{Latex[length=2.5mm]}] (101) -- (111);
            \draw[-{Latex[length=2.5mm]}] (011) -- (111);
            \draw[-{Latex[length=2.5mm]}] (101d) -- (101);
            \draw[-{Latex[length=2.5mm]}] (110d) -- (110);
            \draw[-{Latex[length=2.5mm]}] (010d) -- (010);
            \draw[-{Latex[length=2.5mm]}] (001d) -- (001);

\end{tikzpicture}\caption*{(a) $G(\QQ^{(1)})$}
    \end{minipage}
    \hspace{0.3cm} % Adjust the spacing between the matrix and the graph
    \begin{minipage}[t]{0.65\textwidth}
    \centering

    \begin{tikzpicture}[baseline=(current bounding box.center), scale=0.7, auto, swap]
        % Latent vector nodes (reversed order)
        \node[circle, draw, inner sep=1pt, fill=red!20, font=\normalfont] (1100) at (0,1.5) {1100};
        \node[circle, draw, inner sep=1pt, fill=red!20, font=\normalfont] (1010) at (2,1.5) {1010};
        \node[circle, draw, inner sep=1pt, fill=red!20, font=\normalfont] (1001) at (4,1.5) {1001};
        \node[circle, draw, inner sep=1pt, fill=red!20, font=\normalfont] (0110) at (6,1.5) {0110};
        \node[circle, draw, inner sep=1pt, fill=red!20, font=\normalfont] (0101) at (8,1.5) {0101};
        \node[circle, draw, inner sep=1pt, fill=red!20, font=\normalfont] (0011) at (10,1.5) {0011};

        \node[circle, draw, inner sep=1pt, fill=red!20, font=\normalfont] (1110) at (2,3.5) {1110};
        \node[circle, draw, inner sep=1pt, fill=red!20, font=\normalfont] (1101) at (4,3.5) {1101};
        \node[circle, draw, inner sep=1pt, fill=red!20, font=\normalfont] (1011) at (6,3.5) {1011};
        \node[circle, draw, inner sep=1pt, fill=red!20, font=\normalfont] (0111) at (8,3.5) {0111};

        \node[circle, draw, inner sep=1pt, fill=red!20, font=\normalfont] (1111) at (5,5.5) {1111};

        % $\qq$-vector nodes (reversed order)
        \node[rectangle, draw, inner sep=5pt, fill=red!50, font=\normalfont] (1100d) at (0,-0.5) {1100};
        \node[rectangle, draw, inner sep=5pt, fill=red!50, font=\normalfont] (1010d) at (2,-0.5) {1010};
        \node[rectangle, draw, inner sep=5pt, fill=red!50, font=\normalfont] (1001d) at (4,-0.5) {1001};
        \node[rectangle, draw, inner sep=5pt, fill=red!50, font=\normalfont] (0110d) at (6,-0.5) {0110};
        \node[rectangle, draw, inner sep=5pt, fill=red!50, font=\normalfont] (0101d) at (8,-0.5) {0101};
        \node[rectangle, draw, inner sep=5pt, fill=red!50, font=\normalfont] (0011d) at (10,-0.5) {0011};
        % Edges (connections remain unchanged)
        \draw[-{Latex[length=2.5mm]}] (0011) -- (0111);
        \draw[-{Latex[length=2.5mm]}] (0111) -- (1111);
        \draw[-{Latex[length=2.5mm]}] (1001) -- (1101);
        \draw[-{Latex[length=2.5mm]}] (1101) -- (1111);
        \draw[-{Latex[length=2.5mm]}] (0101) -- (0111);
        \draw[-{Latex[length=2.5mm]}] (1100) -- (1110);
        \draw[-{Latex[length=2.5mm]}] (0011) -- (1011);
        \draw[-{Latex[length=2.5mm]}] (1011) -- (1111);
        \draw[-{Latex[length=2.5mm]}] (0101) -- (1101);
        \draw[-{Latex[length=2.5mm]}] (1001) -- (1011);
        \draw[-{Latex[length=2.5mm]}] (1010) -- (1011);
        \draw[-{Latex[length=2.5mm]}] (1100) -- (1101);
        \draw[-{Latex[length=2.5mm]}] (0110) -- (0111);
        \draw[-{Latex[length=2.5mm]}] (1010) -- (1110);
        \draw[-{Latex[length=2.5mm]}] (0110) -- (1110);
        \draw[-{Latex[length=2.5mm]}] (1110) -- (1111);

        \draw[-{Latex[length=2.5mm]}] (1100d) -- (1100);
        \draw[-{Latex[length=2.5mm]}] (1010d) -- (1010);
        \draw[-{Latex[length=2.5mm]}] (0101d) -- (0101);
        \draw[-{Latex[length=2.5mm]}] (0011d) -- (0011);
        \draw[-{Latex[length=2.5mm]}] (1001d) -- (1001);
        \draw[-{Latex[length=2.5mm]}] (0110d) -- (0110);
        \end{tikzpicture}\caption*{(b) $G(\QQ^{(2)})$}
\end{minipage}
\caption{Hasse-based distribution graphs of \(\QQ^{(1)}\) and \(\QQ^{(2)}\) in Equation~\eqref{eq:ex_two_Q}. }\label{eq:ex2-G}
\end{figure}
\end{example}
The above example demonstrates that the identifiability of the $\QQ$-matrix is determined by the uniqueness of the graphical representation $G(\QQ)$.
Such uniqueness pertains purely to the graphical structure, which demonstrates that the existence of pure nodes is not required for the identifiability of the $\QQ$-matrix.

\subsection{Algebraic Condition and Computation Algorithm}\label{sec:algebraic}
Coupled with the graphical condition, we next provide the algebraic sufficient and necessary condition for the identifiability of the $\QQ$-matrix. 
To facilitate the presentation, 
for any 
\(\QQ\in\{0,1\}^{J\times K}\), we define \(\PPhi(\QQ)\in\{0,1\}^{J\times 2^K}\) via the extended Hasse graph of \(\QQ\) as follows:
%via the extended Hasse graph of \(\QQ\). Concretely,
each column of \(\PPhi(\QQ)\) is indexed by a latent vector 
\(\aaa\in\{0,1\}^K\), and the corresponding \(\aaa\)-column, denoted by $\PPhi_{\Cdot,\aaa}(\QQ)$, is set as the path status vector of $v(\aaa)$, i.e., \(\PPhi_{\Cdot,\aaa}(\QQ) = \Psi\bigl(v(\aaa)\bigr)\). We have the following algebraic theorem.
\begin{theorem}\label{thm:algebraic}
Consider a Boolean factor model in \eqref{eq:Boolean} with \(p_{[\aaa]} > 0\) for \([\aaa] \in \mathcal{E}(\QQ)\). 
%In a Boolean factor model with $\QQ$-matrix, suppose that \(p_{[\aaa]} > 0\) holds for every \([\aaa] \in \mathcal{E}(\QQ)\), 
Then \(\QQ\) is identifiable if and only if
{\color{black}\((\QQ,p_{[\Cdot]})\) is jointly identifiable,
which holds if and only if}
the following algebraic condition holds:
\begin{equation}\label{eq:a_condition}
\begin{split}
\text{There does not exist } \bQQ \nsim \QQ \text{ such that } \operatorname{Cols}\bigl(\PPhi(\QQ)\bigr) \subseteq \operatorname{Cols}\bigl( \PPhi(\bQQ) \bigr),
\end{split}
\end{equation}
where  \(\text{Cols}(\cdot)\) denotes the set of all column vectors of a matrix. 
\end{theorem}

% \begin{remark}
%     In this work, we assume that $p_{\aaa}>0$ for all $\aaa \in \{0,1\}^K$. However, from the proofs of the main theorems, we can see that such assumption is not necessary. As the population proportion parameters only matter at the equivalence class level. That is, we only require the grouped population proportion $p_{[\aaa]}$ to be non-zero, and this holds when there is at least one latent vector $\aaa \in [\aaa]$  such that $p_{\aaa} > 0$. 
%     % Such partial identifiability notion has been considered in the literature \citep{partial}, we defer this exploration in the future.
% \end{remark}
\noindent
% This algebraic condition \eqref{eq:a_condition} is equivalent to the graphical condition  \eqref{eq:g_condition},  and it reaffirms that the identifiability of the $\QQ$-matrix does not depend on the presence of pure nodes. 
{\color{black}
Theorems~\ref{thm:graph} and \ref{thm:algebraic} are equivalent, and they both show that pure nodes are not required for identifiability. Nevertheless, the classical pure node condition is recovered as a sufficient condition under the conjunctive and disjunctive models, as we formalize in the following.
\begin{corollary}\label{cor:pure-node}
Under the conjunctive or disjunctive models, if \(\QQ\) is complete, then \(\QQ\) is identifiable.
\end{corollary}
}
% To illustrate, consider again the $\QQ$-matrix in Example~\ref{ex2}. One can verify that there is no $\bQQ \nsim \QQ$ such that $\text{Cols}\bigl(\PPhi^R(\QQ)\bigr) \subseteq \text{Cols}(\PPhi^R(\bQQ))$, and thus conclude that this $\QQ$-matrix, which does not contain any pure nodes, is indeed identifiable. 

While Theorems~\ref{thm:graph} and~\ref{thm:algebraic} precisely establish the sufficient and necessary conditions for the identifiability of the $\QQ$-matrix, verifying these conditions requires examining \emph{all} 
\(\bQQ \nsim \QQ\). For a $\QQ \in \{0,1\}^{J \times K}$,  the number of candidate $\bQQ$-matrices could grow exponentially as $O(2^{JK})$, making direct computation challenging.

%We show that when the operator $\ast$ is instantiated, 
To address this issue, we show that exploiting the structural properties of the $\QQ$-matrix can significantly reduce this computational complexity. 
%More importantly, the existential, yes–or–no nature of the question—\emph{does there exist a matrix $\bQQ$ satisfying the conditions?}—naturally frames it as a pure \emph{feasibility decision problem}.
Furthermore, the yes–or–no nature of the question—\emph{does there exist a matrix $\bQQ$ satisfying the conditions?}—naturally frames it as a \emph{feasibility  problem}, and the Boolean structure of the problem means it can be expressed entirely using logical operations and hence formulated as a \emph{Boolean satisfiability (SAT)} problem.
By formulating our identifiability conditions explicitly through logical propositional formulas, we can efficiently leverage modern SAT solvers to practically and efficiently determine the identifiability of any given $\QQ$.

%to demonstrate our methods and, for the remainder of this section, assume that the operator is conjunctive and that the partial order is the product order $\preceq$.
%Leveraging the structure of the operator, 

% feasible space from \(2^{JK}\) possibilities to \(\prod_{j=1}^J\sum_{k=0}^{C_j(\mathbf{Q})}\binom{K}{k}\).
% \subsubsection{Structural Properties and SAT Formulation of the Identifiability Problem}\label{sec:SAT}
\subsubsection{Structural Properties and SAT Verification of the
Graphical Condition}\label{sec:SAT}
To illustrate our method, we focus on the conjunctive operator \eqref{eq:ope-conj} in this subsection; all the results easily extend to the disjunctive case  \eqref{eq:ope-disj2}. 
We first establish structural properties that simplify verification of the algebraic condition \eqref{eq:a_condition}. Specifically,
we relate the condition for $\QQ$ to that for its basis
submatrix (defined below), and its row and column submatrices.
\begin{definition}[Basis submatrix $\Qbasis$]\label{def:Qbasis}
{\it For a given $\QQ$, a non-zero \(\qq\)-vector \(\qq_{j}\) is a \emph{basis} \(\qq\)-vector if there is no nonempty subset \(\mcs\subseteq[J]\) satisfying \(\qq_{j}=\bigvee_{j'\in\mcs}\qq_{j'}\) with \(\qq_{j'}\neq\qq_{j}\); that is, \(\qq_j\) cannot be expressed as a logical `` $\bigvee$'' combination of other distinct \(\qq\)-vectors. The \emph{basis submatrix} \(\Qbasis\) is the submatrix consisting of one copy of each distinct basis \(\qq\)-vector of \(\QQ\).}
\end{definition}
\begin{theorem}\label{thm:sub}
Under the conjunctive operator, suppose that $\QQ \neq \zero$. Then
\textup{(i)} \(\QQ\) satisfies the algebraic condition if and only if
\(\Qbasis\) satisfies the algebraic condition;
\textup{(ii)} if \(\QQ\) contains a row submatrix satisfying the
algebraic condition, then \(\QQ\) satisfies the algebraic condition;
\textup{(iii)} if \(\QQ\) contains a column submatrix that does not
satisfy the algebraic condition, then \(\QQ\) does not satisfy the
algebraic condition.
\end{theorem}

Theorem~\ref{thm:sub}\textup{(i)} shows that verifying the
algebraic condition for $\QQ$ reduces to verifying it for
$\Qbasis$. This reduction can be particularly useful when
$J$ is large, since $\Qbasis$ has
$O\bigl(\binom{K}{\lfloor K/2\rfloor}\bigr)$ rows
\citep{kleitman1976extremal}.
Parts~\textup{(ii)--(iii)} allow identifiability results for
smaller matrices to be used in verifying the algebraic
condition. For part~\textup{(ii)},
Corollary~\ref{cor:pure-node} shows that completeness is
sufficient. For part~\textup{(iii)}, Lemma~\ref{lm:onetwo}
in the Supplementary Material shows that a two-column
$\QQ$-matrix satisfies the algebraic condition if and only if it is complete.
This yields the following necessary condition for the
algebraic condition, which will also be used in our
probabilistic identifiability results.

{\color{black}
\begin{definition}[Two-column condition]\label{def:two-column}
{\it
For $K\geq2$, a matrix $\QQ\in\{0,1\}^{J\times K}$
satisfies the \emph{two-column condition} if, for every
distinct $k,\ell\in[K]$, the submatrix
$\QQ_{\Cdot,\{k,\ell\}}$ contains both row patterns
$(1,0)$ and $(0,1)$.
}
\end{definition}

Similarly, Lemma~\ref{lm:threefour} in the Supplementary Material implies that, if $\QQ$ satisfies
the algebraic condition, every three-column submatrix
must contain a row with exactly one 1.
}
% These requirements must hold for all \(\binom{K}{2}\) pairs and
% \(\binom{K}{3}\) triples of columns. 
% We empirically investigate the effects of these submatrix conditions in Section~\ref{sec:sim_preprocessing} of the Supplementary Material.

The preceding results reduce the size of $\QQ$ and provide
preliminary checks of the algebraic condition.
For matrices that remain unresolved, the following proposition
allows us to formulate verification as a SAT problem.
{\color{black}
\begin{proposition}\label{prop:sat-nonequivalence}
Under the conjunctive operator, let
$\QQ\in\{0,1\}^{J\times K}$ have its columns arranged
in weakly non-increasing lexicographic order.
Then $\QQ$ satisfies the algebraic condition if and only if
there do not exist $\bQQ\in\{0,1\}^{J\times K}$,
also with its columns arranged in weakly non-increasing
lexicographic order, and $\HH\in\{0,1\}^{K\times K}$
such that
\begin{equation}\label{eq:sat_cond}
    \QQ=\bQQ\odot\HH,
    \qquad
    \bQQ\neq\QQ.
\end{equation}
\end{proposition}

Proposition~\ref{prop:sat-nonequivalence} yields a Boolean
satisfiability (SAT) formulation, which asks whether binary
variables can jointly satisfy a collection of logical
constraints. Introduce binary variables $x_{j\ell}$ and
$h_{\ell k}$ representing $\bar q_{j\ell}$ and the entries
of $\HH$, respectively. Then $q_{jk}
=
\bigvee_{\ell=1}^K
\bigl(x_{j\ell}\wedge h_{\ell k}\bigr),$ for $j\in[J],\ k\in[K]$. The constraint $\bQQ \neq \QQ$ in ~\eqref{eq:sat_cond} can be encoded by
\begin{equation}\label{eq:sat-factorization}
\begin{gathered}
\left(\bigvee_{(j,k):\,q_{jk}=0}x_{jk}\right)
\vee
\left(\bigvee_{(j,k):\,q_{jk}=1}\neg x_{jk}\right)
=1.
\end{gathered}
\end{equation}

The encoding of the lexicographic ordering constraints in
Proposition~\ref{prop:sat-nonequivalence} is deferred to
Section~\ref{app:sat-primer} of the Supplementary Material.
For readers unfamiliar with SAT, that section also provides
a gentle introduction.
A satisfying assignment certifies that $\QQ$ does not
satisfy the algebraic condition, whereas unsatisfiability
certifies that it does.
Combining this formulation with the preceding reductions
gives Algorithm~\ref{alg:id}.
We implement the formulation using PySAT, with
\texttt{Glucose42} as the default solver.
After basis reduction, the factorization constraints admit
an encoding of size $O(J_bK^2)$, where $J_b$ is the number
of rows of $\Qbasis$.
See Section~\ref{app:sub_comp_comx} of the Supplementary
Material for a discussion of computational complexity.
The formulation can also quantify the extent of
non-identifiability by enumerating alternative matrices
$\bQQ$ up to column permutation; see
Section~\ref{app:quantify_non_id} of the Supplementary Material.
}

\renewcommand{\algorithmicrequire}{\textbf{Input:}}
\renewcommand{\algorithmicensure}{\textbf{Output:}}
\begin{algorithm}[ht]
  \caption{Verification of the algebraic condition}
  \label{alg:id}
  \begin{algorithmic}[1]
    \REQUIRE Binary matrix
      $\QQ\in\{0,1\}^{J\times K}$ with $K\geq2$.
    \ENSURE ``Satisfies the algebraic condition'' or
      ``Does not satisfy the algebraic condition.''

    \STATE \textbf{Step 0: Basis reduction.}
      Compute $\Qbasis$, set $\QQ\gets\Qbasis$.

    \STATE \textbf{Step 1: Submatrix checks.}
    \IF{$\QQ$ contains an identity submatrix $\cI_K$}
      \STATE \textbf{return} Satisfies the algebraic condition
    \ENDIF
    \IF{$\QQ$ does not satisfy the two-column condition}
      \STATE \textbf{return} Does not satisfy the algebraic condition
    \ENDIF
    \IF{some three-column submatrix contains no row
         having exactly one 1}
      \STATE \textbf{return} Does not satisfy the algebraic condition
    \ENDIF

    \STATE \textbf{Step 2: SAT verification.}
    \STATE Arrange the columns of $\QQ$ in weakly non-increasing
      lexicographic order.

    \STATE \quad\textbf{2.1 Decision variables:}
      Introduce $x_{j\ell}$ for $\bQQ$ and
      $h_{\ell k}$ for $\HH$.

    \STATE \quad\textbf{2.2 Constraints:}
      Encode \eqref{eq:sat_cond} and the weakly non-increasing
      lexicographic ordering of the columns of $\bQQ$.

    \STATE \quad\textbf{2.3 Solve the SAT instance.}
    \IF{the solver returns \textsc{sat}}
      \STATE \textbf{return} Does not satisfy the algebraic condition
    \ELSE
      \STATE \textbf{return} Satisfies the algebraic condition
    \ENDIF
  \end{algorithmic}
\end{algorithm}

{\color{black}
To provide a sense of how restrictive the algebraic condition is and
to evaluate the computational efficiency of Algorithm~\ref{alg:id},
we conduct simulation studies in
Section~\ref{sec:simulation} of the Supplementary Material. The studies
show that many matrices satisfying the algebraic condition, and therefore identifiable, are not
complete and contain no pure nodes. 
For example, among 1,000 generated $50\times10$ matrices with
independent $\operatorname{Bernoulli}(0.5)$ entries, all satisfy
the algebraic condition, none are complete, and $62.3\%$
contain no pure nodes.
They also show that basis reduction
can substantially reduce the size of the SAT formulation and that the
submatrix checks can often determine whether the algebraic condition
holds before SAT verification is needed.
}

{\color{black}
\section{Identifiability for Probabilistic Boolean Models}
\label{sec:prob}
In this section, we present identifiability results for probabilistic
Boolean factor models. We adopt the following nondegeneracy convention.
For given $\ast$ and $\QQ$, let $\mca_j^+
=
\{\aaa\in\{0,1\}^K:\qq_j\ast\aaa=1\},$ and $\mca_j^-=\{0,1\}^K\setminus\mca_j^+.$
We assume that
$\varnothing\subsetneq\mca_j^+\subsetneq\{0,1\}^K$
for every $j\in[J]$, so that each item has both capable and
incapable latent vectors. Under the conjunctive operator,
this condition is equivalent to requiring $\qq_j\neq\zero$
for every $j\in[J]$.
% This convention is consistent with standard assumptions in the identifiability literature. For instance, under the conjunctive operator, since $\one\in\mathcal A_j^+$ always holds, the condition is equivalent to requiring $\qq_j\neq\zero$ for all $j\in[J]$. Existing results show that when a $\QQ$-matrix contains zero $\qq$-vectors, its identifiability conclusion remains unchanged after deleting those zero $\qq$-vectors \citep{gu21}.
For results stated specifically under the conjunctive operator, corresponding results for the disjunctive operator follow by duality \citep{chen2015statistical}; for simplicity, we state only the conjunctive versions.

\subsection{Identifiability for One- and Two-parameter Models}
\label{app:prob_12}

We first consider the one-parameter model in~\eqref{eq:one-param-model}, where \(\TTT=(g,p_{[\Cdot]})\). The following result gives a sufficient and necessary condition for identifiability of \(\QQ\) and joint identifiability of \((\QQ,\TTT)\).

 \begin{theorem}\label{thm:1P}
Consider the one-parameter Boolean factor model~\eqref{eq:one-param-model}
with \(J\geq2\) and \(p_{[\aaa]}>0\) for every
\([\aaa]\in\mathcal E(\QQ)\).
Then \(\QQ\) is identifiable if and only if
\((\QQ,\TTT)\) is jointly identifiable, which holds if and only if
\(\QQ\) satisfies the algebraic condition~\eqref{eq:a_condition}
and \(\lvert\mathcal E(\QQ)\rvert<2^J\).\\
 In the special case of the conjunctive operator, these conditions are equivalent
to \(\QQ\) satisfying the algebraic condition~\eqref{eq:a_condition}
and \(J>K\).
\end{theorem}

Since \(\lvert\mathcal E(\QQ)\rvert\leq2^K\), \(\lvert\mathcal E(\QQ)\rvert<2^J\) holds when
\(J>K\) for any  operator. Therefore, when $J>K$  adding the unknown item parameter \(g\) does not
require any condition on \(\QQ\) beyond the algebraic condition. We next provide an example to illustrate the use of Theorem~\ref{thm:1P}.
\begin{example}\label{ex:1P}
{\it
Consider the one-parameter Boolean factor model with the conjunctive
operator and \(\QQ=\QQ^{(2)}\) in Equation~\eqref{eq:ex_two_Q}.
Example~\ref{ex2} establishes that this matrix satisfies the
algebraic condition. Since \(J=6>K=4\), Theorem~\ref{thm:1P}
implies that \((\QQ,g,p_{[\Cdot]})\) is jointly identifiable
whenever \(p_{[\aaa]}>0\) for every
\([\aaa]\in\mathcal E(\QQ)\).
Thus introducing the unknown parameter \(g\) preserves the identifiability of \(\QQ\), while \(g\) and $p_{[\Cdot]}$ are also identifiable.
}
\end{example}

We next consider the two-parameter model in~\eqref{eq:two-param-model},
with \(\TTT=(s,g,p_{[\Cdot]})\). The following result gives
sufficient conditions for identifiability of \(\QQ\) and joint identifiability of \((\QQ,\TTT)\).

\begin{theorem}\label{thm:2P}
Consider the two-parameter Boolean factor model~\eqref{eq:two-param-model}
with \(p_{[\aaa]}>0\) for every
\([\aaa]\in\mathcal E(\QQ)\).
Suppose that \(\bigcap_{j=1}^J\mca_j^+\neq\varnothing\)
and \(\bigcap_{j=1}^J\mca_j^-\neq\varnothing\).
Then \(\QQ\) is identifiable if it satisfies the algebraic
condition~\eqref{eq:a_condition} and
\(2^J>2^K\lvert\mathcal E(\QQ)\rvert\).
%, where the inequality holds whenever \(J>2K\).
Under the same conditions, \((\QQ,\TTT)\) is jointly identifiable. \\
In the special case of the conjunctive operator, the above identifiability conditions hold if \(\QQ\) has no zero \(\qq\)-vectors,
satisfies the algebraic condition~\eqref{eq:a_condition},
and \(2^J>2^K\lvert\mathcal E(\QQ)\rvert\).
\end{theorem}
 Since \(\lvert\mathcal E(\QQ)\rvert\leq2^K\), \(2^J>2^K\lvert\mathcal E(\QQ)\rvert\) holds whenever \(J>2K\) for any operator. Conditions $\bigcap_{j=1}^J\mca_j^+\neq\varnothing$ and $\bigcap_{j=1}^J\mca_j^-\neq\varnothing$ mean that both \(\one_J\) and \(\zero_J\) occur in $\operatorname{Cols}\bigl(\PPhi(\QQ)\bigr)$. In particular, under the conjunctive operator, they hold whenever $\QQ$ has no zero $\qq$-vectors.
At the boundary \(J=2K\), the conjunctive operator admits the following refinement.
\begin{proposition}\label{prop:two-param-conjunctive}
Consider the two-parameter Boolean factor
model~\eqref{eq:two-param-model} under the conjunctive operator,
with \(p_{[\aaa]}>0\) for every
\([\aaa]\in\mathcal E(\QQ)\).
Suppose that \(J=2K\). 
Then \(\QQ\) is identifiable if and only if it satisfies the algebraic condition~\eqref{eq:a_condition}.
Furthermore, \((\QQ,\TTT)\) is jointly identifiable for every such \(\TTT\) if and only if \(\QQ\) satisfies the algebraic condition and  \(\QQ\) cannot be transformed into \(\left(\begin{smallmatrix}\cI_K\\ \cI_K\end{smallmatrix}\right)\) by permuting its rows and columns.
%after rearrangement of rows, \(\QQ\nsim
%\left(\begin{smallmatrix}\cI_K\\ \cI_K\end{smallmatrix}\right)\).
\end{proposition}

Indeed, if, after rearranging rows,
\(\QQ\sim\left(\begin{smallmatrix}\cI_K\\ \cI_K\end{smallmatrix}\right)\),
then \(\QQ\) remains identifiable, but there exist
parameters \(\TTT\) for which \((\QQ,\TTT)\) is not jointly
identifiable; see Section~\ref{app:proofs_12} of the Supplementary Material for a construction.

% The next result uses the minimum number of items whose
% \(\qq\)-vectors jointly cover all \(K\) attributes. Define $\rho(\QQ)
% =
% \min\left\{
% |S|:S\subseteq[J],\
% \bigvee_{j\in S}\qq_j=\one
% \right\}.$
% \begin{proposition}\label{prop:common-cover}
% Consider the two-parameter Boolean factor
% model~\eqref{eq:two-param-model} under the conjunctive operator,
% with \(p_{[\aaa]}>0\) for every
% \([\aaa]\in\mathcal E(\QQ)\).
% Suppose that \(\QQ\) satisfies the algebraic
% condition~\eqref{eq:a_condition} and \(J>K+\rho(\QQ)\).
% Then \((\QQ,\TTT)\) is jointly identifiable.
% \end{proposition}

\subsection{Identifiability for Multi-parameter Models}\label{app:prob_multi}

We consider the multi-parameter model in~\eqref{eq:multi-model},
with \(\TTT=(\mathbf s,\mathbf g,p_{[\Cdot]})\).
The following theorem
provides sufficient conditions for identifiability of
\(\QQ\) and joint identifiability of \((\QQ,\TTT)\).

\begin{theorem}
\label{thm:multi-general}
Consider the multi-parameter Boolean factor model~\eqref{eq:multi-model} with \(p_{[\aaa]}>0\) for all
\([\aaa]\in\mathcal E(\QQ)\).
$\QQ$ is identifiable if there exist $B \geq 2$ blocks
\(I_1,\ldots,I_B\subseteq[J]\) such that
\begin{enumerate}
\item[\textup{(i)}]
For every \(b\in[B]\), \(\QQ_{I_b}\) satisfies the algebraic
condition~\eqref{eq:a_condition}, and every class in
\(\mathcal E(\QQ_{I_b})\) contains a class in
\(\mathcal E(\QQ_{[J]\setminus I_b})\).
\item[\textup{(ii)}]
The blocks cover \([J]\) and can be ordered so that, for every
\(b\geq2\), the matrix
\(\QQ_{\,I_b\cap(I_1\cup\cdots\cup I_{b-1})}\)
has \(K\) pairwise distinct columns.
\end{enumerate}
In addition, \((\QQ,\TTT)\) is jointly identifiable if (i), (ii) and the following condition (iii) hold:
\begin{enumerate}
\item[\textup{(iii)}]
For every \(j\in[J]\), there exists a block \(I_b\) containing
\(j\) such that a class in
\(\mathcal E(\QQ_{I_b\setminus\{j\}})\) is contained in
\(\mca_j^+\), and there exists a block \(I_{b'}\) containing
\(j\) such that a class in
\(\mathcal E(\QQ_{I_{b'}\setminus\{j\}})\) is contained in
\(\mca_j^-\).
\end{enumerate}
\end{theorem}

To illustrate, suppose $I_1\cup I_2=[J]$ and $\QQ=
\bigl(
\QQ_{I_1\setminus I_2}^\top,\,
\QQ_{I_1\cap I_2}^\top,\,
\QQ_{I_2\setminus I_1}^\top
\bigr)^\top$ after a rearrangement of rows.
If $\QQ_{I_1}$ and $\QQ_{I_2}$ satisfy the algebraic
condition, $\mathcal E(\QQ_{I_1\setminus I_2})
=\mathcal E(\QQ_{I_2\setminus I_1})
=\mathcal E(\QQ)$, and $\QQ_{I_1\cap I_2}$ has $K$ pairwise distinct
columns, then $\QQ$ is identifiable. Joint identifiability
follows if \textup{(iii)} also holds. Notably, these conditions do not require pure nodes at all. 

Under the conjunctive operator, the conditions can be further simplified. The next theorem
requires the algebraic condition on one block and the
two-column condition (Definition~\ref{def:two-column}) on a disjoint block.

\begin{theorem}[Identifiability of the DINA model]\label{thm:multi-conj}
Consider model~\eqref{eq:multi-model} under the conjunctive operator with \(p_{[\aaa]}>0\) for all
\([\aaa]\in\mathcal E(\QQ)\) and $K \geq 2$. $\QQ$ is
identifiable if, after a rearrangement of rows, $\QQ=(\QQ_{I_1}^\top,\QQ_{I_2}^\top)^\top$,
where $\QQ_{I_1}$ satisfies the algebraic condition and $\QQ_{I_2}$ satisfies the two-column condition.
In addition, if $\sum_{j=1}^Jq_{jk}\geq3$ for all
$k\in[K]$, then
$(\QQ,\mathbf s,\mathbf g,p_{[\Cdot]})$ is jointly identifiable.
\end{theorem}

When $\QQ$ has no pure nodes, the block conditions in
Theorem~\ref{thm:multi-conj} imply
$\sum_{j=1}^J q_{jk}\geq 3$ for every $k\in[K]$,
so the additional condition for joint identifiability
is automatic.

\begin{remark}\label{rmk:compare}
{\it 
Existing DINA identifiability results  \citep{JLGXZY2011,chen2015statistical,xu2018,gu21} all require the completeness of $\QQ$, whereas our results do not require pure nodes. 
We point out that the necessity of completeness
established in these works concerns joint identifiability
of $(\QQ,\mathbf s,\mathbf g,\pp)$, rather than
identifiability of $\QQ$ itself.
In their necessity proofs, completeness is shown to be
necessary for identifying $\pp$.
Consequently, it is necessary for jointly
identifying $(\QQ,\mathbf s,\mathbf g,\pp)$, but this does not imply that an incomplete $\QQ$
is itself unidentifiable.
The two block conditions in
Theorem~\ref{thm:multi-conj} suffice for identifying
$\QQ$ without requiring the other model parameters
to be identifiable.
These conditions can hold even when $\QQ$ contains
no pure nodes, as illustrated in the following
Example~\ref{ex:multi-conjunctive}.
Furthermore, even for complete matrices, the joint
identifiability conditions in \citet{gu21} can be
stronger than necessary for identifying $\QQ$ only.
For example, our theorem establishes identifiability of $\QQ=\left(\begin{smallmatrix}\cI_K\\ \cI_K\end{smallmatrix}\right),$
although this matrix does not satisfy their requirement
of at least three ones per column. Under the same block conditions, our theorem further
shows that $(\mathbf s,\mathbf g,p_{[\Cdot]})$ is
identifiable jointly with $\QQ$ if every column of
$\QQ$ additionally contains at least three ones.
This joint identifiability conclusion also holds
without requiring pure nodes.
}
\end{remark}

The following example illustrates
Theorem~\ref{thm:multi-conj} with a $\QQ$-matrix
containing no pure nodes.

\begin{example}\label{ex:multi-conjunctive}
{\it
Consider a multi-parameter Boolean factor model
under the conjunctive operator with the following
$\QQ$-matrix:
\[
{\renewcommand{\arraystretch}{0.6}
\QQ=
\begin{pmatrix}
\QQ_{I_1}\\
\QQ_{I_2}
\end{pmatrix},
\qquad
\QQ_{I_1}=
\begin{pmatrix}
1&1&0&0\\
1&0&1&0\\
1&0&0&1\\
0&1&1&0\\
0&1&0&1\\
0&0&1&1
\end{pmatrix},
\qquad
\QQ_{I_2}=
\begin{pmatrix}
1&1&0&0\\
0&1&1&0\\
0&0&1&1\\
1&0&0&1
\end{pmatrix}.
}
\]
The completeness condition in \citet{gu21} fails for this
$\QQ$, and $(\QQ,\mathbf s,\mathbf g,\pp)$ is not jointly
identifiable. Indeed, latent vectors $\zero$ and $\ee_k$ for $k \in [4]$ induce identical conditional
response distributions, so their individual proportions
cannot be identified separately.
Nevertheless, by Example~\ref{ex2}, $\QQ_{I_1}$ satisfies
the algebraic condition. Since $\QQ_{I_2}$ satisfies the
two-column condition and every column of $\QQ$ contains
five ones, Theorem~\ref{thm:multi-conj} establishes joint
identifiability of $(\QQ,\mathbf s,\mathbf g,p_{[\Cdot]})$
whenever $p_{[\aaa]}>0$ for every
$[\aaa]\in\mathcal E(\QQ)$.
}
\end{example}
We finally consider the special case where $\mathbf g$ is known
\citep{JLGXZY2011,chen2015statistical,xu2016}.
This setting is practically relevant. For example, in educational
testing, items that are difficult to answer correctly by guessing
may motivate treating \(g_j=0\) as a modeling assumption
\citep{chen2015statistical}.
When \(\mathbf g\) is known, the conjunctive operator admits a simpler sufficient and necessary condition for identifiability of \(\QQ\) and joint identifiability.

\begin{proposition}\label{prop:multi-known-g-conjunctive}
Consider the multi-parameter Boolean factor model~\eqref{eq:multi-model}
under the conjunctive operator, with \(K\geq2\), \(\mathbf g\) known,
and \(p_{[\aaa]}>0\) for every \([\aaa]\in\mathcal E(\QQ)\).
Then \(\QQ\) is identifiable if and only if
\((\QQ,\mathbf s,p_{[\Cdot]})\) is jointly identifiable,
which holds if and only if \(\QQ\) satisfies the algebraic
condition~\eqref{eq:a_condition} and every column of \(\QQ\)
contains at least two ones.
\end{proposition}
% For $K \geq 2$, if $\QQ$ has no pure nodes, the algebraic condition guarantees at
% least two ones in every column and thus suffices for
% joint identifiability.
}
\section{Discussion}\label{sec: discussion}
In this paper, we study the identifiability of graphical structures in Boolean factor models. These models allow for flexible interactions between observed and latent variables, encompassing widely used models such as cognitive diagnosis models and Boolean matrix decomposition, which are popular in various disciplines.
Our proposed sufficient and necessary conditions are novel in that they do not rely on the conventional and stringent pure-node assumption. In particular, our graphical condition uses a new tool based on Hasse diagrams to transform the identifiability problem into a uniqueness problem for the graphical representation, helping to explain why pure nodes are not necessary for identifiability.
%Our proposed sufficient and necessary conditions are novel in a sense that they do not rely on the conventional stringent assumption of the pure nodes. Particularly, our graphical condition leverages a novel graphical tool we develop which utilizes Hasse diagram and transforms the identifiability problem into the uniqueness of the graphical representation, aiding in understanding why pure nodes are not necessary for identifiability. 

% As indicated in Section~\ref{},
% while there is extensive literature 
There are several interesting future directions building upon this work.
First, although our primary focus in this work is on identifying the graphical structure linking the observed variables \(\RR\) and those latent variables \(\aaa\) that are direct parents to \(\RR\), the techniques and identifiability conditions could naturally extend to more general latent–latent connections. For instance, in models with hierarchical structure of latent variables (e.g., Figure~\ref{fig:models}(b)), one could apply similar arguments by treating each pair of consecutive latent layers as a local Boolean factor model. This extension 
would allow researchers to infer the graphical structures among deeper latent layers, potentially revealing complex dependencies 
that generate \(\aaa\), thereby enhancing both interpretability and 
inference accuracy in rich graphical models.
{\color{black}
Second, extending our results to more general models,
including LCDM, GDINA, and GDM, is an important direction.
The present paper studies strict identifiability, whereas
generic identifiability permits non-identifiability on
a measure-zero subset of the continuous parameter space
\citep{chen2020sparse,gu2020partial}.
In these more general models, additional main effects
and interactions may distinguish latent vectors with
identical columns of $\PPhi(\QQ)$.
Our graphical and algebraic characterizations may help
construct parameter configurations satisfying the rank
conditions in existing generic-identifiability arguments.
For example, under GDINA, showing that the relevant
polynomial determinants are nonzero at one configuration
establishes the corresponding rank conditions generically.
Establishing these rank conditions and identifying the
unknown $\QQ$ under the more general model assumptions
require further arguments beyond our strict
identifiability results.
Finally, in this paper, \(K\) is treated as fixed. When $K$ is unknown, identifying \(K\) is also an important problem. Existing results which consider unknown $K$, such as \citet{Bing2020}, rely on pure variable anchoring and additional model structure. It would be interesting to extend our framework to unknown-$K$ settings and investigate whether identifiability without pure node assumptions is possible under suitable restrictions on the possible latent representations. 
% We leave this direction for future work.
}

%Additionally, the identifiability of models for polytomous latent attributes is increasingly significant as the availability of general categorical data grows \citep{de2011generalized,GPDINA}. While existing studies have examined the identifiability of polytomous models \citep{lin2024polytomousDINA}, they often rely on the assumption of pure nodes. Extending our identifiability conditions to polytomous data models without such assumptions is a promising direction for future work.

% For instance, by incorporating weighted edges into the graphical representations, it becomes possible to capture richer, non-deterministic distributions. Additionally, reformulating the identifiability problem as a general optimization problem could enable it to be addressed using convex or non-convex feasibility frameworks, significantly broadening the scope and applicability of our methodological approach.

%%%%%%%%%%%%%%%%%%%%%%%%%%%%%%%%%%%%%%%%%%%%%%%%%%

%%%%%%%%%%%%%%%%%%%%%%%%%%%%%%%%%%%%%%%%%%%%%%%%%%
%% References
%%%%%%%%%%%%%%%%%%%%%%%%%%%%%%%%%%%%%%%%%%%%%%%%%%

\begingroup
\setlength{\bibsep}{0pt}

\bibliography{ref}
\endgroup

% =====================================================================
% Supplementary material (included in this file; no external .aux files).
% =====================================================================
\clearpage
\newgeometry{margin=1in}
\setlength{\footskip}{24pt}
\doublespacing
\setlength{\parindent}{1.5em}
\setlength{\parskip}{0pt plus 1pt}
\setcounter{secnumdepth}{3}
\let\paragraph\oldparagraph
\let\subparagraph\oldsubparagraph

% Preserve the original S-numbering and give hyperlinks unique destinations.
\setcounter{section}{0}
\renewcommand{\thesection}{S\arabic{section}}
\expandafter\def\csname theHsection\endcsname{supp.\arabic{section}}
\setcounter{figure}{0}
\renewcommand{\thefigure}{S\arabic{figure}}
\expandafter\def\csname theHfigure\endcsname{supp.\arabic{figure}}
\setcounter{table}{0}
\renewcommand{\thetable}{S\arabic{table}}
\expandafter\def\csname theHtable\endcsname{supp.\arabic{table}}
\setcounter{equation}{0}
\renewcommand{\theequation}{S\arabic{equation}}
\expandafter\def\csname theHequation\endcsname{supp.\arabic{equation}}
\setcounter{theorem}{0}
\renewcommand{\thetheorem}{S\arabic{theorem}}
\expandafter\def\csname theHtheorem\endcsname{supp.\arabic{theorem}}
\setcounter{lemma}{0}
\renewcommand{\thelemma}{S\arabic{lemma}}
\expandafter\def\csname theHlemma\endcsname{supp.\arabic{lemma}}
\setcounter{proposition}{0}
\renewcommand{\theproposition}{S\arabic{proposition}}
\expandafter\def\csname theHproposition\endcsname{supp.\arabic{proposition}}
\setcounter{corollary}{0}
\renewcommand{\thecorollary}{S\arabic{corollary}}
\expandafter\def\csname theHcorollary\endcsname{supp.\arabic{corollary}}
\setcounter{definition}{0}
\renewcommand{\thedefinition}{S\arabic{definition}}
\expandafter\def\csname theHdefinition\endcsname{supp.\arabic{definition}}
\setcounter{assumption}{0}
\renewcommand{\theassumption}{S\arabic{assumption}}
\expandafter\def\csname theHassumption\endcsname{supp.\arabic{assumption}}
\setcounter{condition}{0}
\renewcommand{\thecondition}{S\arabic{condition}}
\expandafter\def\csname theHcondition\endcsname{supp.\arabic{condition}}
\setcounter{example}{0}
\renewcommand{\theexample}{S\arabic{example}}
\expandafter\def\csname theHexample\endcsname{supp.\arabic{example}}
\setcounter{notation}{0}
\renewcommand{\thenotation}{S\arabic{notation}}
\expandafter\def\csname theHnotation\endcsname{supp.\arabic{notation}}
\setcounter{problem}{0}
\renewcommand{\theproblem}{S\arabic{problem}}
\expandafter\def\csname theHproblem\endcsname{supp.\arabic{problem}}
\setcounter{remark}{0}
\renewcommand{\theremark}{S\arabic{remark}}
\expandafter\def\csname theHremark\endcsname{supp.\arabic{remark}}
\setcounter{claim}{0}
\renewcommand{\theclaim}{S\arabic{claim}}
\expandafter\def\csname theHclaim\endcsname{supp.\arabic{claim}}
\setcounter{fact}{0}
\renewcommand{\thefact}{S\arabic{fact}}
\expandafter\def\csname theHfact\endcsname{supp.\arabic{fact}}
\setcounter{conjecture}{0}
\renewcommand{\theconjecture}{S\arabic{conjecture}}
\expandafter\def\csname theHconjecture\endcsname{supp.\arabic{conjecture}}
\setcounter{algorithm}{0}
\renewcommand{\thealgorithm}{S\arabic{algorithm}}
\expandafter\def\csname theHalgorithm\endcsname{supp.\arabic{algorithm}}

\renewcommand{\thesubsection}{\thesection.\arabic{subsection}}
\renewcommand{\thesubsubsection}{\thesubsection.\arabic{subsubsection}}
\let\fact\supplementfact
\let\endfact\endsupplementfact

\pdfbookmark[0]{Supplementary Material}{supplementary-material}
\begin{center}
\setstretch{1.1}
{\Large\bfseries Supplementary Material for\\[0.5em]
``Characterizing Identifiability in Boolean Factor Models''\par}
\end{center}
\bigskip

%——————————————————————————————————————————
% Body of Supplement
%——————————————————————————————————————————

This supplementary material presents simulation studies,
additional explanations of the graphical construction and
identifiability conditions, details of SAT implementation, and proofs of all theoretical results.
It is organized as follows:
\begin{itemize}
    \item Section~\ref{sec:simulation} presents simulation
    studies examining the identifiability conditions and
    the computational performance of Algorithm~\ref{alg:id}.

    \item Section~\ref{app:more_graph} provides a guide to
    the graphical construction and further interpretation
    of the identifiability results. It also details the
    SAT formulation and implementation, analyzes
    computational complexity, and explains how SAT
    enumeration can quantify non-identifiability.

    \item Sections~\ref{app:secgraph}
    and~\ref{app:boolean_proof} contain proofs of the
    graphical and algebraic characterizations, the
    structural reductions, and the SAT formulation.

    \item Section~\ref{app:prob} contains proofs of the
    identifiability results for the one-, two-, and
    multi-parameter Boolean factor models presented
    in Section~\ref{sec:prob}.
\end{itemize}

\newpage
{\color{black}
\section{Simulation Studies}
\label{sec:simulation}
We conduct simulations under the conjunctive operator to examine how
frequently randomly generated $\QQ$-matrices satisfy the algebraic
condition~\eqref{eq:a_condition} and to evaluate the computational
performance of Algorithm~\ref{alg:id}. The first study compares the
algebraic condition with completeness and the necessary two- and
three-column conditions, with particular attention to identifiable
matrices without pure nodes. The second study examines basis reduction,
the submatrix checks, and runtime.

\subsection{Empirical behavior of the identifiability conditions}
\label{sec:sim_conditions}

\paragraph{Simulation designs.}
We consider two ways of generating $\QQ$-matrices. In the
\emph{Bernoulli design}, the entries are independent with
$q_{jk}\sim\operatorname{Bernoulli}(p)$, where
\[
J\in\{25,50,100\},\qquad
K\in\{5,10\},\qquad
p\in\{0.1,0.3,0.5,0.7,0.9\}.
\]
In the \emph{sparsity design}, each item involves at most $m$
attributes. Independently for each item $j$, we draw
$D_j=\|\qq_j\|_0$ uniformly from $\{1,\ldots,m\}$ and select its
$D_j$ nonzero coordinates uniformly without replacement from the $K$
attributes. We consider $J\in\{25,50,100\}$, $K=10$, and
$m\in\{3,4\}$. This design reflects diagnostic settings in which
each item involves only a few attributes.
For each setting under either design, we independently generate
$1{,}000$ matrices and apply Algorithm~\ref{alg:id}.

\paragraph{Quantities of interest.}
Let $I_{\mathrm{id}}(\QQ)$ indicate whether $\QQ$ satisfies the
algebraic condition~\eqref{eq:a_condition}. By
Theorem~\ref{thm:algebraic}, this is equivalent to identifiability of
$\QQ$ when $p_{[\aaa]}>0$ for every
$[\aaa]\in\mathcal E(\QQ)$. We use ``identifiable'' in this sense
throughout this section. Because the algebraic condition depends only
on $\QQ$, no response data or latent proportions need to be generated.

We compare $I_{\mathrm{id}}$ with three indicators corresponding to
the submatrix checks in Algorithm~\ref{alg:id}. Let
$I_{\mathrm{comp}}(\QQ)$ indicate that $\QQ$ is complete, and let
$I_{\mathrm{2c}}(\QQ)$ indicate that every two-column submatrix
contains both row patterns $(1,0)$ and $(0,1)$. Finally, let
$I_{\mathrm{3c}}(\QQ)$ indicate that $I_{\mathrm{2c}}(\QQ)=1$ and
every three-column submatrix contains a row with exactly one nonzero
entry among the three selected columns. Thus, $I_{\mathrm{3c}}$
records whether the matrix passes both necessary checks. By
Corollary~\ref{cor:pure-node}, Theorem~\ref{thm:sub}, and
Lemmas~\ref{lm:onetwo} and~\ref{lm:threefour},
\begin{equation}
\label{eq:indicator-hierarchy}
I_{\mathrm{comp}}(\QQ)
\leq I_{\mathrm{id}}(\QQ)
\leq I_{\mathrm{3c}}(\QQ)
\leq I_{\mathrm{2c}}(\QQ).
\end{equation}
We estimate the probability that each indicator equals one by its
empirical proportion, denoted by $\widehat{\P}$.

To assess how often identifiability holds beyond completeness, we also
report $\widehat{\P}(I_{\mathrm{comp}}=0\mid I_{\mathrm{id}}=1)$.
Under the Bernoulli design, we further report
$\widehat{\P}(I_{\mathrm{np}}=1\mid I_{\mathrm{id}}=1)$, where
$I_{\mathrm{np}}(\QQ)$ indicates that $\QQ$ contains no pure nodes,
that is, no row with exactly one nonzero entry. Conditional proportions
are computed among the identifiable matrices and are left undefined
when none of the generated matrices is identifiable. All proportions
are reported as percentages.

\paragraph{Bernoulli design results.}
Table~\ref{tab:bern} shows that the algebraic condition holds frequently
over a broad range of the settings considered. For
$p\in\{0.3,0.5\}$ and $J\geq50$, at least $99.8\%$ of the generated
matrices are identifiable. For fixed $K$ and $p$, the empirical
proportion does not decrease with $J$. This agrees with
Theorem~\ref{thm:sub}\textup{(ii)}: adding rows to an identifiable
matrix preserves identifiability. For fixed $J$ and $p$, the proportion
for $K=10$ is no larger than that for $K=5$ in any setting considered.
For example, at $(J,p)=(25,0.7)$, it decreases from $79.9\%$ to
$5.8\%$. Increasing $J$ can offset this difference: at $(J,p)=(100,0.7)$,
all generated matrices are identifiable for both values of $K$.

Identifiability is less frequent at the extreme values of $p$,
especially at $p=0.9$. The necessary conditions help explain this
pattern. For a given pair of columns, each of the patterns $(1,0)$
and $(0,1)$ has probability $p(1-p)$ in any row. Thus, the probability
of satisfying the two-column condition is symmetric in $p$ and $1-p$.
For a given triple, however, the probability that a row has exactly
one nonzero entry is $3p(1-p)^2$, which is much smaller at $p=0.9$
than at $p=0.1$. For $(J,K)=(100,10)$, $99.5\%$ of the generated
matrices pass the two-column check at both values of $p$, whereas
$99.5\%$ pass both necessary checks at $p=0.1$ and only $3.3\%$
do so at $p=0.9$.

The necessary checks do not always determine identifiability.
At $(J,K,p)=(25,10,0.7)$, $56.5\%$ of the generated matrices pass
both checks, but only $5.8\%$ satisfy the algebraic condition.
Thus, SAT verification remains necessary for matrices that pass these
checks and are not complete.

\begin{table}[htbp]
\centering

\begin{subtable}[t]{0.48\textwidth}
\centering
\resizebox{\linewidth}{!}{%
\begin{tabular}{c cc cc cc}
\toprule
& \multicolumn{2}{c}{\(J=25\)}
& \multicolumn{2}{c}{\(J=50\)}
& \multicolumn{2}{c}{\(J=100\)}\\
\cmidrule(lr){2-3}\cmidrule(lr){4-5}\cmidrule(lr){6-7}
\(p\)
& \(K=5\) & \(K=10\)
& \(K=5\) & \(K=10\)
& \(K=5\) & \(K=10\)\\
\midrule
0.1 & 44.6 & 6.1 & 90.0 & 71.6 & 99.6 & 99.5\\
0.3 & 94.2 & 83.9 & 100.0 & 99.8 & 100.0 & 100.0\\
0.5 & 98.0 & 93.2 & 100.0 & 100.0 & 100.0 & 100.0\\
0.7 & 79.9 & 5.8 & 99.8 & 93.5 & 100.0 & 100.0\\
0.9 & 0.1 & 0.0 & 1.1 & 0.0 & 20.4 & 0.0\\
\bottomrule
\end{tabular}%
}
\caption{\(\widehat{\P}(I_{\mathrm{id}}=1)\).}
\label{tab:bern-id}
\end{subtable}
\hfill
\begin{subtable}[t]{0.48\textwidth}
\centering
\resizebox{\linewidth}{!}{%
\begin{tabular}{c cc cc cc}
\toprule
& \multicolumn{2}{c}{\(J=25\)}
& \multicolumn{2}{c}{\(J=50\)}
& \multicolumn{2}{c}{\(J=100\)}\\
\cmidrule(lr){2-3}\cmidrule(lr){4-5}\cmidrule(lr){6-7}
\(p\)
& \(K=5\) & \(K=10\)
& \(K=5\) & \(K=10\)
& \(K=5\) & \(K=10\)\\
\midrule
0.1 & 37.3 & 0.3 & 83.9 & 19.5 & 99.2 & 81.7\\
0.3 & 40.2 & 0.0 & 88.6 & 0.0 & 99.8 & 2.5\\
0.5 & 3.4 & 0.0 & 28.3 & 0.0 & 81.1 & 0.0\\
0.7 & 0.0 & 0.0 & 0.0 & 0.0 & 0.9 & 0.0\\
0.9 & 0.0 & 0.0 & 0.0 & 0.0 & 0.0 & 0.0\\
\bottomrule
\end{tabular}%
}
\caption{\(\widehat{\P}(I_{\mathrm{comp}}=1)\).}
\end{subtable}

\vspace{0.8em}

\begin{subtable}[t]{0.48\textwidth}
\centering
\resizebox{\linewidth}{!}{%
\begin{tabular}{c cc cc cc}
\toprule
& \multicolumn{2}{c}{\(J=25\)}
& \multicolumn{2}{c}{\(J=50\)}
& \multicolumn{2}{c}{\(J=100\)}\\
\cmidrule(lr){2-3}\cmidrule(lr){4-5}\cmidrule(lr){6-7}
\(p\)
& \(K=5\) & \(K=10\)
& \(K=5\) & \(K=10\)
& \(K=5\) & \(K=10\)\\
\midrule
0.1 & 44.6 & 6.3 & 90.0 & 71.6 & 99.6 & 99.5\\
0.3 & 94.2 & 83.9 & 100.0 & 99.8 & 100.0 & 100.0\\
0.5 & 98.0 & 93.8 & 100.0 & 100.0 & 100.0 & 100.0\\
0.7 & 95.6 & 82.6 & 100.0 & 100.0 & 100.0 & 100.0\\
0.9 & 43.5 & 6.0 & 92.5 & 73.0 & 99.9 & 99.5\\
\bottomrule
\end{tabular}%
}
\caption{\(\widehat{\P}(I_{\mathrm{2c}}=1)\).}
\label{tab:bern-2c}
\end{subtable}
\hfill
\begin{subtable}[t]{0.48\textwidth}
\centering
\resizebox{\linewidth}{!}{%
\begin{tabular}{c cc cc cc}
\toprule
& \multicolumn{2}{c}{\(J=25\)}
& \multicolumn{2}{c}{\(J=50\)}
& \multicolumn{2}{c}{\(J=100\)}\\
\cmidrule(lr){2-3}\cmidrule(lr){4-5}\cmidrule(lr){6-7}
\(p\)
& \(K=5\) & \(K=10\)
& \(K=5\) & \(K=10\)
& \(K=5\) & \(K=10\)\\
\midrule
0.1 & 44.6 & 6.1 & 90.0 & 71.6 & 99.6 & 99.5\\
0.3 & 94.2 & 83.9 & 100.0 & 99.8 & 100.0 & 100.0\\
0.5 & 98.0 & 93.7 & 100.0 & 100.0 & 100.0 & 100.0\\
0.7 & 91.4 & 56.5 & 99.9 & 99.7 & 100.0 & 100.0\\
0.9 & 1.7 & 0.0 & 13.2 & 0.0 & 60.5 & 3.3\\
\bottomrule
\end{tabular}%
}
\caption{\(\widehat{\P}(I_{\mathrm{3c}}=1)\).}
\label{tab:bern-3c}
\end{subtable}

\caption{Empirical proportions under the Bernoulli design,
expressed as percentages. Panels \textup{(a)}--\textup{(d)} report,
respectively, the proportions of generated \(\QQ\)-matrices that are
identifiable, complete, satisfy the two-column condition, and pass both necessary checks.}
\label{tab:bern}
\end{table}

Table~\ref{tab:bern-pure-node-gap} shows that many identifiable
matrices are not complete and, in several settings, contain no pure
nodes. At $(J,K,p)=(50,10,0.5)$, all generated matrices are
identifiable, none are complete, and $62.3\%$ contain no pure nodes.
At $(J,K,p)=(50,10,0.7)$, $93.5\%$ of the generated matrices are
identifiable; among these, none are complete and $99.5\%$ contain no
pure nodes. Thus, in these settings, the absence of pure nodes is
common among identifiable matrices. In contrast, at $p=0.1$, every
identifiable matrix generated contains at least one pure node,
although many are not complete.

\begin{table}[htbp]
\centering

\begin{subtable}[t]{0.48\textwidth}
\centering
\resizebox{\linewidth}{!}{%
\begin{tabular}{c cc cc cc}
\toprule
& \multicolumn{2}{c}{\(J=25\)}
& \multicolumn{2}{c}{\(J=50\)}
& \multicolumn{2}{c}{\(J=100\)}\\
\cmidrule(lr){2-3}\cmidrule(lr){4-5}\cmidrule(lr){6-7}
\(p\)
& \(K=5\) & \(K=10\)
& \(K=5\) & \(K=10\)
& \(K=5\) & \(K=10\)\\
\midrule
0.1 & 16.4 & 95.1 & 6.8 & 72.8 & 0.4 & 17.9\\
0.3 & 57.3 & 100.0 & 11.4 & 100.0 & 0.2 & 97.5\\
0.5 & 96.5 & 100.0 & 71.7 & 100.0 & 18.9 & 100.0\\
0.7 & 100.0 & 100.0 & 100.0 & 100.0 & 99.1 & 100.0\\
0.9 & 100.0 & -- & 100.0 & -- & 100.0 & --\\
\bottomrule
\end{tabular}%
}
\caption{\(\widehat{\P}(I_{\mathrm{comp}}=0\mid
I_{\mathrm{id}}=1)\).}
\end{subtable}
\hfill
\begin{subtable}[t]{0.48\textwidth}
\centering
\resizebox{\linewidth}{!}{%
\begin{tabular}{c cc cc cc}
\toprule
& \multicolumn{2}{c}{\(J=25\)}
& \multicolumn{2}{c}{\(J=50\)}
& \multicolumn{2}{c}{\(J=100\)}\\
\cmidrule(lr){2-3}\cmidrule(lr){4-5}\cmidrule(lr){6-7}
\(p\)
& \(K=5\) & \(K=10\)
& \(K=5\) & \(K=10\)
& \(K=5\) & \(K=10\)\\
\midrule
0.1 & 0.0 & 0.0 & 0.0 & 0.0 & 0.0 & 0.0\\
0.3 & 0.0 & 3.5 & 0.0 & 0.1 & 0.0 & 0.0\\
0.5 & 1.1 & 80.6 & 0.0 & 62.3 & 0.0 & 36.4\\
0.7 & 46.6 & 100.0 & 23.3 & 99.5 & 6.7 & 98.6\\
0.9 & 100.0 & -- & 90.9 & -- & 86.3 & --\\
\bottomrule
\end{tabular}%
}
\caption{\(\widehat{\P}(I_{\mathrm{np}}=1\mid
I_{\mathrm{id}}=1)\).}
\label{tab:bern-np-given-id}
\end{subtable}

\caption{Conditional empirical proportions among identifiable
\(\QQ\)-matrices under the Bernoulli design, expressed as percentages.
Panel \textup{(a)} reports the proportion of identifiable matrices that
are not complete, and panel \textup{(b)} reports the proportion of
identifiable matrices that contain no pure nodes. A dash indicates that
no generated matrix in the corresponding configuration is identifiable.
For $(J,K,p)=(25,5,0.9)$, the conditional percentages are based on
one identifiable matrix.}
\label{tab:bern-pure-node-gap}
\end{table}

\paragraph{Sparsity design results.}
Table~\ref{tab:sparse} shows that identifiability also holds frequently
when each item involves only a few attributes. At $J=25$, the
proportions of identifiable matrices are $65.4\%$ for $m=3$ and
$74.1\%$ for $m=4$; at least $99.8\%$ of these identifiable matrices
are not complete. At $J=50$, the proportions of identifiable matrices
increase to $99.2\%$ and $99.7\%$, respectively, while $90.8\%$ and
$97.5\%$ of the identifiable matrices remain incomplete. At $J=100$,
all generated matrices are identifiable, and the corresponding
proportions that are incomplete are $27.3\%$ and $55.0\%$.

In all six settings, the proportions passing the two-column check,
passing both necessary checks, and satisfying the algebraic condition
coincide. Hence, every generated matrix that passes the two-column
check is identifiable in this study. This empirical agreement is
specific to these simulations; as the Bernoulli results illustrate,
passing the necessary checks does not generally establish
identifiability.

\begin{table}[htbp]
\centering

\begin{subtable}[t]{0.48\textwidth}
\centering
\begin{tabular}{c cc}
\toprule
\(J\) & \(m=3\) & \(m=4\)\\
\midrule
25 & 65.4 & 74.1\\
50 & 99.2 & 99.7\\
100 & 100.0 & 100.0\\
\bottomrule
\end{tabular}
\caption{\(\widehat{\P}(I_{\mathrm{id}}=1)\).}
\end{subtable}
\hfill
\begin{subtable}[t]{0.48\textwidth}
\centering
\begin{tabular}{c cc}
\toprule
\(J\) & \(m=3\) & \(m=4\)\\
\midrule
25 & 99.8 & 100.0\\
50 & 90.8 & 97.5\\
100 & 27.3 & 55.0\\
\bottomrule
\end{tabular}
\caption{\(\widehat{\P}(I_{\mathrm{comp}}=0\mid I_{\mathrm{id}}=1)\).}
\end{subtable}

\vspace{0.8em}

\begin{subtable}[t]{0.48\textwidth}
\centering
\begin{tabular}{c cc}
\toprule
\(J\) & \(m=3\) & \(m=4\)\\
\midrule
25 & 65.4 & 74.1\\
50 & 99.2 & 99.7\\
100 & 100.0 & 100.0\\
\bottomrule
\end{tabular}
\caption{\(\widehat{\P}(I_{\mathrm{2c}}=1)\).}
\end{subtable}
\hfill
\begin{subtable}[t]{0.48\textwidth}
\centering
\begin{tabular}{c cc}
\toprule
\(J\) & \(m=3\) & \(m=4\)\\
\midrule
25 & 65.4 & 74.1\\
50 & 99.2 & 99.7\\
100 & 100.0 & 100.0\\
\bottomrule
\end{tabular}
\caption{\(\widehat{\P}(I_{\mathrm{3c}}=1)\).}
\end{subtable}

\caption{Empirical proportions under the sparsity design, expressed as
percentages. Panel \textup{(a)} reports identifiability, panel
\textup{(b)} reports noncompleteness conditional on identifiability,
and panels \textup{(c)} and~\textup{(d)} report the proportions
passing the two-column check and both necessary checks, respectively.}
\label{tab:sparse}
\end{table}

\subsection{Computational performance of Algorithm~\ref{alg:id}}
\label{sec:sim_preprocessing}
We examine the effects of basis reduction and the submatrix checks
under the Bernoulli design with $K=10$,
$J\in\{25,50,100\}$, and $p\in\{0.1,0.3,0.5,0.7,0.9\}$.
Table~\ref{tab:bern-computation} reports the mean number $J_b$ of
rows in $\Qbasis$, the percentage of matrices requiring SAT
verification, and the mean runtimes for Steps~0--1 and for the full
algorithm. Each summary is computed over $1{,}000$ generated matrices,
including those resolved before SAT verification.
Computations were performed on a computing cluster using Intel Xeon Gold 6154 processors (3.00\,GHz), with one CPU
core per task. We used \texttt{idQ} version~0.1.3 with Python~3.11.7
and PySAT~1.9.dev15, using \texttt{Glucose42} as the SAT solver.
Runtimes are elapsed wall-clock times. The full algorithm time includes
basis reduction, the submatrix checks, CNF construction, SAT solving,
and verification of returned counterexamples; it excludes matrix
generation and file writing.

\begin{table}[htbp]
\centering
\small
\begin{tabular}{ccrrrr}
\toprule
\(J\) & \(p\)
& Mean \(J_b\)
& SAT invoked (\%)
& Steps~0--1 runtime
& Total runtime\\
\midrule
25 & 0.1 & 10.3 & 5.8 & \(<0.001\) & \(<0.001\)\\
25 & 0.3 & 20.0 & 83.9 & 0.002 & 0.009\\
25 & 0.5 & 21.6 & 93.7 & 0.002 & 0.040\\
25 & 0.7 & 18.2 & 56.5 & 0.002 & 0.210\\
25 & 0.9 & 10.2 & 0.0 & \(<0.001\) & \(<0.001\)\\
\addlinespace
50 & 0.1 & 11.8 & 52.1 & 0.001 & 0.004\\
50 & 0.3 & 27.7 & 99.8 & 0.003 & 0.013\\
50 & 0.5 & 35.3 & 100.0 & 0.004 & 0.026\\
50 & 0.7 & 29.4 & 99.7 & 0.003 & 2.166\\
50 & 0.9 & 14.1 & 0.0 & 0.001 & 0.001\\
\addlinespace
100 & 0.1 & 10.6 & 17.8 & \(<0.001\) & 0.002\\
100 & 0.3 & 25.3 & 97.5 & 0.004 & 0.012\\
100 & 0.5 & 49.5 & 100.0 & 0.006 & 0.026\\
100 & 0.7 & 44.6 & 100.0 & 0.005 & 0.548\\
100 & 0.9 & 21.7 & 3.3 & 0.002 & 0.003\\
\bottomrule
\end{tabular}\caption{Computational performance under the Bernoulli design for
\(K=10\). ``SAT invoked'' is the percentage of matrices passed to the
SAT verification. The runtimes for Steps~0--1 and for the full
algorithm are reported as mean seconds per matrix, averaged over all $1{,}000$ matrices in each setting.}\label{tab:bern-computation}
\end{table}

\paragraph{Basis reduction.}
By Theorem~\ref{thm:sub}\textup{(i)}, replacing $\QQ$ by $\Qbasis$
preserves the algebraic condition. This reduces the number of rows
entering the Boolean-factorization encoding from $J$ to $J_b$;
the resulting encoding has size $O(J_bK^2)$, as detailed in
Section~\ref{app:sub_comp_comx}. The reduction is substantial in several
settings. At $(J,p)=(100,0.3)$, the mean $J_b$ is $25.3$, and at
$(J,p)=(100,0.7)$ it is $44.6$. SAT verification is required for
$97.5\%$ and $100.0\%$ of the matrices in these settings,
respectively, so the reduction applies to nearly all subsequent SAT
instances. These results demonstrate a reduction in formulation size;
they do not measure the runtime improvement relative to solving
without basis reduction.

\paragraph{Submatrix checks.}
Algorithm~\ref{alg:id} invokes SAT precisely when a matrix is not
complete and passes both necessary checks. By
\eqref{eq:indicator-hierarchy}, the empirical proportion requiring
SAT is therefore
$\widehat{\P}(I_{\mathrm{3c}}=1)
-\widehat{\P}(I_{\mathrm{comp}}=1)$.
The contribution of each check varies across settings.
At $(J,p)=(100,0.1)$, completeness establishes identifiability for
$81.7\%$ of the generated matrices, while the two-column check rules
it out for $0.5\%$, leaving $17.8\%$ for SAT verification.
At $(J,p)=(25,0.9)$, the two-column check rules out identifiability
for $94.0\%$ of the matrices, and the three-column check rules it out
for the remaining $6.0\%$, so SAT is never invoked.
At $(J,p)=(100,0.9)$, the two-column check rules out identifiability
for only $0.5\%$, but the three-column check does so for another
$96.2\%$, leaving $3.3\%$ for SAT verification.
By contrast, at $p\in\{0.3,0.5\}$ and $J\geq50$, at least $97.5\%$
of the matrices require SAT. The submatrix checks can therefore
resolve most matrices in some settings, while SAT remains essential
in others.

\paragraph{Runtime.}
The reported mean runtime for Steps~0--1 is below $0.007$ seconds
per matrix in every setting. The largest mean total runtime is
$2.166$ seconds at $(J,p)=(50,0.7)$; all other reported means are
at most $0.548$ seconds. Runtime does not increase monotonically
with the matrix dimensions. For example, at $p=0.7$, the mean total
runtime decreases from $2.166$ seconds for $J=50$ to $0.548$ seconds
for $J=100$, even though the mean $J_b$ and the proportion requiring
SAT both increase.

}

{\color{black}
\section{More results of Section~\ref{sec:graph}}\label{app:more_graph}
\subsection{A guide to the graphical construction}
\label{app:graphical-guide}

The graphical construction involves several intermediate objects.
To help readers keep track of them, Table~\ref{tab:graphical-objects}
collects the main notation, objects, and definitions, whereas
Figure~\ref{fig:Q-to-G} provides a schematic overview of how $G(\QQ)$ is constructed. 

\begin{table}[htbp]
\centering
\small
\renewcommand{\arraystretch}{1.18}
\setlength{\tabcolsep}{4pt}
\begin{tabular}{
    |p{0.16\textwidth}
     |p{0.24\textwidth}
     |p{0.51\textwidth}|}
\hline
\textbf{Notation} & \textbf{Object} & \textbf{Definition} \\
\hline

\(\ast\)
&
Boolean operator
&
A map from
\(\{0,1\}^K\times\{0,1\}^K\) to \(\{0,1\}\).
\\
\hline

\(\QQ\)
&
\(\QQ\)-matrix
&
The graphical structure of the model.
\\
\hline

\(\preceq^\ast\)
&
Induced partial order
&
\(\aaa\preceq^\ast\aaa'\) if
\(\qq\ast\aaa\leq\qq\ast\aaa'\) for every
\(\qq\in\{0,1\}^K\).
\\
\hline

\(H(\{0,1\}^K,\preceq^\ast)\)
&
Hasse diagram
&
The Hasse diagram of the latent vector space
\(\{0,1\}^K\) under the partial order \(\preceq^\ast\).
\\
\hline

\(v(\aaa)\)
&
latent vector node
&
The node in \(H(\{0,1\}^K,\preceq^\ast)\) corresponding to
\(\aaa\in\{0,1\}^K\).
\\
\hline

\(\mca_j^+\)
&
Capable class
&
\(\mca_j^+=\{\aaa\in\{0,1\}^K:\qq_j\ast\aaa=1\}\).
\\
\hline

\(v_j^*(\qq_j)\)
&
$\qq_j$-vector node
&
The node that connects to the minimal latent vector
nodes in \(\mca_j^+\).
\\
\hline

\(\mch(\QQ)\)
&
Extended Hasse diagram
&
The graph obtained by adding the nodes
\(v_1^*(\qq_1),\ldots,v_J^*(\qq_J)\) and their associated edges to
\(H(\{0,1\}^K,\preceq^\ast)\).
\\
\hline

\(\Psi(v(\aaa))\)
&
Path status vector
&
Its \(j\)th entry equals one if there is a directed path from
\(v_j^*(\qq_j)\) to \(v(\aaa)\) in \(\mch(\QQ)\), and zero otherwise.
\\
\hline

\(v(\aaa)\sim v(\aaa')\)
&
Node equivalence
&
Two latent vector nodes are equivalent if
\(\Psi(v(\aaa))=\Psi(v(\aaa'))\).
\\
\hline

\([v(\aaa)]\)
&
Equivalence class
&
The collection of latent vector nodes having the same path status vector as
\(v(\aaa)\).
\\
\hline
\(\mcr(\QQ)\)
&
Representative latent vectors
&
A set containing one latent vector from each equivalence class.
\\
\hline

\(R(\QQ)\)
&
Representative latent vector nodes
&
A set containing one latent vector node from each equivalence class.
\\
\hline

\(G(\QQ)\)
&
Hasse-based distribution graph
&
The graph obtained by merging equivalent latent vector nodes
of \(\mch(\QQ)\) and redrawing the edges according to
Definition~\ref{def:GQ}.
\\
\hline
\end{tabular}\caption{Notation and definitions used in the graphical construction.}
\label{tab:graphical-objects}
\end{table}

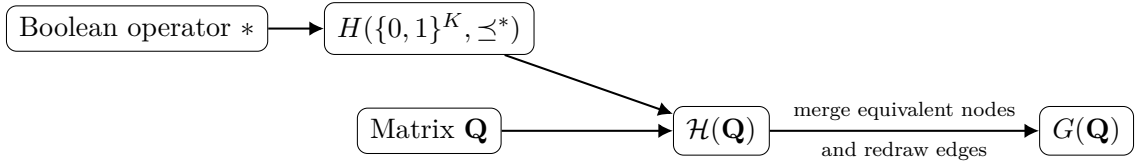
\begin{figure}[htbp]
\centering
\resizebox{0.92\textwidth}{!}{
\begin{tikzpicture}[
    box/.style={
        draw,
        rounded corners,
        align=center,
        inner xsep=5pt,
        inner ysep=4pt,
        font=\small
    },
    arr/.style={->, thick}
]

\node[box] (star) at (0,0)
    {Boolean operator \(\ast\)};
\node[box] (H) at (4.0,0)
    {$H(\{0,1\}^K,\preceq^\ast)$};

\node[box] (Q) at (4.0,-1.4)
    {Matrix \(\QQ\)};
\node[box] (mch) at (8.0,-1.4)
    {$\mch(\QQ)$};
\node[box] (G) at (13.0,-1.4)
    {$G(\QQ)$};

\draw[arr] (star) -- (H);
\draw[arr] (H) -- (mch);
\draw[arr] (Q) -- (mch);
\draw[arr] (mch) --
    node[above, font=\scriptsize] {merge equivalent nodes}
    node[below, font=\scriptsize] {and redraw edges}
    (G);

\end{tikzpicture}
}
\caption{Pipeline for the construction of \(G(\QQ)\).
The Boolean operator \(\ast\) induces the Hasse diagram on the latent vector
space. Incorporating the \(\QQ\)-matrix produces the extended Hasse diagram
\(\mch(\QQ)\), and merging equivalent latent vector nodes yields
\(G(\QQ)\).}
\label{fig:Q-to-G}
\end{figure}

% Thus, \(H(\{0,1\}^K,\preceq^\ast)\) is determined by the Boolean operator \(\ast\), while \(\mch(\QQ)\) and \(G(\QQ)\) also incorporate the information in
% \(\QQ\). The passage from \(\mch(\QQ)\) to \(G(\QQ)\) merges precisely those
% latent vector nodes having the same path status vector.

\subsection{Interpretation of Theorems~\ref{thm:graph}
and~\ref{thm:algebraic} and the role of \texorpdfstring{\(K\)}{K}}
\label{app:theory_intuition}

In this subsection, we provide some intuition for
Theorems~\ref{thm:graph} and~\ref{thm:algebraic} and discuss how the number of
latent variables \(K\) affects identifiability.

\paragraph{Graphical characterization.}
In  \(G(\QQ)\), each representative latent vector node \(v(\aaa)\) corresponds
to a distinct path status vector and hence to a distinct conditional
distribution of \(\RR\). The marginal distribution of \(\RR\) is therefore a
mixture over these nodes, with mixing weights given by the grouped latent probabilities \(p_{[\aaa]}\).
Suppose that, for some non-equivalent matrix \(\bQQ\nsim\QQ\), there exists a
set of representative latent vector nodes \(A\) such that
\(G(\bQQ)-A\sim G(\QQ)\). The graph isomorphism matches each representative
latent vector node under \(\QQ\) with one under \(\bQQ\), and hence matches
every conditional response distribution generated by \(\QQ\) with one
generated by \(\bQQ\). Assigning the same grouped probabilities to the
matched nodes and zero grouped probability to the additional nodes in \(A\)
allows \(\bQQ\) to generate the same marginal distribution of \(\RR\) as
\(\QQ\). Hence, \(\QQ\) is not identifiable. 
Conversely, if no such \(\bQQ\) and \(A\) exist, then every
non-equivalent \(\bQQ\) lacks at least one path status vector
generated by \(\QQ\). Because the grouped probability associated
with this path status vector under \(\QQ\) is positive, the
corresponding response vector has positive probability under
\(\QQ\) but zero probability under \(\bQQ\). Thus, \(\bQQ\) cannot
generate the same marginal distribution of \(\RR\) as \(\QQ\).

\paragraph{Algebraic characterization.}
The algebraic characterization expresses this reasoning directly in terms of
path status vectors. Each column of \(\PPhi(\QQ)\) is the path status vector
generated by one latent vector. Therefore, the condition
\(\operatorname{Cols}\bigl(\PPhi(\QQ)\bigr)\subseteq
\operatorname{Cols}\bigl(\PPhi(\bQQ)\bigr)\) means that every path status
vector generated under \(\QQ\), and hence every corresponding conditional
response distribution, also appears under \(\bQQ\). Any additional distinct
columns under \(\bQQ\) correspond to the additional nodes deleted from
\(G(\bQQ)\). Theorems~\ref{thm:graph} and~\ref{thm:algebraic} therefore give
graphical and algebraic descriptions of the same condition.

\paragraph{The role of \(K\).}
$K$ determines the size and complexity of the latent space. The number of latent vectors is \(2^K\), and the \(\PPhi(\QQ)\)-matrix has one column for each latent vector. Thus, as \(K\) increases, the rows of \(\QQ\) generate path status vectors over a larger latent space. Identifiability does not require all \(2^K\) columns of \(\PPhi(\QQ)\) to be distinct. By Theorem~\ref{thm:algebraic}, however, it requires that \(\operatorname{Cols}\bigl(\PPhi(\QQ)\bigr)\nsubseteq \operatorname{Cols}\bigl(\PPhi(\bQQ)\bigr)\) for every non-equivalent \(\bQQ\nsim\QQ\). If the rows of \(\QQ\) are not sufficiently diverse, many latent vectors may generate the same path status vector, leaving fewer distinct columns with which to distinguish \(\QQ\) from a non-equivalent \(\bQQ\).

For the basic conjunctive and disjunctive models,
Theorem~\ref{thm:sub}\textup{(iii)} makes this effect more concrete. The
necessary two- and three-column conditions must hold for all
\(\binom{K}{2}\) pairs and \(\binom{K}{3}\) triples of columns. These numbers
increase from \(10\) and \(10\), respectively, when \(K=5\), to \(45\) and
\(120\) when \(K=10\). Thus, when \(J\) is fixed, the same \(J\) rows must
provide the required patterns over many more subsets of latent variables.
This helps explain the empirical tendency for identifiability to become less
common as \(K\) increases with \(J\) fixed. The simulations in
Section~\ref{sec:sim_conditions} illustrate this effect and also show that
increasing \(J\) can compensate for an increase in \(K\).

% Manuscript fragment: retain the macros and main-text labels in your paper.
% New bibliography entries are provided in idQ-SAT-references.bib.
% The size bounds below apply to the explicit encoding described here.
% If the existing implementation uses another cardinality encoding, its
% additional cost should be checked before attributing the full bound to it.

\subsection{SAT formulation and implementation}
\label{app:sat-primer}

Boolean satisfiability (SAT) asks whether binary variables can be
assigned values that make a logical formula true. The values 1 and 0
represent true and false. The operations $\wedge$, $\vee$, and $\neg$
denote AND, OR, and NOT, respectively. Thus, $x\wedge y$ is true when
both variables are true, $x\vee y$ is true when at least one is true,
and $\neg x$ is true when $x$ is false. A variable or its negation is
called a \emph{literal}. A \emph{clause} is a disjunction of literals,
and a formula in \emph{conjunctive normal form} (CNF) is a conjunction
of clauses. Such a formula is true precisely when every clause is true.
For example, $(x_1\vee x_2)\wedge\neg x_1$ is satisfied by $x_1=0$
and $x_2=1$. A formula is \emph{satisfiable} if at least one satisfying
assignment exists and \emph{unsatisfiable} otherwise. See
\citet{biere2021handbook} for further background.

In our problem, $\QQ$ is fixed, and the unknowns are the entries of
$\bQQ$ and $\HH$. As in Proposition~\ref{prop:sat-nonequivalence},
the columns of $\QQ$ are arranged in weakly non-increasing
lexicographic order. For Algorithm~\ref{alg:id}, this ordering is
applied after basis reduction. Throughout this subsection, $J$
denotes the number of rows supplied to the encoding, which is the
number of rows of $\Qbasis$ when basis reduction is used.
The binary variables $x_{j\ell}$ and $h_{\ell k}$ represent
$\bar q_{j\ell}$ and the entries of $\HH$, respectively.
We encode the Boolean factorization, the lexicographic ordering
of the columns of $\bQQ$, and the inequality $\bQQ\neq\QQ$.

First consider the Boolean factorization. If $q_{jk}=0$, every term
$x_{j\ell}\wedge h_{\ell k}$ must be zero. This is enforced by the
$K$ clauses
\begin{equation}\label{eq:supp-sat-zero}
\neg x_{j\ell}\vee\neg h_{\ell k},
\qquad \ell\in[K].
\end{equation}
If $q_{jk}=1$, at least one of these terms must be one. For this entry,
introduce auxiliary binary variables $z_{j\ell k}$, $\ell\in[K]$,
and impose
\begin{equation}\label{eq:supp-sat-and}
\neg z_{j\ell k}\vee x_{j\ell},
\qquad
\neg z_{j\ell k}\vee h_{\ell k},
\qquad \ell\in[K],
\end{equation}
together with
\begin{equation}\label{eq:supp-sat-positive}
\bigvee_{\ell=1}^K z_{j\ell k}.
\end{equation}
These clauses require at least one selected term to have both
factors equal to one. The reverse implication
$(x_{j\ell}\wedge h_{\ell k})\Rightarrow z_{j\ell k}$
is unnecessary: the $z$-variables need only witness one valid term.
Every assignment satisfying the Boolean factorization extends to
these clauses by selecting such a term for each positive entry.
Conversely, the clauses provide a positive term at every positive
entry and forbid all positive terms at every zero entry.
They therefore encode $\QQ=\bQQ\odot\HH$ exactly after
existentially quantifying the auxiliary variables.

Next, we impose weakly non-increasing lexicographic order on the
columns of $\bQQ$. Lexicographic ordering is a standard way to
remove permutation symmetries; see \citet{frisch2006propagation}.
For each adjacent pair of columns $k$ and $k+1$, introduce
auxiliary variables $e_{j,k}$ for $j=0,\ldots,J-1$.
The variable $e_{j,k}$ records whether the first $j$ entries
of these columns agree. Set $e_{0,k}=1$ using a unit clause and,
for $j=1,\ldots,J-1$, impose
\[
e_{j,k}\leftrightarrow
\bigl(e_{j-1,k}\wedge(x_{j,k}\leftrightarrow x_{j,k+1})\bigr).
\]
This equivalence is encoded by the five clauses
\begin{equation}\label{eq:lex-encoding}
\begin{aligned}
&\neg e_{j,k}\vee e_{j-1,k},\\
&\neg e_{j,k}\vee\neg x_{j,k}\vee x_{j,k+1},\\
&\neg e_{j,k}\vee x_{j,k}\vee\neg x_{j,k+1},\\
&\neg e_{j-1,k}\vee\neg x_{j,k}\vee\neg x_{j,k+1}\vee e_{j,k},\\
&\neg e_{j-1,k}\vee x_{j,k}\vee x_{j,k+1}\vee e_{j,k}.
\end{aligned}
\end{equation}
For every $j\in[J]$, we also impose
\begin{equation}\label{eq:supp-sat-lex-order}
\neg e_{j-1,k}\vee x_{j,k}\vee\neg x_{j,k+1}.
\end{equation}
If the preceding entries agree, this clause forbids
$x_{j,k}=0$ and $x_{j,k+1}=1$.
Thus, at the first position where the columns differ, the left
column must have a 1 and the right column a 0. Equal columns
are allowed. Applying these constraints for $k=1,\ldots,K-1$
gives the required weak lexicographic ordering.

Since both $\QQ$ and $\bQQ$ have their columns in this order,
they are equivalent up to column permutation if and only if
they are equal. Consequently, non-equivalence is enforced
by the single clause
\begin{equation}\label{eq:supp-sat-exclusion}
\left(\bigvee_{(j,k):\,q_{jk}=0}x_{jk}\right)
\vee
\left(\bigvee_{(j,k):\,q_{jk}=1}\neg x_{jk}\right).
\end{equation}
This clause requires at least one entry of $\bQQ$ to differ
from the corresponding entry of $\QQ$.

For matrices reaching Step~2 of Algorithm~\ref{alg:id}, every
row and column of $\QQ$ is nonzero: basis reduction removes
zero rows, and the two-column check excludes zero columns.
The implementation also includes the implied clauses
\begin{equation}\label{eq:supp-sat-nonempty}
\bigvee_{\ell=1}^K x_{j\ell},
\qquad j\in[J],
\qquad\text{and}\qquad
\bigvee_{\ell=1}^K h_{\ell k},
\qquad k\in[K].
\end{equation}
A zero row of $\bQQ$ would produce a zero row of $\QQ$,
and a zero column of $\HH$ would produce a zero column of $\QQ$.
These additional clauses therefore do not change satisfiability.

For matrices reaching Step~2 of Algorithm~\ref{alg:id}, the resulting CNF formula is satisfiable precisely when the
matrices in Proposition~\ref{prop:sat-nonequivalence} exist.
Indeed, every satisfying assignment gives the required
factorization with ordered columns of $\bQQ$ and $\bQQ\neq\QQ$.
Conversely, every such pair of matrices extends to a satisfying
assignment by choosing the $z$-variables as above and assigning
the $e$-variables their prefix-equality values.
The ordering does not remove any alternative up to column
permutation, because columns of $\bQQ$ can be sorted by
simultaneously permuting the corresponding rows of $\HH$.
Thus, satisfiability demonstrates failure of the algebraic
condition, whereas unsatisfiability establishes the condition.

We implement the formulation using PySAT \citep{ignatiev2018pysat},
with \texttt{Glucose42} as the default solver.
A satisfiable result provides values for the entries of $\bQQ$
and $\HH$, and the implementation independently verifies the
returned factorization and non-equivalence.
A solver interruption or time limit, if used, leaves the instance
unresolved and is not interpreted as an unsatisfiable result.

\subsection{Computational complexity of Algorithm~\ref{alg:id}}
\label{app:sub_comp_comx}

Algorithm~\ref{alg:id} first reduces $\QQ$ to $\Qbasis$ in
Step~0 and then applies the submatrix checks in Step~1.
For an input with $J$ rows and $K$ columns, a direct basis
reduction can be implemented in $O(J^2K)$ time: for each
distinct nonzero row, compute the Boolean union of all rows
strictly below it in the coordinatewise order and check
whether this union equals the row.
Let $J_b$ denote the number of rows of $\Qbasis$.
The identity, two-column, and three-column checks can be
performed by direct scans in $O(J_bK)$, $O(J_bK^2)$, and
$O(J_bK^3)$ time, respectively.
Basis reduction can substantially reduce the subsequent SAT
instance, while the submatrix checks can settle the problem
without invoking the solver.
Section~\ref{sec:sim_preprocessing} examines their empirical
costs and their effects on SAT verification and total runtime.

For matrices reaching Step~2, there are $J_bK$ variables
for $\bQQ$ and $K^2$ variables for $\HH$.
Each zero entry of $\Qbasis$ contributes $K$ binary clauses.
Each positive entry contributes $K$ auxiliary variables
and $2K+1$ clauses: $2K$ binary implication clauses and
one clause of length $K$.
Writing $N=\|\Qbasis\|_0$ for the number of positive entries,
the factorization encoding therefore uses
\[
KN\ \text{auxiliary variables},\qquad
J_bK^2+(K+1)N\ \text{clauses},\qquad
2J_bK^2+3KN\ \text{literal occurrences}.
\]
Since $N\leq J_bK$, each of these quantities is $O(J_bK^2)$.

The lexicographic encoding uses $O(J_b)$ auxiliary variables,
clauses, and literal occurrences for each of the $K-1$
adjacent column pairs, giving $O(J_bK)$ in total.
The exclusion of $\bQQ=\QQ$ requires one clause with $J_bK$
literals and no auxiliary variables.
The implied nonzero-row and nonzero-column constraints
add $J_b+K$ clauses and $J_bK+K^2$ literal occurrences.
Consequently, the complete CNF formula has $O(J_bK^2)$
variables, clauses, and literal occurrences.
Sorting the columns of $\Qbasis$ takes
$O(J_bK\log K)$ time by comparison sorting, and constructing
and storing the CNF formula takes $O(J_bK^2)$ time and
space, respectively.

These bounds describe preprocessing and construction of
the SAT instance, not the time required to solve it.
A polynomial-size encoding does not provide a polynomial-time
guarantee for SAT verification.
Solver runtime depends on the structure of the instance
as well as the encoding, and a smaller formula need not
be faster on every instance.
We therefore evaluate the computational performance
empirically in Section~\ref{sec:simulation}.
\subsection{Quantifying non-identifiability by SAT enumeration}
\label{app:quantify_non_id}

The SAT formulation can also quantify the extent of
non-identifiability by counting $\bQQ$-matrices up to column
permutation. For a fixed $\QQ\in\{0,1\}^{J\times K}$, define
\[
\mathcal C_{\supseteq}(\QQ)
=
\left\{
[\bQQ]_\sim:
\bQQ\nsim\QQ,\ 
\operatorname{Cols}\{\PPhi(\QQ)\}
\subseteq
\operatorname{Cols}\{\PPhi(\bQQ)\}
\right\},
\]
where $\bQQ$ ranges over $\{0,1\}^{J\times K}$ and
$[\bQQ]_\sim$ denotes its equivalence class under column
permutation. Let
\[
D_{\supseteq}(\QQ)=|\mathcal C_{\supseteq}(\QQ)|.
\]
Then $D_{\supseteq}(\QQ)=0$ if and only if $\QQ$ satisfies
the algebraic condition. Under the assumptions of
Theorem~\ref{thm:algebraic}, this is equivalent to
identifiability of $\QQ$.
When positive, $D_{\supseteq}(\QQ)$ counts the distinct
equivalence classes of $\bQQ$-matrices satisfying the
containment relation above.
We compute this quantity using the original $\QQ$, without
basis reduction, because basis reduction preserves the
algebraic condition but may change the number of alternatives.

Under the conjunctive operator, the alternatives counted by
$D_{\supseteq}(\QQ)$ are exactly the matrices $\bQQ\nsim\QQ$
for which $\QQ=\bQQ\odot\HH$ holds for some
$\HH\in\{0,1\}^{K\times K}$.
As in Proposition~\ref{prop:sat-nonequivalence}, arrange the
columns of $\QQ$ in weakly non-increasing lexicographic order
and impose the same ordering on $\bQQ$. We then encode
\begin{equation}\label{eq:supp-sat-enumeration}
\begin{aligned}
&\QQ=\bQQ\odot\HH,
\qquad
\bQQ\neq\QQ,\\
&\bQQ_{\Cdot,1}
\succeq_{\mathrm{lex}}
\bQQ_{\Cdot,2}
\succeq_{\mathrm{lex}}\cdots
\succeq_{\mathrm{lex}}
\bQQ_{\Cdot,K}.
\end{aligned}
\end{equation}
This formulation does not require the two-column condition.
The ordering loses no feasible equivalence class, since
sorting the columns of $\bQQ$ and correspondingly permuting
the rows of $\HH$ preserves the Boolean product.
Because both $\QQ$ and $\bQQ$ have ordered columns,
$\bQQ\sim\QQ$ holds if and only if $\bQQ=\QQ$.
Thus, each class in $\mathcal C_{\supseteq}(\QQ)$ has exactly
one ordered matrix representative, including when columns
are repeated.

The factorization, ordering, and inequality constraints are
encoded as described in Section~\ref{app:sat-primer}.
For enumeration, the additional nonzero-row clauses for
$\bQQ$ and nonzero-column clauses for $\HH$ in
\eqref{eq:supp-sat-nonempty} are included only when the
corresponding row or column of $\QQ$ is nonzero.
No further restrictions on $\bQQ$ are imposed.

Let $\mathcal F_{\QQ}$ denote the resulting SAT formula,
with variables $x_{j\ell}$ encoding the entries of $\bQQ$.
Solve $\mathcal F_{\QQ}$ repeatedly. Suppose the $m$th
solution returns the matrix $\bQQ^{(m)}$.
Record this matrix and add the blocking clause
\begin{equation}\label{eq:supp-sat-block}
\left(
\bigvee_{(j,\ell):\,\bar q^{(m)}_{j\ell}=1}
\neg x_{j\ell}
\right)
\vee
\left(
\bigvee_{(j,\ell):\,\bar q^{(m)}_{j\ell}=0}
x_{j\ell}
\right).
\end{equation}
This clause requires at least one entry of the next matrix
to differ from $\bQQ^{(m)}$.
Since it involves only the $x_{j\ell}$ variables, it excludes
this matrix for every choice of $\HH$ and of the auxiliary
variables, while leaving every other feasible matrix
unchanged. Consequently, each equivalence class is counted
once, regardless of how many factorizations or auxiliary
assignments its representative admits.

Repeat until the solver returns \textsc{unsat}.
At that point, the number of recorded matrices equals
$D_{\supseteq}(\QQ)$.
If enumeration is terminated earlier, the number recorded
is a lower bound on $D_{\supseteq}(\QQ)$.
% If enumeration is stopped before completion, the number recorded is a lower bound. Complete enumeration can be substantially more expensive than deciding whether any alternative exists, since it requires finding and excluding every alternative matrix.
}
\section{Proofs of results in Section~\ref{sec:graph}}
\subsection{Proofs of results in Sections~\ref{sec:graph-representative} and~\ref{sec:graphcond}}\label{app:secgraph}

\subsubsection{Statement and Proof of Fact~\ref{fact:injective}}

\begin{fact}\label{fact:injective}
When the operator $\ast$ is defined either by an existing non-strict partial order as:
\begin{equation}\label{eq:ope_exist}
    \qq*\aaa=\mathbf 1\{\qq\preccurlyeq\aaa\}
\end{equation}
or as the disjunctive operator~\eqref{eq:ope-disj2}, it is injective in both $\aaa$ and $\qq$. 
\end{fact}

\begin{proof}

We first prove for the case when $\ast$ is defined as in Equation~\eqref{eq:ope_exist}.

\smallskip
\noindent\textbf{Injectivity in $\qq$.}  
Suppose that for all $\aaa \in \{0,1\}^K$,  
\[
\qq_1 \preccurlyeq \aaa \quad\text{holds if and only if}\quad \qq_2 \preccurlyeq \aaa \ \text{holds}.
\]  
Taking $\aaa = \qq_1$, we have $\qq_1 \preccurlyeq \qq_1$, hence $\qq_2 \preccurlyeq \qq_1$.  
Similarly, taking $\aaa = \qq_2$ yields $\qq_1 \preccurlyeq \qq_2$.  
Since $\preccurlyeq$ is a partial order, by antisymmetry, these two relations imply $\qq_1 = \qq_2$.  
Thus, $\ast$ is injective in $\qq$.  

\smallskip
\noindent\textbf{Injectivity in $\aaa$.}  
Suppose that for all $\qq \in \{0,1\}^K$,  
\[
\qq \preccurlyeq \aaa_1 \quad\text{holds if and only if}\quad \qq \preccurlyeq \aaa_2 \ \text{holds}.
\]  
Taking $\qq = \aaa_1$ gives $\aaa_1 \preccurlyeq \aaa_2$, and taking $\qq = \aaa_2$ gives $\aaa_2 \preccurlyeq \aaa_1$.  
Since $\preccurlyeq$ is a partial order, by antisymmetry,  $\aaa_1 = \aaa_2$, so $\ast$ is injective in $\aaa$.  

\smallskip

We next prove when the operator is the disjunctive operator~\eqref{eq:ope-disj2}.
Notice that the disjunctive operator can be written equivalently as
\begin{equation}\label{eq:ope-disj3}
    \qq \ast \aaa \;=\; \qq \npreceq \tilde{\aaa},
\end{equation}
where $\tilde{\aaa}$ is obtained by flipping the zeros and ones in
$\aaa$, and $\preceq$ denotes the product order.

\smallskip
\noindent\textbf{Injectivity in $\qq$.}
Assume that for all
$\aaa\in\{0,1\}^{K}$ we have
\[
  \qq_1 \npreceq \tilde{\aaa}
  \quad\text{holds if and only if}\quad
  \qq_2 \npreceq \tilde{\aaa}.
\]
Equivalently,
\[
  \qq_1 \preceq \tilde{\aaa}
  \quad\text{if and only if}\quad
  \qq_2 \preceq \tilde{\aaa}.
\]
Taking $\tilde{\aaa}=\qq_1$ yields
$\qq_2 \preceq \qq_1$, and taking $\tilde{\aaa}=\qq_2$ yields
$\qq_1 \preceq \qq_2$.  Hence $\qq_1=\qq_2$, so the operator is
injective in $\qq$.

\smallskip
\noindent\textbf{Injectivity in $\aaa$.}
Assume that for every $\qq\in\{0,1\}^{K}$ we have
\[
  \qq \npreceq \tilde{\aaa}_1
  \quad\text{holds if and only if}\quad
  \qq \npreceq \tilde{\aaa}_2,
\]
that is,
\[
  \qq \preceq \tilde{\aaa}_1
  \quad\text{if and only if}\quad
  \qq \preceq \tilde{\aaa}_2.
\]
Taking $\qq=\tilde{\aaa}_1$ gives
$\tilde{\aaa}_1 \preceq \tilde{\aaa}_2$, and taking
$\qq=\tilde{\aaa}_2$ gives the reverse inequality.  Therefore
$\tilde{\aaa}_1=\tilde{\aaa}_2$ and hence $\aaa_1=\aaa_2$.  Thus the
operator is injective in $\aaa$.
\end{proof}

\subsubsection{Proof of Proposition~\ref{prop:partial}}

\begin{proof}
A binary relation \(\preccurlyeq\) on a set $S$ is a \emph{non-strict partial order} \citep{davey2002introduction} if, for all $a,b,c \in S$, it satisfies reflexivity (\(a \preccurlyeq a\)), antisymmetry (\(a \preccurlyeq b\) and \(b \preccurlyeq a\) imply \(a = b\)), and transitivity (\(a \preccurlyeq b\) and \(b \preccurlyeq c\) imply \(a \preccurlyeq c\)). We verify these three properties of the relation $\preceq^\ast$ as follows.

\smallskip
\noindent\textbf{Reflexivity.}  
For any \(\aaa \in \{0,1\}^K\), we have \(\qq \ast \aaa \leq \qq \ast \aaa\) for all \(\qq \in \{0,1\}^K\), implying \(\aaa \preceq^\ast \aaa\).

\smallskip
\noindent\textbf{Antisymmetry.}  
For any \(\aaa_1, \aaa_2 \in \{0,1\}^K\), if \(\aaa_1 \preceq^\ast \aaa_2\) and \(\aaa_2 \preceq^\ast \aaa_1\), then for all \(\qq \in \{0,1\}^K\) we have \(\qq \ast \aaa_1 \leq \qq \ast \aaa_2 \leq \qq \ast \aaa_1\). This implies \(\qq \ast \aaa_1 = \qq \ast \aaa_2\) for all \(\qq\). By the standing assumption that \(\ast\) is injective in \(\aaa\), we conclude that \(\aaa_1 = \aaa_2\).

\smallskip
\noindent\textbf{Transitivity.}  
For any \(\aaa_1, \aaa_2, \aaa_3 \in \{0,1\}^K\), if \(\aaa_1 \preceq^\ast \aaa_2\) and \(\aaa_2 \preceq^\ast \aaa_3\), then for all \(\qq \in \{0,1\}^K\) we have \(\qq \ast \aaa_1 \leq \qq \ast \aaa_2 \leq \qq \ast \aaa_3\), which implies \(\qq \ast \aaa_1 \leq \qq \ast \aaa_3\) for all \(\qq\). Hence \(\aaa_1 \preceq^\ast \aaa_3\).

Since $\preceq^\ast$ satisfies reflexivity, antisymmetry, and transitivity, it is a non-strict partial order on $\{0,1\}^K$.
\end{proof}

\subsubsection{Induced partial orders under the conjunctive and disjunctive operators}
\label{supp:product-order}

We first establish a general observation.

\begin{lemma}\label{lem:ope_partial}
Let $\preccurlyeq$ be a partial order on $\{0,1\}^K$.
If $\qq\ast\aaa=\1\{\qq\preccurlyeq\aaa\}$, then the
induced order $\preceq^\ast$ coincides with $\preccurlyeq$.
\end{lemma}

\begin{proof}
By definition, $\aaa_1\preceq^\ast\aaa_2$ if and only if
$\1\{\qq\preccurlyeq\aaa_1\}\leq
\1\{\qq\preccurlyeq\aaa_2\}$ for every $\qq\in\{0,1\}^K$.
Taking $\qq=\aaa_1$ gives $\aaa_1\preccurlyeq\aaa_2$.
Conversely, if $\aaa_1\preccurlyeq\aaa_2$, the required
inequality follows from transitivity.
\end{proof}

For the conjunctive operator,
$\qq\ast\aaa=\1\{\qq\preceq\aaa\}$, so
Lemma~\ref{lem:ope_partial} gives $\preceq^\ast=\preceq$.

For the disjunctive operator,
$\qq\ast\aaa=1-\1\{\qq\preceq\one-\aaa\}$,
where $\one$ is the all-ones vector. Thus, by definition,
\[
\aaa_1\preceq^\ast\aaa_2
\quad\Longleftrightarrow\quad
\1\{\qq\preceq\one-\aaa_2\}
\leq
\1\{\qq\preceq\one-\aaa_1\}
\quad\text{for every }\qq\in\{0,1\}^K.
\]
Taking $\qq=\one-\aaa_2$ gives
$\one-\aaa_2\preceq\one-\aaa_1$, or equivalently
$\aaa_1\preceq\aaa_2$. Conversely, if
$\aaa_1\preceq\aaa_2$, then
$\one-\aaa_2\preceq\one-\aaa_1$, and the displayed
inequality follows from transitivity. Hence, the
disjunctive operator also induces the product order.
\subsubsection{Proof of Proposition~\ref{prop:path-response}}

\begin{proof}
We first show the equivalence between $\qq_j \ast \aaa = 1$ (i.e., $\aaa \in \mca_j^+$) and the existence of a minimal capable vector $\aaa_j^{\min} \in \mcm_j^+$ such that $\aaa_j^{\min} \preceq^\ast \aaa$.

\smallskip
\noindent\textbf{Necessity.}  
Suppose $\aaa \in \mca_j^+$, i.e., $\qq_j \ast \aaa = 1$.  
Assume, for the sake of contradiction, that there does not exist any $\aaa_j^{\min} \in \mcm_j^+$ with $\aaa_j^{\min} \preceq^\ast \aaa$.  
This means $\aaa \notin \mcm_j^+$, since if $\aaa \in \mcm_j^+$ we could take $\aaa_j^{\min} = \aaa$ and the property would hold trivially.  
Because $\aaa \notin \mcm_j^+$, by the definition of $\mcm_j^+$ there must exist some $\aaa_1 \in \mca_j^+$ such that $\aaa_1 \prec^\ast \aaa$.  
By assumption, $\aaa_1 \notin \mcm_j^+$, so there exists $\aaa_2 \in \mca_j^+$ with $\aaa_2 \prec^\ast \aaa_1 \prec^\ast \aaa$.  
Continuing in this way produces a strictly decreasing chain
\[
\aaa \succ^\ast \aaa_1 \succ^\ast \aaa_2 \succ^\ast \cdots
\]
inside the finite set $\{0,1\}^K$.  
Such an infinite descending chain in a finite set is impossible, yielding a contradiction.  
Therefore, there must exist some $\aaa_j^{\min} \in \mcm_j^+$ such that $\aaa_j^{\min} \preceq^\ast \aaa$.

\smallskip
\noindent\textbf{Sufficiency.}  
Let $\aaa_j^{\min} \in \mcm_j^+$ with $\aaa_j^{\min} \preceq^\ast \aaa$.  
Since $\aaa_j^{\min} \in \mca_j^+$, we have $\qq_j \ast \aaa_j^{\min} = 1$.  
By the definition of $\preceq^\ast$, we have $\qq_j \ast \aaa \ge \qq_j \ast \aaa_j^{\min} = 1$, hence $\qq_j \ast \aaa = 1$.

\bigskip
We now show the equivalence between the existence of such $\aaa_j^{\min}$ for $\aaa$ with $\aaa_j^{\min} \preceq^\ast \aaa$ and the existence of a directed path from $v_j^*(\qq_j)$ to $v(\aaa)$ in the extended Hasse diagram.

\smallskip
\noindent\textbf{ Necessity.}  
If $\aaa_j^{\min} \preceq^\ast \aaa$ for some $\aaa_j^{\min} \in \mcm_j^+$, then by Definition~\ref{def:extended_hasse} there is a directed edge from $v_j^*(\qq_j)$ to $v(\aaa_j^{\min})$.  
If $\aaa = \aaa_j^{\min}$, this edge connects $v_j^*(\qq_j)$ directly to $v(\aaa)$.  
If $\aaa_j^{\min} \prec^\ast \aaa$, then by the properties of the Hasse diagram there is a directed path from $v(\aaa_j^{\min})$ to $v(\aaa)$, and concatenating the edge from $v_j^*(\qq_j)$ to $v(\aaa_j^{\min})$ gives a path from $v_j^*(\qq_j)$ to $v(\aaa)$.

\smallskip
\noindent\textbf{ Sufficiency.}  
If there is a directed path from $v_j^*(\qq_j)$ to $v(\aaa)$, then by Definition~\ref{def:extended_hasse} the first edge of this path must be an edge from $v_j^*(\qq_j)$ to some $v(\aaa_j^{\min})$ with $\aaa_j^{\min} \in \mcm_j^+$.  
If $v(\aaa_j^{\min}) = v(\aaa)$, then $\aaa_j^{\min} = \aaa$ and clearly $\aaa_j^{\min} \preceq^\ast \aaa$.  
Otherwise, the remainder of the path is from $v(\aaa_j^{\min})$ to $v(\aaa)$ in the Hasse diagram, which implies $\aaa_j^{\min} \preceq^\ast \aaa$.
\end{proof}

\subsubsection{Proof of Theorem~\ref{thm:graph}}

By Definition~\ref{def:path_status}, for every
\(\aaa\in\{0,1\}^K\),
\[
\Psi(v(\aaa))
=
\PPhi(\QQ)_{\Cdot,\aaa}.
\]
Since \(R(\QQ)\) contains exactly one representative node from each
equivalence class, its path status vectors are distinct. Therefore,
\begin{equation}\label{eq:path-status-columns}
\bigl\{\Psi(v):v\in R(\QQ)\bigr\}
=
\operatorname{Cols}\{\PPhi(\QQ)\}.
\end{equation}

\begin{lemma}\label{lem:graph-column}
For any \(A\subseteq R(\bQQ)\), $G(\QQ)\sim G(\bQQ)-A$
if and only if
\[
\operatorname{Cols}\{\PPhi(\QQ)\}
=
\bigl\{
\Psi(v):v\in R(\bQQ)\setminus A
\bigr\}.
\]
\end{lemma}

\begin{proof}
First, in either \(G(\QQ)\) or \(G(\bQQ)-A\), a surviving
representative node \(v\) has \(\Psi_j(v)=1\) if and only if
there is a directed path from the \(j\)th item node to \(v\).
Indeed, if \(\Psi_j(v)=1\), choose a minimal surviving path
status vector \(\Psi(u)\preceq\Psi(v)\) with
\(\Psi_j(u)=1\). Definition~\ref{def:GQ} gives an edge from the
\(j\)th item node to \(u\); if \(u\neq v\), the cover relations
give a directed path from \(u\) to \(v\). Conversely, every
edge from the \(j\)th item node ends at a node with \(j\)th
entry one, and this entry remains one along each subsequent
edge. These arguments apply to \(G(\bQQ)-A\) because its edges
are redrawn using the same rules on the surviving nodes.

Suppose first that \(G(\QQ)\sim G(\bQQ)-A\). By
Definition~\ref{def:iso}, the graph isomorphism matches the \(j\)th
item node of \(G(\QQ)\) with the \(j\)th item node of
\(G(\bQQ)-A\). It therefore preserves the existence of a directed
path from each item node to every representative node. Hence, matched
representative nodes have the same path status vector.

The isomorphism gives a bijection from \(R(\QQ)\) to
\(R(\bQQ)\setminus A\). Equation~\eqref{eq:path-status-columns}
therefore gives
\[
\operatorname{Cols}\{\PPhi(\QQ)\}
=
\bigl\{
\Psi(v):v\in R(\bQQ)\setminus A
\bigr\}.
\]

Conversely, suppose that the two sets of path status vectors are
equal. Since the path status vectors of the representative nodes are
distinct, every node in \(R(\QQ)\) can be matched uniquely with a node
in \(R(\bQQ)\setminus A\) having the same path status vector. Match
also the \(j\)th item nodes of the two graphs.

The edges between representative nodes are determined by the cover
relations among their path status vectors. The edges from the \(j\)th
item node are determined by the minimal path status vectors whose
\(j\)th entry equals one. Since \(G(\bQQ)-A\) is obtained by applying
these same edge rules to the surviving path status vectors, the
matching preserves every edge. It therefore defines an isomorphism $G(\QQ)\sim G(\bQQ)-A.$
\end{proof}

\begin{proof}[Proof of Theorem~\ref{thm:graph}]
By Theorem~\ref{thm:algebraic}, \(\QQ\) is not identifiable if and
only if there exists \(\bQQ\nsim\QQ\) such that
\[
\operatorname{Cols}\{\PPhi(\QQ)\}
\subseteq
\operatorname{Cols}\{\PPhi(\bQQ)\}.
\]
By Equation~\eqref{eq:path-status-columns}, this holds if and only if there exists \(A\subseteq R(\bQQ)\) such that
\[
\operatorname{Cols}\{\PPhi(\QQ)\}
=
\bigl\{
\Psi(v):v\in R(\bQQ)\setminus A
\bigr\}.
\]
Indeed, when the inclusion holds, we may take
\[
A
=
\left\{
v(\bar\aaa)\in R(\bQQ):
\PPhi(\bQQ)_{\Cdot,\bar\aaa}
\notin
\operatorname{Cols}\{\PPhi(\QQ)\}
\right\}.
\]
Conversely, the preceding equality immediately implies the column-set
inclusion.

By Lemma~\ref{lem:graph-column}, this is equivalent to the existence
of \(\bQQ\nsim\QQ\) and \(A\subseteq R(\bQQ)\) such that
\[
G(\QQ)\sim G(\bQQ)-A.
\]
Therefore, \(\QQ\) is not identifiable if and only if the graphical
condition~\eqref{eq:g_condition} fails.
\end{proof}

\subsection{Proofs of results in Section~\ref{sec:algebraic}}
\label{app:boolean_proof}

\subsubsection{Proofs of Theorem~\ref{thm:algebraic} and Corollary~\ref{cor:pure-node}}\label{app:alge_and_cor}
\begin{lemma}\label{lem:boolean-column-set-equivalence}
Consider a Boolean factor model with
\(p_{[\aaa]}>0\) for every
\([\aaa]\in\mathcal E(\QQ)\). For any
\(\bQQ\in\{0,1\}^{J\times K}\), the following statements are
equivalent:

\textup{(i)} There exists \(\bpp\) such that
\[
\P(\RR=\rr\mid\QQ,\pp)
=
\P(\RR=\rr\mid\bQQ,\bpp),
\qquad \rr\in\{0,1\}^J.
\]

\textup{(ii)}
\[
\operatorname{Cols}\{\PPhi(\QQ)\}
\subseteq
\operatorname{Cols}\{\PPhi(\bQQ)\}.
\]
\end{lemma}

\begin{proof}
Under the Boolean factor model, for every
\(\rr\in\{0,1\}^J\),
\begin{equation}\label{eq:boolean-response-probability}
\P(\RR=\rr\mid\QQ,\pp)
=
\sum_{\aaa\in\mcr(\QQ)}
p_{[\aaa]}\,
\mathbf 1\left\{
\rr=\PPhi(\QQ)_{\Cdot,\aaa}
\right\}.
\end{equation}

Suppose first that \textup{(i)} holds. Fix
\(\aaa\in\mcr(\QQ)\) and take
\(\rr=\PPhi(\QQ)_{\Cdot,\aaa}\). Since the profiles in
\(\mcr(\QQ)\) index distinct columns,
\[
\P\left(
\RR=\PPhi(\QQ)_{\Cdot,\aaa}
\mid\QQ,\pp
\right)
=
p_{[\aaa]}
>0.
\]
Equality of the response distributions therefore gives
\[
\P\left(
\RR=\PPhi(\QQ)_{\Cdot,\aaa}
\mid\bQQ,\bpp
\right)>0.
\]
Applying \eqref{eq:boolean-response-probability} to
\((\bQQ,\bpp)\), there must exist
\(\bar\aaa\in\mcr(\bQQ)\) such that
\[
\PPhi(\bQQ)_{\Cdot,\bar\aaa}
=
\PPhi(\QQ)_{\Cdot,\aaa}.
\]
Since this holds for every \(\aaa\in\mcr(\QQ)\), \textup{(ii)}
follows.

Conversely, suppose that \textup{(ii)} holds. For each
\(\aaa\in\mcr(\QQ)\), choose a corresponding
\(\bar\aaa\in\mcr(\bQQ)\) such that
\[
\PPhi(\bQQ)_{\Cdot,\bar\aaa}
=
\PPhi(\QQ)_{\Cdot,\aaa}.
\]
The selected profiles \(\bar\aaa\) are distinct because the profiles
in \(\mcr(\QQ)\) index distinct columns of \(\PPhi(\QQ)\).
Set \(\bar p_{[\bar\aaa]}=p_{[\aaa]}\) for each selected
\(\bar\aaa\), and assign zero probability to every unselected
equivalence class of \(\bQQ\). Then, for every
\(\rr\in\{0,1\}^J\),
\[
\begin{aligned}
\P(\RR=\rr\mid\bQQ,\bpp)
&=
\sum_{\bar\aaa'\in\mcr(\bQQ)}
\bar p_{[\bar\aaa']}\,
\mathbf 1\left\{
\rr=\PPhi(\bQQ)_{\Cdot,\bar\aaa'}
\right\}\\
&=
\sum_{\aaa\in\mcr(\QQ)}
\bar p_{[\bar\aaa]}\,
\mathbf 1\left\{
\rr=\PPhi(\bQQ)_{\Cdot,\bar\aaa}
\right\}\\
&=
\sum_{\aaa\in\mcr(\QQ)}
p_{[\aaa]}\,
\mathbf 1\left\{
\rr=\PPhi(\QQ)_{\Cdot,\aaa}
\right\}\\
&=
\P(\RR=\rr\mid\QQ,\pp).
\end{aligned}
\]
Thus, \textup{(i)} holds.
\end{proof}

\begin{proof}[Proof of Theorem~\ref{thm:algebraic}]
By Lemma~\ref{lem:boolean-column-set-equivalence}, a matrix
\(\bQQ\nsim\QQ\) can generate the same response distribution as
\(\QQ\) for some \(\bpp\) if and only if
\[
\operatorname{Cols}\{\PPhi(\QQ)\}
\subseteq
\operatorname{Cols}\{\PPhi(\bQQ)\}.
\]
Therefore, \(\QQ\) is identifiable if and only if the algebraic
condition holds.

Once \(\QQ\) is identified, equality of the response distributions
and Equation~\eqref{eq:boolean-response-probability} identify every
grouped proportion \(p_{[\aaa]}\). Hence
\((\QQ,p_{[\Cdot]})\) is jointly identifiable.
\end{proof}

\begin{proposition}\label{prop:conjunctive-JK}
Under the conjunctive operator, let
\(\QQ\in\{0,1\}^{J\times K}\) with \(1\leq J\leq K\).
Then \(\QQ\) satisfies the algebraic condition if and only if
\(J=K\) and \(\QQ\sim\cI_K\).
\end{proposition}

\begin{proof}
Suppose first that \(J<K\). Let
\[
\HH=
\begin{pmatrix}
\QQ\\
\zero_{(K-J)\times K}
\end{pmatrix},
\qquad
\bQQ_0=
\left(
\cI_J,\zero_{J\times(K-J)}
\right),
\]
and
\[
\bQQ_1=
\left(
\cI_J,
\underbrace{\one_J,\ldots,\one_J}_{K-J}
\right).
\]
Then
\[
\QQ
=
\bQQ_0\odot\HH
=
\bQQ_1\odot\HH.
\]
Since \(\bQQ_0\nsim\bQQ_1\), at least one of
\(\bQQ_0\) and \(\bQQ_1\) is not equivalent to \(\QQ\).
Thus, \eqref{eq:boolean-identifiability} fails, and \(\QQ\) does not
satisfy the algebraic condition.

Now suppose that \(J=K\). Since
\[
\QQ=\cI_K\odot\QQ,
\]
the algebraic condition for \(\QQ\) implies \(\QQ\sim\cI_K\).

Conversely, suppose
\(\cI_K=\bQQ\odot\HH\). For every \(k\in[K]\),
\[
\ee_k
=
\bigvee_{\ell:h_{\ell k}=1}
\bQQ_{\Cdot,\ell}.
\]
Hence, for some \(\ell_k\),
\[
\zero
\neq
\bQQ_{\Cdot,\ell_k}
\preceq
\ee_k,
\]
and therefore \(\bQQ_{\Cdot,\ell_k}=\ee_k\). The indices
\(\ell_1,\ldots,\ell_K\) are distinct, so
\(\bQQ\sim\cI_K\). Thus \(\cI_K\), and hence every
\(\QQ\sim\cI_K\), satisfies the algebraic condition.
\end{proof}

\begin{proof}[Proof of Corollary~\ref{cor:pure-node}]
By duality, it suffices to consider the conjunctive operator.
Proposition~\ref{prop:conjunctive-JK} shows that \(\cI_K\)
satisfies the algebraic condition. The result follows from
Theorem~\ref{thm:sub}\textup{(ii)} and
Theorem~\ref{thm:algebraic}.
\end{proof}

\subsubsection{Proof of Theorem~\ref{thm:sub}}
The profiles \(\one-\ee_k\), \(k\in[K]\), each lack exactly one attribute. We
call them the \emph{near-top profiles}; the vectors
\(\PPhi(\QQ)_{\Cdot,\one-\ee_k}\), \(k\in[K]\), are the
\emph{near-top columns} of \(\PPhi(\QQ)\).

\begin{lemma}\label{lm:conjunctive-generation}
Under the conjunctive operator, let
\(\QQ\in\{0,1\}^{J\times K}\). For every
\(\aaa_1,\aaa_2\in\{0,1\}^K\),
\[
\PPhi(\QQ)_{\Cdot,\aaa_1\wedge\aaa_2}
=
\PPhi(\QQ)_{\Cdot,\aaa_1}
\wedge
\PPhi(\QQ)_{\Cdot,\aaa_2}.
\]
Moreover, for every \(\aaa\in\{0,1\}^K\),
\begin{equation}
\label{eq:conjunctive-generation}
\PPhi(\QQ)_{\Cdot,\aaa}
=
\bigwedge_{k:\alpha_k=0}
\PPhi(\QQ)_{\Cdot,\one-\ee_k}
=
\one-\bigvee_{k:\alpha_k=0}\QQ_{\Cdot,k}
=
\one-\QQ\odot(\one-\aaa),
\end{equation}
where the empty conjunction and disjunction are understood as \(\one\) and
\(\zero\), respectively. Consequently,
\[
\operatorname{Cols}\bigl(\PPhi(\QQ)\bigr)
=
\left\{
\bigwedge_{k\in S}
\PPhi(\QQ)_{\Cdot,\one-\ee_k}:
S\subseteq[K]
\right\},
\]
and
\[
\QQ_{\Cdot,k}
=
\one-\PPhi(\QQ)_{\Cdot,\one-\ee_k},
\qquad k\in[K].
\]
\end{lemma}

\begin{proof}
For every \(j\in[J]\),
\[
\PPhi(\QQ)_{j,\aaa}
=
\bigwedge_{k:\alpha_k=0}(1-q_{jk})
=
1-\bigvee_{k:\alpha_k=0}q_{jk}
=
1-\qq_j\odot(\one-\aaa).
\]
This proves \eqref{eq:conjunctive-generation}. Also,
\[
\{k:(\aaa_1\wedge\aaa_2)_k=0\}
=
\{k:\alpha_{1k}=0\}\cup\{k:\alpha_{2k}=0\},
\]
which gives the first identity. As \(\aaa\) ranges over \(\{0,1\}^K\), the
set \(\{k:\alpha_k=0\}\) ranges over all subsets of \([K]\); taking
\(\aaa=\one-\ee_k\) gives the final identity.
\end{proof}

\begin{lemma}\label{lem:boolean-factorization}
Under the conjunctive operator, let
\(\QQ,\bar{\QQ}\in\{0,1\}^{J\times K}\). The following statements are
equivalent:

\textup{(i)}
\(\operatorname{Cols}\bigl(\PPhi(\QQ)\bigr)\subseteq
\operatorname{Cols}\bigl(\PPhi(\bar{\QQ})\bigr)\);

\textup{(ii)}
\(\PPhi(\QQ)_{\Cdot,\one-\ee_k}\in
\operatorname{Cols}\bigl(\PPhi(\bar{\QQ})\bigr)\) for every \(k\in[K]\);

\textup{(iii)}
\(\QQ=\bar{\QQ}\odot\HH\) for some
\(\HH\in\{0,1\}^{K\times K}\).

\end{lemma}

\begin{proof}
The implication \textup{(i)} \(\Rightarrow\) \textup{(ii)} is immediate.
Suppose \textup{(ii)} holds. For each \(k\in[K]\), choose
\(\aaa^{(k)}\in\{0,1\}^K\) such that
\(\PPhi(\bar{\QQ})_{\Cdot,\aaa^{(k)}}
=
\PPhi(\QQ)_{\Cdot,\one-\ee_k}\),
and set \(\HH_{\Cdot,k}=\one-\aaa^{(k)}\). By
\eqref{eq:conjunctive-generation},
\(\bar{\QQ}\odot\HH_{\Cdot,k}=\QQ_{\Cdot,k}\) for every \(k\), so
\(\QQ=\bar{\QQ}\odot\HH\). Thus \textup{(iii)} holds.

Conversely, suppose \textup{(iii)} holds. For
\(\aaa\in\{0,1\}^K\), set
\(\bar\aaa=\one-\HH\odot(\one-\aaa)\). Then
\[
\PPhi(\QQ)_{\Cdot,\aaa}
=
\one-\bar{\QQ}\odot\HH\odot(\one-\aaa)
=
\PPhi(\bar{\QQ})_{\Cdot,\bar\aaa}
\]
by \eqref{eq:conjunctive-generation}. Hence \textup{(i)} holds.
\end{proof}

Combining Theorem~\ref{thm:algebraic} and
Lemma~\ref{lem:boolean-factorization}, under the conjunctive
operator, \(\QQ\) satisfies the algebraic condition if and only if every
\(\bar{\QQ}\in\{0,1\}^{J\times K}\) and
\(\HH\in\{0,1\}^{K\times K}\) satisfy
\begin{equation}
\label{eq:boolean-identifiability}
\QQ=\bar{\QQ}\odot\HH
\quad\Longrightarrow\quad
\bar{\QQ}\sim\QQ.
\end{equation}

\begin{lemma}\label{lm:onetwo}
Under the conjunctive operator, the following statements hold for
\(\QQ\in\{0,1\}^{J\times K}\), where \(J\geq1\).

\textup{(i)} If \(K=1\), then \(\QQ\) satisfies the algebraic
condition if and only if \(\QQ\neq\zero\).

\textup{(ii)} If \(K=2\), then \(\QQ\) satisfies the algebraic
condition if and only if it contains an identity submatrix
\(\cI_2\), up to a permutation of its columns.
\end{lemma}

\begin{proof}
We use \eqref{eq:boolean-identifiability}. If \(K=1\) and
\(\QQ\neq\zero\), the single entry of \(\HH\) in any factorization must be
one, so \(\bar{\QQ}=\QQ\). If \(\QQ=\zero\), then
\(\QQ=\one_J\odot\zero\) gives a non-equivalent factorization. This proves
\textup{(i)}.

Suppose \(K=2\) and \(\QQ\) contains \(\cI_2\). Then
\(\QQ_{\Cdot,1}\npreceq\QQ_{\Cdot,2}\) and
\(\QQ_{\Cdot,2}\npreceq\QQ_{\Cdot,1}\). In any factorization
\(\QQ=\bar{\QQ}\odot\HH\), each column of \(\HH\) is nonzero. If
\(h_{1k}=h_{2k}=1\) for some \(k\), then, with \(k'=3-k\),
\(\QQ_{\Cdot,k'}\preceq\QQ_{\Cdot,k}\), a contradiction. Thus each column
of \(\HH\) contains exactly one one. Moreover,
\(\HH_{\Cdot,1}=\HH_{\Cdot,2}\) would imply
\(\QQ_{\Cdot,1}=\QQ_{\Cdot,2}\). Hence \(\HH\) is a permutation matrix and
\(\bar{\QQ}\sim\QQ\).

Conversely, suppose \(\QQ\) does not contain \(\cI_2\). After possibly
permuting its columns, assume
\(\QQ_{\Cdot,1}\preceq\QQ_{\Cdot,2}\). If
\(\QQ_{\Cdot,1}\neq\zero\), let
\[
\bar{\QQ}
=
\left(
\QQ_{\Cdot,1},
\QQ_{\Cdot,2}\wedge(\one_J-\QQ_{\Cdot,1})
\right),
\qquad
\HH=
\begin{pmatrix}
1&1\\
0&1
\end{pmatrix}.
\]
Then \(\QQ=\bar{\QQ}\odot\HH\). Moreover,
\(\bar{\QQ}_{\Cdot,2}\wedge\QQ_{\Cdot,1}=\zero\), whereas
\(\QQ_{\Cdot,r}\wedge\QQ_{\Cdot,1}=\QQ_{\Cdot,1}\neq\zero\) for
\(r\in[2]\); hence \(\bar{\QQ}\nsim\QQ\). If
\(\QQ_{\Cdot,1}=\zero\), take
\[
\bar{\QQ}=(\one_J,\QQ_{\Cdot,2}),
\qquad
\HH=
\begin{pmatrix}
0&0\\
0&1
\end{pmatrix}.
\]
Again \(\QQ=\bar{\QQ}\odot\HH\). If \(\bar{\QQ}\sim\QQ\), then either
\(\one_J=\zero\) or \(\one_J=\QQ_{\Cdot,2}=\zero\), both impossible.
Thus \(\bar{\QQ}\nsim\QQ\), proving \textup{(ii)}.
\end{proof}

\begin{proof}[Proof of Theorem~\ref{thm:sub}]
We use \eqref{eq:boolean-identifiability} throughout. We first prove
\textup{(iii)}. After permuting columns, write \(\QQ=(\QQ',\QQ'')\), where
\(\QQ'\) does not satisfy the algebraic condition. Then
\(\QQ'=\bar{\QQ}'\odot\HH'\) for some
\(\bar{\QQ}'\nsim\QQ'\). Set
\[
\bar{\QQ}=(\bar{\QQ}',\QQ''),
\qquad
\HH=
\begin{pmatrix}
\HH'&\zero\\
\zero&\cI_{K-L}
\end{pmatrix},
\]
where \(L\) is the number of columns of \(\QQ'\). Then
\(\QQ=\bar{\QQ}\odot\HH\), and
\[
(\bar{\QQ}',\QQ'')\sim(\QQ',\QQ'')
\quad\Longrightarrow\quad
\bar{\QQ}'\sim\QQ'.
\]
Hence \(\bar{\QQ}\nsim\QQ\), so \(\QQ\) does not satisfy the algebraic
condition.

For \textup{(ii)}, after rearranging rows, write
\(\QQ=(\QQ'^{\top},\QQ''^{\top})^{\top}\), where \(\QQ'\) is a nonempty
row submatrix satisfying the algebraic condition. If \(\QQ=\bar{\QQ}\odot\HH\), restriction to these rows
gives \(\QQ'=\bar{\QQ}'\odot\HH\), and hence
\(\bar{\QQ}'\sim\QQ'\). After permuting the columns of \(\bar{\QQ}\) and
the corresponding rows of \(\HH\), assume \(\bar{\QQ}'=\QQ'\). By
\textup{(iii)} and Lemma~\ref{lm:onetwo},
\[
\QQ'_{\Cdot,k}\neq\zero,
\qquad
\QQ'_{\Cdot,\ell}\npreceq\QQ'_{\Cdot,k}
\quad (\ell\neq k).
\]
Since \(\QQ'=\QQ'\odot\HH\), the relation \(h_{\ell k}=1\) implies
\(\QQ'_{\Cdot,\ell}\preceq\QQ'_{\Cdot,k}\), and hence \(\ell=k\).
Moreover, \(\QQ'_{\Cdot,k}\neq\zero\) gives \(h_{kk}=1\). Thus
\(\HH=\cI_K\) and
\(\QQ=\bar{\QQ}\). Undoing the column permutation proves
\textup{(ii)}.

For \textup{(i)}, suppose that \(\QQ\) has at least one nonzero
row. Every nonzero row is a disjunction of basis rows: if a
nonzero row is not a basis row, its defining decomposition uses
rows of strictly smaller numbers of ones. Induction on the
number of ones therefore proves the assertion. In particular,
\(\Qbasis\) is nonempty. Since \(\Qbasis\) is a row submatrix
of \(\QQ\), \textup{(ii)} proves one direction of \textup{(i)}.
Conversely, suppose \(\Qbasis\) does not
satisfy the algebraic condition. Let
\(\qq_{b_1},\ldots,\qq_{b_{J_b}}\) be the rows of \(\QQ\) selected to form
\(\Qbasis\), and choose
\(\bar{\QQ}^{\mathrm{basis}}\nsim\Qbasis\) and \(\HH\) such that
\(\Qbasis=\bar{\QQ}^{\mathrm{basis}}\odot\HH\). For every \(j\in[J]\),
choose \(S_j\subseteq[J_b]\) such that
\[
\qq_j=\bigvee_{r\in S_j}\qq_{b_r},
\]
taking \(S_{b_r}=\{r\}\) for each selected basis row and
\(S_j=\varnothing\) whenever \(\qq_j=\zero\), with the empty
disjunction equal to \(\zero\). Define
\(\bar\qq_j=\bigvee_{r\in S_j}\bar\qq_r^{\mathrm{basis}}\). Then
\[
\bar\qq_j\odot\HH
=
\bigvee_{r\in S_j}
(\bar\qq_r^{\mathrm{basis}}\odot\HH)
=
\bigvee_{r\in S_j}\qq_{b_r}
=
\qq_j,
\]
so \(\QQ=\bar{\QQ}\odot\HH\). If
\(\bar{\QQ}\sim\QQ\), restricting to the selected basis rows would give
\(\bar{\QQ}^{\mathrm{basis}}\sim\Qbasis\), a contradiction. Hence \(\QQ\)
does not satisfy the algebraic condition, proving the converse.
\end{proof}
{\color{black}
\subsubsection{Three-column submatrices}

\begin{lemma}\label{lm:threefour}
Under the conjunctive operator, let
\(\QQ\in\{0,1\}^{J\times3}\).
Then \(\QQ\) satisfies the algebraic condition if and only if
it satisfies the two-column condition and contains a row
with exactly one 1.
Consequently, for \(K\geq3\), every three-column submatrix
of a matrix satisfying the algebraic condition contains
a row with exactly one 1.
\end{lemma}

\begin{proof}
We first prove necessity. Suppose that \(\QQ\) satisfies
the algebraic condition. By
Theorem~\ref{thm:sub}\textup{(iii)} and
Lemma~\ref{lm:onetwo}\textup{(ii)}, \(\QQ\) satisfies
the two-column condition.
If \(\QQ\) contains no row with exactly one one, its
possible rows are \(000,110,101,011,111\).
The two-column condition then requires all three rows
\(110,101,011\) to occur.

Set
\[
\HH=
\begin{pmatrix}
1&1&0\\
1&0&1\\
0&1&1
\end{pmatrix},
\]
and obtain \(\bar{\QQ}\) by replacing each row of \(\QQ\)
according to
\[
110\mapsto100,\qquad
101\mapsto010,\qquad
011\mapsto001,\qquad
111\mapsto111,\qquad
000\mapsto000.
\]
Then \(\QQ=\bar{\QQ}\odot\HH\).
The replacements of \(110,101,011\) change their row sums,
so \(\bar{\QQ}\nsim\QQ\), contradicting the algebraic
condition.

We next prove sufficiency. Suppose that \(\QQ\) satisfies
the two-column condition and contains a row with exactly
one one. After permuting columns, assume that this row is
\((1,0,0)\). Consider any Boolean factorization
\[
\QQ=\bar{\QQ}\odot\HH,
\qquad
\bar{\QQ}\in\{0,1\}^{J\times3},
\quad
\HH\in\{0,1\}^{3\times3}.
\]
Each row of \(\QQ\) is a logical \(\bigvee\) combination
of rows of \(\HH\). Thus the row \((1,0,0)\) forces
\(\HH\) to contain \((1,0,0)\).

The two-column condition for columns \(2\) and \(3\)
requires \(\QQ\) to contain a row with second entry one
and third entry zero, and another with second entry zero
and third entry one. Their representations require two
distinct rows of \(\HH\), neither equal to \((1,0,0)\).
Hence, after permuting the rows of \(\HH\) and the
corresponding columns of \(\bar{\QQ}\), we have
\[
\HH=
\begin{pmatrix}
1&0&0\\
a&1&0\\
b&0&1
\end{pmatrix},
\qquad a,b\in\{0,1\}.
\]
The two-column condition also requires a row of \(\QQ\)
with \(q_{j1}=0\) and \(q_{j2}=1\).
Representing this row requires the second row of \(\HH\),
so \(a=0\).
Likewise, a row with \(q_{j1}=0\) and \(q_{j3}=1\)
requires the third row of \(\HH\), so \(b=0\).
Therefore, \(\HH\) is a permutation matrix before the
rearrangement, giving \(\bar{\QQ}\sim\QQ\).
This proves the algebraic condition.

The final assertion follows from
Theorem~\ref{thm:sub}\textup{(iii)}.
\end{proof}

\subsubsection{Proof of Proposition~\ref{prop:sat-nonequivalence}}
\begin{proof}
By Lemma~\ref{lem:boolean-factorization}, $\QQ$ fails the algebraic condition if and only if
there exist $\bQQ\not\sim\QQ$ and
$H\in\{0,1\}^{K\times K}$ such that
\[
\QQ=\bQQ\odot H.
\]
Simultaneously permuting the columns of $\bQQ$ and the
corresponding rows of $H$ preserves this factorization
and the relation $\bQQ\not\sim\QQ$.
We may therefore assume that the columns of $\bQQ$ are
arranged in weakly non-increasing lexicographic order.
Since $\bQQ\not\sim\QQ$, we also have $\bQQ\ne\QQ$.

Conversely, suppose that $\QQ=\bQQ\odot H$, where
$\bQQ\ne\QQ$ and both matrices have their columns arranged
in weakly non-increasing lexicographic order.
Two matrices in this prescribed order are equivalent under
column permutation if and only if they are equal.
Hence $\bQQ\not\sim\QQ$, and Lemma~\ref{lem:boolean-factorization} implies that
$\QQ$ fails the algebraic condition.
\end{proof}

}

{\color{black}
\section{Proofs of results in Section~\ref{sec:prob}}\label{app:prob}
\subsection{Auxiliary lemmas}
We first record a lemma showing that, when the item parameters are known, the probabilistic Boolean factor model has the same identifiability conditions
for \((p_{[\Cdot]},\QQ)\) as the corresponding Boolean factor model.

To distinguish equivalence classes induced by different
$\QQ$-matrices, we denote the equivalence class of $\aaa$
under $\QQ$ by
\[
[\aaa]_{\QQ}
=
\left\{
\aaa'\in\{0,1\}^K:
\PPhi(\QQ)_{\Cdot,\aaa'}
=
\PPhi(\QQ)_{\Cdot,\aaa}
\right\}.
\]

\begin{lemma}\label{lem:knownitem}
Consider a probabilistic Boolean factor model with known
\((\mathbf s,\mathbf g)\) and \(p_{[\aaa]}>0\) for every
\([\aaa]\in\mathcal E(\QQ)\). Then
\((\QQ,p_{[\Cdot]})\) is jointly identifiable if and only if
it is jointly identifiable under the corresponding Boolean factor
model, or equivalently, if and only if \(\QQ\) satisfies the
algebraic condition~\eqref{eq:a_condition}.
\end{lemma}

\begin{proof}
For \(\vv\in\{0,1\}^J\), let
\[
\mu_{\vv}
=\sum_{\aaa:\,\PPhi_{\Cdot,\aaa}(\QQ)=\vv}p_{\aaa},
\qquad
\bar\mu_{\vv}
=\sum_{\aaa:\,\PPhi_{\Cdot,\aaa}(\bar{\QQ})=\vv}\bar p_{\aaa},
\]
and write \(\boldsymbol\mu=(\mu_{\vv})_{\vv}\) and
\(\bar{\boldsymbol\mu}=(\bar\mu_{\vv})_{\vv}\). Let
\[
M_j=\begin{pmatrix}1-g_j&s_j\\g_j&1-s_j\end{pmatrix},
\qquad j\in[J].
\]
Since the item parameters are known, local independence gives
\[
\begin{aligned}
&\P(\RR=\rr\mid\QQ,\mathbf s,\mathbf g,p_{[\Cdot]})
=\P(\RR=\rr\mid\bar{\QQ},\mathbf s,\mathbf g,
\bar p_{[\Cdot]}),\quad \rr\in\{0,1\}^J\\
&\qquad\Longleftrightarrow\quad
\left(\bigotimes_{j=1}^J M_j\right)\boldsymbol\mu
=\left(\bigotimes_{j=1}^J M_j\right)\bar{\boldsymbol\mu}
\quad\Longleftrightarrow\quad
\boldsymbol\mu=\bar{\boldsymbol\mu},
\end{aligned}
\]
where the last equivalence follows from
\(\det(M_j)=1-s_j-g_j>0\). The vectors
\(\boldsymbol\mu\) and \(\bar{\boldsymbol\mu}\) are exactly
the response distributions under the corresponding Boolean
factor models. This proves the first equivalence.

Moreover, positivity of all grouped proportions gives
\(\operatorname{supp}(\boldsymbol\mu)
=\operatorname{Cols}\{\PPhi(\QQ)\}\). Thus an alternative
\(\bar{\QQ}\) can give the same distribution if and only if
\[
\operatorname{Cols}\{\PPhi(\QQ)\}
\subseteq\operatorname{Cols}\{\PPhi(\bar{\QQ})\}.
\]
The algebraic condition is precisely the requirement that this
inclusion imply \(\bar{\QQ}\sim\QQ\). After aligning their
columns, \(\boldsymbol\mu=\bar{\boldsymbol\mu}\) also gives
\(p_{[\aaa]}=\bar p_{[\aaa]}\) for every
\([\aaa]\in\mathcal E(\QQ)\), proving the last equivalence.
\end{proof}

\begin{lemma}\label{lem:one-item-separation}
Suppose that
\(\varnothing\subsetneq\mca_j^+(\QQ)\subsetneq\{0,1\}^K\) and
\[
\P(\RR=\rr\mid\QQ,\mathbf s,\mathbf g,p_{[\Cdot]})
=
\P(\RR=\rr\mid\bar{\QQ},\bar{\mathbf s},\bar{\mathbf g},
\bar p_{[\Cdot]})
\quad\text{for all }\rr\in\{0,1\}^J.
\]
If \(p_{[\aaa]}>0\) for every
\([\aaa]\in\mathcal E(\QQ)\), then
\[
0<\sum_{\aaa\in\mca_j^+(\QQ)}p_{\aaa}<1,
\qquad
1-s_j>\bar g_j,
\qquad
1-\bar s_j>g_j.
\]
\end{lemma}

\begin{proof}
If \(\aaa'\in[\aaa]_{\QQ}\), then
\(\PPhi(\QQ)_{j,\aaa'}=\PPhi(\QQ)_{j,\aaa}\), and hence
\(\qq_j\ast\aaa'=\qq_j\ast\aaa\). Thus each equivalence class is
contained entirely in either \(\mca_j^+(\QQ)\) or
\(\mca_j^-(\QQ)\). Since both sets are nonempty, each contains at
least one equivalence class. Therefore, \(p_{[\aaa]}>0\) for every
\([\aaa]\in\mathcal E(\QQ)\) implies that both
\(\sum_{\aaa\in\mca_j^+(\QQ)}p_{\aaa}\) and
\(\sum_{\aaa\in\mca_j^-(\QQ)}p_{\aaa}\) are positive. Since the two
sets partition \(\{0,1\}^K\),
\[
0<\sum_{\aaa\in\mca_j^+(\QQ)}p_{\aaa}<1.
\]

Finally, we have
\[
\P(R_j=1\mid\QQ,\mathbf s,\mathbf g,p_{[\Cdot]})
=
g_j+(1-s_j-g_j)\sum_{\aaa\in\mca_j^+(\QQ)}p_{\aaa}.
\]
Since \(0<\sum_{\aaa\in\mca_j^+(\QQ)}p_{\aaa}<1\), this probability lies strictly between
\(g_j\) and \(1-s_j\). 
By equality of the response distributions,
\[
\bar g_j
\leq \P(R_j=1)
<1-s_j,
\qquad
g_j
<\P(R_j=1)
\leq 1-\bar s_j.
\]
Thus \(1-s_j>\bar g_j\) and \(1-\bar s_j>g_j\).
\end{proof}

\subsection{Proofs of results in Section~\ref{app:prob_12}}
\label{app:proofs_12}

\subsubsection{Proof of Theorem~\ref{thm:1P}}
\label{app:proof-1P}

\begin{lemma}\label{lem:one-param-error-identifiability}
Consider the one-parameter Boolean factor model with \(J>K\). If
\[
\P(\RR=\rr\mid\QQ,g,p_{[\Cdot]})
=
\P(\RR=\rr\mid\bar{\QQ},\bar g,\bar p_{[\Cdot]}),
\qquad \rr\in\{0,1\}^J,
\]
then \(g=\bar g\).
\end{lemma}

\begin{proof}
For \(\vv\in\{0,1\}^J\), let
\[
\mu_{\vv}
=
\sum_{\aaa:\,\PPhi_{\Cdot,\aaa}(\QQ)=\vv}p_{\aaa},
\qquad
\bar\mu_{\vv}
=
\sum_{\aaa:\,\PPhi_{\Cdot,\aaa}(\bar{\QQ})=\vv}\bar p_{\aaa},
\]
and write
\(\boldsymbol\mu=(\mu_{\vv})_{\vv}\) and
\(\bar{\boldsymbol\mu}=(\bar\mu_{\vv})_{\vv}\). Let
\[
A(u)=
\begin{pmatrix}
1-u&u\\
u&1-u
\end{pmatrix}.
\]
By local independence, equality of the response distributions gives
\[
A(g)^{\otimes J}\boldsymbol\mu
=
A(\bar g)^{\otimes J}\bar{\boldsymbol\mu}.
\]

Suppose first that \(g>\bar g\). Let
\[
t=\frac{g-\bar g}{1-2\bar g}\in(0,1/2).
\]
Since \(A(a)A(b)=A(a+b-2ab)\), we have
\(A(g)=A(\bar g)A(t)\). Therefore,
\[
\bar{\boldsymbol\mu}
=
A(t)^{\otimes J}\boldsymbol\mu.
\]
Every entry of \(A(t)^{\otimes J}\) is positive, so every entry of
\(\bar{\boldsymbol\mu}\) is positive. Hence
\[
2^J
=
\bigl|\operatorname{supp}(\bar{\boldsymbol\mu})\bigr|
\leq
\left|\operatorname{Cols}\{\PPhi(\bar{\QQ})\}\right|
\leq
2^K,
\]
contradicting \(J>K\).

If \(g<\bar g\), let
\[
t=\frac{\bar g-g}{1-2g}\in(0,1/2).
\]
Then \(A(\bar g)=A(g)A(t)\), and therefore
\[
\boldsymbol\mu
=
A(t)^{\otimes J}\bar{\boldsymbol\mu}.
\]
Thus every entry of \(\boldsymbol\mu\) is positive, which gives
\[
2^J
=
\bigl|\operatorname{supp}(\boldsymbol\mu)\bigr|
\leq
\left|\operatorname{Cols}\{\PPhi(\QQ)\}\right|
\leq
2^K,
\]
again contradicting \(J>K\). Hence \(g=\bar g\).
\end{proof}

% Revised proof of Theorem 1P, v1 (2026-09-20).
% For the theorem stated with J >= 2 and both identifiability targets.
% Uses the manuscript notation and the existing error-identifiability lemma.

\begin{proof}[Proof of Theorem~\ref{thm:1P}]
We first prove sufficiency. Suppose that \(\QQ\) satisfies the
algebraic condition and \(\lvert\mathcal E(\QQ)\rvert<2^J\).
Let \((\bar{\QQ},\bar g,\bar p_{[\Cdot]})\) give the same
response distribution. Use the notation in the proof of
Lemma~\ref{lem:one-param-error-identifiability}. Since
\(p_{[\aaa]}>0\) for every \([\aaa]\in\mathcal E(\QQ)\),
\[
\operatorname{supp}(\boldsymbol\mu)
=\operatorname{Cols}\{\PPhi(\QQ)\},
\qquad
\bigl|\operatorname{supp}(\boldsymbol\mu)\bigr|
=\lvert\mathcal E(\QQ)\rvert<2^J.
\]

If \(J>K\), Lemma~\ref{lem:one-param-error-identifiability}
gives \(g=\bar g\). For \(J\leq K\), we use the two calculations in the proof of
Lemma~\ref{lem:one-param-error-identifiability}. If \(g<\bar g\), then
\[
\boldsymbol\mu
=A\left(\frac{\bar g-g}{1-2g}\right)^{\otimes J}
\bar{\boldsymbol\mu},
\]
so \(\mu_{\vv}>0\) for every \(\vv\in\{0,1\}^J\), a
contradiction. If \(g>\bar g\), then
\[
\bar{\boldsymbol\mu}
=A\left(\frac{g-\bar g}{1-2\bar g}\right)^{\otimes J}
\boldsymbol\mu,
\]
so \(\bar\mu_{\vv}>0\) for every \(\vv\in\{0,1\}^J\).
Consequently,
\[
\operatorname{Cols}\{\PPhi(\QQ)\}
\subseteq\{0,1\}^J
=\operatorname{Cols}\{\PPhi(\bar{\QQ})\}.
\]
The algebraic condition gives \(\bar{\QQ}\sim\QQ\). By
permutation invariance,
\(\operatorname{Cols}\{\PPhi(\bar{\QQ})\}
=\operatorname{Cols}\{\PPhi(\QQ)\}\), again contradicting
\(\lvert\mathcal E(\QQ)\rvert<2^J\). Hence \(g=\bar g\).

Since \(\det A(g)=1-2g>0\),
\[
A(g)^{\otimes J}\boldsymbol\mu
=A(g)^{\otimes J}\bar{\boldsymbol\mu}
\quad\Longrightarrow\quad
\boldsymbol\mu=\bar{\boldsymbol\mu}.
\]
It follows that
\[
\operatorname{Cols}\{\PPhi(\QQ)\}
=\operatorname{supp}(\boldsymbol\mu)
=\operatorname{supp}(\bar{\boldsymbol\mu})
\subseteq\operatorname{Cols}\{\PPhi(\bar{\QQ})\}.
\]
The algebraic condition gives \(\bar{\QQ}\sim\QQ\). After
aligning their columns, \(\boldsymbol\mu=\bar{\boldsymbol\mu}\)
gives \(p_{[\aaa]}=\bar p_{[\aaa]}\) for every
\([\aaa]\in\mathcal E(\QQ)\). This proves joint
identifiability. Since this holds for every \(\TTT\) with
positive grouped proportions, \(\QQ\) is identifiable.

We next prove necessity. If the algebraic condition fails, there
exists \(\bar{\QQ}\nsim\QQ\) such that
\[
\operatorname{Cols}\{\PPhi(\QQ)\}
\subseteq\operatorname{Cols}\{\PPhi(\bar{\QQ})\}.
\]
Choose alternative grouped proportions giving
\(\bar{\boldsymbol\mu}=\boldsymbol\mu\), and set
\(\bar g=g\). Then
\[
A(\bar g)^{\otimes J}\bar{\boldsymbol\mu}
=A(g)^{\otimes J}\boldsymbol\mu,
\]
so \(\QQ\) is not identifiable and \((\QQ,\TTT)\) is not
jointly identifiable.

If \(\lvert\mathcal E(\QQ)\rvert=2^J\), then
\(\mu_{\vv}>0\) for every \(\vv\in\{0,1\}^J\). Keep
\(\bar{\QQ}=\QQ\), and choose \(g<\bar g<1/2\)
sufficiently close to \(g\). Define
\[
\bar{\boldsymbol\mu}
=\{A(\bar g)^{-1}A(g)\}^{\otimes J}\boldsymbol\mu.
\]
As \(\bar g\to g\),
\(\bar{\boldsymbol\mu}\to\boldsymbol\mu\). Thus
\(\bar\mu_{\vv}>0\) for every \(\vv\) when
\(\bar g-g\) is sufficiently small. The columns of
\(A(\bar g)^{-1}A(g)\) sum to one, so
\(\sum_{\vv}\bar\mu_{\vv}=\sum_{\vv}\mu_{\vv}=1\).
These probabilities define alternative grouped proportions under
\(\bar{\QQ}=\QQ\), and
\[
A(\bar g)^{\otimes J}\bar{\boldsymbol\mu}
=A(g)^{\otimes J}\boldsymbol\mu.
\]
Since \(\bar g\neq g\), joint identifiability fails.

To prove necessity for identifiability of \(\QQ\), suppose
again that \(\lvert\mathcal E(\QQ)\rvert=2^J\). We now
construct a true parameter specification and an alternative
with \(\bar{\QQ}\nsim\QQ\).
Let every row of \(\bar{\QQ}\) equal \(\qq_1\).
Since \(\operatorname{Cols}\{\PPhi(\QQ)\}=\{0,1\}^J\),
item \(1\) has both ideal responses, and therefore
\[
\operatorname{Cols}\{\PPhi(\bar{\QQ})\}
=\{\zero_J,\one_J\}.
\]
Moreover, no two rows of \(\QQ\) can coincide, since that
would force the corresponding coordinates of every ideal
response pattern to agree. As \(J\geq2\), we have
\(\bar{\QQ}\nsim\QQ\).

Choose alternative grouped proportions giving
\[
\bar\mu_{\zero_J}=\bar\mu_{\one_J}=\frac12,
\qquad
\bar\mu_{\vv}=0,
\quad \vv\notin\{\zero_J,\one_J\}.
\]
Choose \(0\leq g<\bar g<1/2\), and let
\[
t=\frac{\bar g-g}{1-2g}\in(0,1/2),
\qquad
\boldsymbol\mu=A(t)^{\otimes J}\bar{\boldsymbol\mu}.
\]
Every entry of \(A(t)^{\otimes J}\) is positive, and its
columns sum to one. Hence
\[
\mu_{\vv}>0\quad\text{for every }\vv\in\{0,1\}^J,
\qquad
\sum_{\vv}\mu_{\vv}=1.
\]
Since \(\operatorname{Cols}\{\PPhi(\QQ)\}=\{0,1\}^J\),
these probabilities define positive true grouped proportions by
\[
p_{[\aaa]}=\mu_{\PPhi_{\Cdot,\aaa}(\QQ)},
\qquad [\aaa]\in\mathcal E(\QQ).
\]
The identity \(A(g)A(t)=A(\bar g)\), used in the proof of
Lemma~\ref{lem:one-param-error-identifiability}, gives
\[
A(g)^{\otimes J}\boldsymbol\mu
=\{A(g)A(t)\}^{\otimes J}\bar{\boldsymbol\mu}
=A(\bar g)^{\otimes J}\bar{\boldsymbol\mu}.
\]
Thus the response distributions coincide, although
\(\bar{\QQ}\nsim\QQ\). This proves that \(\QQ\) is not
identifiable.

Finally, consider the conjunctive operator. If
\(J>K\), then
\(\lvert\mathcal E(\QQ)\rvert\leq2^K<2^J\).
If \(J\leq K\), the matrix
\(\bar{\QQ}=(\cI_J\ \mathbf0)\) satisfies
\[
\operatorname{Cols}\{\PPhi(\bar{\QQ})\}
=\{0,1\}^J.
\]
Thus the algebraic condition for \(\QQ\) implies
\(\QQ\sim\bar{\QQ}\), and hence
\(\lvert\mathcal E(\QQ)\rvert=2^J\). Therefore, under
the conjunctive operator, the algebraic condition and
\(\lvert\mathcal E(\QQ)\rvert<2^J\) hold if and only if
the algebraic condition and \(J>K\) hold.
\end{proof}

\subsubsection{Proof of Theorem~\ref{thm:2P}}
\label{app:proof-2P}

\begin{lemma}\label{lem:two-param-support}
Let \(A\) be an invertible \(2\times2\) matrix with no zero entries.
For every nonzero \(x\in\mathbb R^{2^n}\),
\[
\bigl|\operatorname{supp}(x)\bigr|
\bigl|\operatorname{supp}(A^{\otimes n}x)\bigr|
\geq2^n.
\]
\end{lemma}

\begin{proof}
We use induction on \(n\). If \(n=1\) and \(x\) has one
nonzero entry, then \(Ax\) has two nonzero entries. If \(x\)
has two nonzero entries, then \(Ax\neq0\) because \(A\) is
invertible. Thus the result holds for \(n=1\).

Suppose the result holds for \(n-1\). Write
\[
x=\begin{pmatrix}x_0\\x_1\end{pmatrix},
\qquad
u=A^{\otimes(n-1)}x_0,
\qquad
v=A^{\otimes(n-1)}x_1,
\]
so that
\[
y:=A^{\otimes n}x
=\begin{pmatrix}a_{00}u+a_{01}v\\a_{10}u+a_{11}v\end{pmatrix}.
\]
If \(x_1=0\), then \(v=0\) and
\[
\bigl|\operatorname{supp}(x)\bigr|
\bigl|\operatorname{supp}(y)\bigr|
=2\bigl|\operatorname{supp}(x_0)\bigr|
\bigl|\operatorname{supp}(u)\bigr|
\geq2^n.
\]
If \(x_0=0\), then \(u=0\) and
\[
\bigl|\operatorname{supp}(x)\bigr|
\bigl|\operatorname{supp}(y)\bigr|
=2\bigl|\operatorname{supp}(x_1)\bigr|
\bigl|\operatorname{supp}(v)\bigr|
\geq2^n.
\]

Now suppose \(x_0\neq0\) and \(x_1\neq0\). Since \(A\)
is invertible, for each index \(z\),
\[
\begin{pmatrix}a_{00}u_z+a_{01}v_z\\a_{10}u_z+a_{11}v_z\end{pmatrix}
=0
\quad\Longleftrightarrow\quad
u_z=v_z=0.
\]
Consequently,
\[
\bigl|\operatorname{supp}(y)\bigr|
\geq
\bigl|\operatorname{supp}(u)\cup\operatorname{supp}(v)\bigr|.
\]
If \(\lvert\operatorname{supp}(x_0)\rvert
\leq\lvert\operatorname{supp}(x_1)\rvert\), then
\[
\begin{aligned}
\bigl|\operatorname{supp}(x)\bigr|
\bigl|\operatorname{supp}(y)\bigr|
&\geq
\bigl\{\lvert\operatorname{supp}(x_0)\rvert
+\lvert\operatorname{supp}(x_1)\rvert\bigr\}
\bigl|\operatorname{supp}(u)\bigr|\\
&\geq2\bigl|\operatorname{supp}(x_0)\bigr|
\bigl|\operatorname{supp}(u)\bigr|
\geq2^n.
\end{aligned}
\]
If \(\lvert\operatorname{supp}(x_1)\rvert
<\lvert\operatorname{supp}(x_0)\rvert\), then
\[
\begin{aligned}
\bigl|\operatorname{supp}(x)\bigr|
\bigl|\operatorname{supp}(y)\bigr|
&\geq
\bigl\{\lvert\operatorname{supp}(x_0)\rvert
+\lvert\operatorname{supp}(x_1)\rvert\bigr\}
\bigl|\operatorname{supp}(v)\bigr|\\
&\geq2\bigl|\operatorname{supp}(x_1)\bigr|
\bigl|\operatorname{supp}(v)\bigr|
\geq2^n.
\end{aligned}
\]
This completes the induction.
\end{proof}

\begin{lemma}\label{lem:two-param-error-identifiability}
Consider the two-parameter Boolean factor model with
\(2^J>2^K\lvert\mathcal E(\QQ)\rvert\).
Suppose that
\[
\bigcap_{j=1}^J\mca_j^+(\QQ)\neq\varnothing,
\qquad
\bigcap_{j=1}^J\mca_j^-(\QQ)\neq\varnothing,
\]
and \(p_{[\aaa]}>0\) for every \([\aaa]\in\mathcal E(\QQ)\).
If
\[
\P(\RR=\rr\mid\QQ,s,g,p_{[\Cdot]})
=
\P(\RR=\rr\mid\bar{\QQ},\bar s,\bar g,\bar p_{[\Cdot]}),
\qquad \rr\in\{0,1\}^J,
\]
then \(s=\bar s\) and \(g=\bar g\).
\end{lemma}

\begin{proof}
For \(\vv\in\{0,1\}^J\), let
\[
\mu_{\vv}
=\sum_{\aaa:\,\PPhi_{\Cdot,\aaa}(\QQ)=\vv}p_{\aaa},
\qquad
\bar\mu_{\vv}
=\sum_{\aaa:\,\PPhi_{\Cdot,\aaa}(\bar{\QQ})=\vv}\bar p_{\aaa},
\]
and write
\(\boldsymbol\mu=(\mu_{\vv})_{\vv}\) and
\(\bar{\boldsymbol\mu}=(\bar\mu_{\vv})_{\vv}\).
The assumed bound implies \(J>K\). Full grouped positivity and
the two nonempty intersections give
\[
\mu_{\zero_J}>0,
\qquad
\mu_{\one_J}>0,
\qquad
\bigl|\operatorname{supp}(\boldsymbol\mu)\bigr|=\lvert\mathcal E(\QQ)\rvert,
\qquad
\bigl|\operatorname{supp}(\bar{\boldsymbol\mu})\bigr|\leq2^K.
\]
Let
\[
M=\begin{pmatrix}1-g&s\\g&1-s\end{pmatrix},
\qquad
\bar M=\begin{pmatrix}1-\bar g&\bar s\\\bar g&1-\bar s\end{pmatrix}.
\]
By local independence and equality of the response distributions,
\[
M^{\otimes J}\boldsymbol\mu
=\bar M^{\otimes J}\bar{\boldsymbol\mu}.
\]
Since \(\det(M)=1-s-g>0\) and
\(\det(\bar M)=1-\bar s-\bar g>0\),
\[
\bar{\boldsymbol\mu}=B^{\otimes J}\boldsymbol\mu,
\qquad
B=\bar M^{-1}M
=\frac1{1-\bar s-\bar g}
\begin{pmatrix}
1-\bar s-g&s-\bar s\\
g-\bar g&1-s-\bar g
\end{pmatrix}.
\]
By Lemma~\ref{lem:one-item-separation},
\(1-\bar s>g\) and \(1-s>\bar g\), so both diagonal
entries of \(B\) are positive. If \(s\neq\bar s\) and
\(g\neq\bar g\), then \(B\) has no zero entries.
Lemma~\ref{lem:two-param-support} gives
\[
2^J
\leq
\bigl|\operatorname{supp}(\boldsymbol\mu)\bigr|
\bigl|\operatorname{supp}(\bar{\boldsymbol\mu})\bigr|
\leq2^K\lvert\mathcal E(\QQ)\rvert<2^J,
\]
a contradiction. Therefore,
\(s=\bar s\) or \(g=\bar g\).

Suppose first that \(g=\bar g\). If \(s>\bar s\), then
\[
B=
\begin{pmatrix}
1&\dfrac{s-\bar s}{1-\bar s-g}\\[1ex]
0&\dfrac{1-s-g}{1-\bar s-g}
\end{pmatrix}.
\]
All entries of \(B\) are nonnegative, and both entries of
its second column are positive. Thus, for every
\(\vv\in\{0,1\}^J\),
\[
\bar\mu_{\vv}
\geq\mu_{\one_J}\prod_{j=1}^J B_{v_j,1}>0.
\]
Hence
\(\lvert\operatorname{supp}(\bar{\boldsymbol\mu})\rvert
=2^J>2^K\), a contradiction.

If \(s<\bar s\), then
\[
B^{-1}=
\begin{pmatrix}
1&\dfrac{\bar s-s}{1-s-g}\\[1ex]
0&\dfrac{1-\bar s-g}{1-s-g}
\end{pmatrix},
\qquad
\boldsymbol\mu=(B^{-1})^{\otimes J}\bar{\boldsymbol\mu}.
\]
The row indexed by \(\one_J\) gives
\[
0<\mu_{\one_J}
=\left(\frac{1-\bar s-g}{1-s-g}\right)^J
\bar\mu_{\one_J},
\]
so \(\bar\mu_{\one_J}>0\). For every
\(\vv\in\{0,1\}^J\),
\[
\mu_{\vv}
\geq\bar\mu_{\one_J}\prod_{j=1}^J(B^{-1})_{v_j,1}>0.
\]
Hence \(\lvert\operatorname{supp}(\boldsymbol\mu)\rvert
=2^J>2^K\), again a contradiction. Therefore,
\(g=\bar g\) implies \(s=\bar s\).

Now suppose that \(s=\bar s\). If \(g>\bar g\), then
\[
B=
\begin{pmatrix}
\dfrac{1-s-g}{1-s-\bar g}&0\\[1ex]
\dfrac{g-\bar g}{1-s-\bar g}&1
\end{pmatrix}.
\]
All entries of \(B\) are nonnegative, and both entries of
its first column are positive. Thus, for every
\(\vv\in\{0,1\}^J\),
\[
\bar\mu_{\vv}
\geq\mu_{\zero_J}\prod_{j=1}^J B_{v_j,0}>0.
\]
Hence
\(\lvert\operatorname{supp}(\bar{\boldsymbol\mu})\rvert
=2^J>2^K\), a contradiction.

If \(g<\bar g\), then
\[
B^{-1}=
\begin{pmatrix}
\dfrac{1-s-\bar g}{1-s-g}&0\\[1ex]
\dfrac{\bar g-g}{1-s-g}&1
\end{pmatrix},
\qquad
\boldsymbol\mu=(B^{-1})^{\otimes J}\bar{\boldsymbol\mu}.
\]
The row indexed by \(\zero_J\) gives
\[
0<\mu_{\zero_J}
=\left(\frac{1-s-\bar g}{1-s-g}\right)^J
\bar\mu_{\zero_J},
\]
so \(\bar\mu_{\zero_J}>0\). For every
\(\vv\in\{0,1\}^J\),
\[
\mu_{\vv}
\geq\bar\mu_{\zero_J}\prod_{j=1}^J(B^{-1})_{v_j,0}>0.
\]
Hence \(\lvert\operatorname{supp}(\boldsymbol\mu)\rvert
=2^J>2^K\), again a contradiction. Therefore,
\(s=\bar s\) implies \(g=\bar g\).
We conclude that \(s=\bar s\) and \(g=\bar g\).
\end{proof}

\begin{proof}[Proof of Theorem~\ref{thm:2P}]
Fix \(\TTT=(s,g,p_{[\Cdot]})\) with \(p_{[\aaa]}>0\) for
every \([\aaa]\in\mathcal E(\QQ)\). Suppose that
\[
\P(\RR=\rr\mid\QQ,s,g,p_{[\Cdot]})
=
\P(\RR=\rr\mid\bar{\QQ},\bar s,\bar g,\bar p_{[\Cdot]}),
\qquad \rr\in\{0,1\}^J.
\]
By Lemma~\ref{lem:two-param-error-identifiability},
\(s=\bar s\) and \(g=\bar g\). Use
\(M\), \(\boldsymbol\mu\), and \(\bar{\boldsymbol\mu}\)
as defined in the proof of
Lemma~\ref{lem:two-param-error-identifiability}. Then
\[
M^{\otimes J}\boldsymbol\mu
=M^{\otimes J}\bar{\boldsymbol\mu},
\qquad
\det(M)=1-s-g>0,
\]
so \(\boldsymbol\mu=\bar{\boldsymbol\mu}\). Therefore,
\[
\operatorname{Cols}\{\PPhi(\QQ)\}
=\operatorname{supp}(\boldsymbol\mu)
=\operatorname{supp}(\bar{\boldsymbol\mu})
\subseteq\operatorname{Cols}\{\PPhi(\bar{\QQ})\}.
\]
The algebraic condition~\eqref{eq:a_condition} gives
\(\bar{\QQ}\sim\QQ\). After a common relabeling of the
latent variables, assume that \(\bar{\QQ}=\QQ\). For every
\([\aaa]\in\mathcal E(\QQ)\),
\[
p_{[\aaa]}
=\mu_{\PPhi_{\Cdot,\aaa}(\QQ)}
=\bar\mu_{\PPhi_{\Cdot,\aaa}(\QQ)}
=\bar p_{[\aaa]}.
\]
Thus \((\QQ,\TTT)\) is jointly identifiable. Since \(\TTT\)
was arbitrary subject to positive grouped proportions, \(\QQ\)
is identifiable.
\end{proof}

\subsubsection{Proof of Proposition~\ref{prop:two-param-conjunctive}}
\label{app:proof-two-param-conjunctive}

\begin{lemma}\label{lem:two-param-support-equality}
Let \(A\) be an invertible \(2\times2\) matrix with no zero entries,
and let \(y=A^{\otimes n}x\), where \(x\neq0\). Suppose that
\[
\bigl|\operatorname{supp}(x)\bigr|
\bigl|\operatorname{supp}(y)\bigr|=2^n.
\]
Fix \(j\in[n]\). If \(\operatorname{supp}(x)\) contains vectors
whose \(j\)th coordinates are zero and one, then
\[
\bigl|\{\vv\in\operatorname{supp}(x):v_j=0\}\bigr|
=\bigl|\{\vv\in\operatorname{supp}(x):v_j=1\}\bigr|.
\]
Moreover, \(\operatorname{supp}(y)\) contains no two vectors
that differ only in coordinate \(j\).
\end{lemma}

\begin{proof}
It suffices to take \(j=1\). Write
\[
x=\begin{pmatrix}x_0\\x_1\end{pmatrix},
\qquad
u=A^{\otimes(n-1)}x_0,
\qquad
v=A^{\otimes(n-1)}x_1.
\]
Both \(x_0\) and \(x_1\) are nonzero. Suppose first that
\(\lvert\operatorname{supp}(x_0)\rvert
\leq\lvert\operatorname{supp}(x_1)\rvert\). As in the proof of
Lemma~\ref{lem:two-param-support},
\[
\begin{aligned}
2^n
&=\bigl\{\lvert\operatorname{supp}(x_0)\rvert
       +\lvert\operatorname{supp}(x_1)\rvert\bigr\}
       \lvert\operatorname{supp}(y)\rvert\\
&\geq\bigl\{\lvert\operatorname{supp}(x_0)\rvert
       +\lvert\operatorname{supp}(x_1)\rvert\bigr\}
       \bigl|\operatorname{supp}(u)\cup\operatorname{supp}(v)\bigr|\\
&\geq2\lvert\operatorname{supp}(x_0)\rvert
       \lvert\operatorname{supp}(u)\rvert
\geq2^n.
\end{aligned}
\]
For \(n=1\), the last inequality holds because \(x_0\) and
\(u=x_0\) are nonzero scalars; for \(n>1\), it follows from
Lemma~\ref{lem:two-param-support}. Thus every inequality is an
equality, and
\[
\lvert\operatorname{supp}(x_0)\rvert
=\lvert\operatorname{supp}(x_1)\rvert,
\qquad
\lvert\operatorname{supp}(y)\rvert
=\bigl|\operatorname{supp}(u)\cup\operatorname{supp}(v)\bigr|.
\]
For each index \(z\) in this union, invertibility of \(A\)
implies that at least one of
\[
y_{0,z}=a_{00}u_z+a_{01}v_z,
\qquad
y_{1,z}=a_{10}u_z+a_{11}v_z
\]
is nonzero. The last equality therefore implies that exactly one
is nonzero. This proves the second conclusion. If
\(\lvert\operatorname{supp}(x_1)\rvert
<\lvert\operatorname{supp}(x_0)\rvert\), interchanging the two
slices in the inequalities gives the same conclusions.
\end{proof}

\begin{proof}[Proof of Proposition~\ref{prop:two-param-conjunctive}]
We first prove sufficiency. Suppose that \(\QQ\) satisfies the
algebraic condition and that
\[
\P(\RR=\rr\mid\QQ,s,g,p_{[\Cdot]})
=\P(\RR=\rr\mid\bar{\QQ},\bar s,\bar g,\bar p_{[\Cdot]}),
\qquad \rr\in\{0,1\}^J.
\]
Use the notation \(\boldsymbol\mu\),
\(\bar{\boldsymbol\mu}\), and \(B=\bar M^{-1}M\) from the
proof of Lemma~\ref{lem:two-param-error-identifiability}.
Under the conjunctive operator, the standing nondegeneracy
convention and positive grouped proportions give
\(\mu_{\zero_J}>0\) and \(\mu_{\one_J}>0\).
Since \(J=2K>K\), the triangular-matrix arguments in the proof of
Lemma~\ref{lem:two-param-error-identifiability} show that
\(s=\bar s\) if and only if \(g=\bar g\); those arguments
require only \(J>K\).

If \(s=\bar s\) and \(g=\bar g\), then
\(\boldsymbol\mu=\bar{\boldsymbol\mu}\). Consequently,
\[
\operatorname{Cols}\{\PPhi(\QQ)\}
=\operatorname{supp}(\boldsymbol\mu)
=\operatorname{supp}(\bar{\boldsymbol\mu})
\subseteq\operatorname{Cols}\{\PPhi(\bar{\QQ})\}.
\]
The algebraic condition gives \(\bar{\QQ}\sim\QQ\), and,
after aligning their columns,
\(p_{[\aaa]}=\bar p_{[\aaa]}\) for every
\([\aaa]\in\mathcal E(\QQ)\).

It remains to consider \(s\neq\bar s\) and
\(g\neq\bar g\). In this case, \(B\) and \(B^{-1}\)
have no zero entries. By Lemma~\ref{lem:two-param-support},
\[
2^{2K}=2^J
\leq\bigl|\operatorname{supp}(\boldsymbol\mu)\bigr|
     \bigl|\operatorname{supp}(\bar{\boldsymbol\mu})\bigr|
\leq2^{2K}.
\]
Therefore,
\[
\bigl|\operatorname{supp}(\boldsymbol\mu)\bigr|
=\bigl|\operatorname{supp}(\bar{\boldsymbol\mu})\bigr|
=2^K.
\]
Both \(\PPhi(\QQ)\) and \(\PPhi(\bar{\QQ})\) thus have
\(2^K\) distinct columns, and \(p_{\aaa}>0\) and
\(\bar p_{\aaa}>0\) for every \(\aaa\in\{0,1\}^K\).

Since \(\mu_{\zero_J}>0\) and \(\mu_{\one_J}>0\),
Lemma~\ref{lem:two-param-support-equality}, applied to
\(\bar{\boldsymbol\mu}=B^{\otimes J}\boldsymbol\mu\),
shows that \(\operatorname{supp}(\bar{\boldsymbol\mu})\)
contains no two vectors that differ in exactly one coordinate.
For any nonzero row \(\bar{\qq}_j\), both values of its ideal
response occur in this support. Applying the same lemma to
\(\boldsymbol\mu=(B^{-1})^{\otimes J}\bar{\boldsymbol\mu}\)
gives
\[
\begin{aligned}
2^{K-1}
&=\bigl|\{\vv\in\operatorname{supp}(\bar{\boldsymbol\mu}):v_j=1\}\bigr|\\
&=\bigl|\{\aaa\in\{0,1\}^K:\aaa\succeq\bar{\qq}_j\}\bigr|
=2^{K-\sum_{k=1}^K\bar q_{jk}}.
\end{aligned}
\]
Thus every nonzero row of \(\bar{\QQ}\) is a pure row.
Every attribute occurs in at least one row, since otherwise
\(\PPhi(\bar{\QQ})\) would have fewer than \(2^K\)
distinct columns. If an attribute occurred in exactly one row,
the ideal-response vectors at \(\one\) and
\(\one-\ee_k\) would differ only in that row, a contradiction.
Hence every attribute occurs in at least two rows. Since
\(J=2K\), every row is pure and every attribute occurs exactly
twice.

After rearranging the rows and relabeling the alternative
attributes, we may therefore take
\(\bar{\QQ}=(\cI_K^\top\ \cI_K^\top)^\top\).
For every \(\vv\in\{0,1\}^K\), the equation
\(\boldsymbol\mu=(B^{-1})^{\otimes 2K}\bar{\boldsymbol\mu}\)
gives
\[
\mu_{(\vv,\vv)}
=\sum_{\aaa\in\{0,1\}^K}
  \bar p_{\aaa}\prod_{k=1}^K
  \bigl\{(B^{-1})_{v_k,\alpha_k}\bigr\}^{2}>0.
\]
There are \(2^K\) such vectors \((\vv,\vv)\). Consequently,
\[
\operatorname{Cols}\{\PPhi(\QQ)\}
=\operatorname{supp}(\boldsymbol\mu)
=\{(\vv,\vv):\vv\in\{0,1\}^K\}
=\operatorname{Cols}\{\PPhi(\bar{\QQ})\}.
\]
The algebraic condition again gives \(\bar{\QQ}\sim\QQ\).
In particular, \(\QQ\) also consists of two copies of every
pure row, after rearranging its rows and columns.

We have proved \(\bar{\QQ}\sim\QQ\) in every case, so
\(\QQ\) is identifiable. If \(\QQ\) cannot be rearranged
into two copies of \(\cI_K\), the case
\(s\neq\bar s\), \(g\neq\bar g\) is impossible. The first
case then gives \(s=\bar s\), \(g=\bar g\), and equality
of the grouped proportions after aligning the columns. This
proves joint identifiability for every \(\TTT\) with positive
grouped proportions.

We next prove necessity. If the algebraic condition fails,
there exists \(\bar{\QQ}\nsim\QQ\) such that
\[
\operatorname{Cols}\{\PPhi(\QQ)\}
\subseteq\operatorname{Cols}\{\PPhi(\bar{\QQ})\}.
\]
Keep \(\bar s=s\) and \(\bar g=g\), and choose
\[
\bar p_{[\bbb]}
=\mu_{\PPhi_{\Cdot,\bbb}(\bar{\QQ})},
\qquad [\bbb]\in\mathcal E(\bar{\QQ}).
\]
These grouped proportions are nonnegative and sum to one, and
\(\bar{\boldsymbol\mu}=\boldsymbol\mu\). Hence
\[
\bar M^{\otimes J}\bar{\boldsymbol\mu}
=M^{\otimes J}\boldsymbol\mu.
\]
Thus \(\QQ\) is not identifiable, and joint identifiability
fails at every true parameter specification with positive
grouped proportions.

Finally, suppose that \(\QQ\) can be transformed into two
copies of \(\cI_K\) by permuting its rows and columns.
We construct a parameter specification at which joint
identifiability fails. After these permutations, write
\[
\QQ=\begin{pmatrix}\cI_K\\\cI_K\end{pmatrix}.
\]
Under the conjunctive operator,
\(\PPhi_{\Cdot,\aaa}(\QQ)=(\aaa^\top,\aaa^\top)^\top\),
so every equivalence class contains a single latent vector.
Let \(\P\) and \(\bar{\P}\) denote the probabilities under two
parameter specifications with this same \(\QQ\). Under each
specification, let the attributes be mutually independent, and set
\[
\begin{aligned}
(s,g)&=\left(\frac{3}{10},\frac{1}{10}\right),
&\P(\alpha_k=1)&=\frac{7}{12},\\
(\bar s,\bar g)&=\left(\frac{1}{5},\frac{1}{5}\right),
&\bar{\P}(\alpha_k=1)&=\frac{5}{12},
\qquad k\in[K].
\end{aligned}
\]
Both specifications satisfy \(0<g<1-s<1\) and
\(0<\bar g<1-\bar s<1\), and
\(p_{\aaa}>0\), \(\bar p_{\aaa}>0\) for every
\(\aaa\in\{0,1\}^K\).

For every \(k\in[K]\), items \(k\) and \(K+k\) have the
same ideal response \(\alpha_k\). By local independence,
\[
\begin{aligned}
\P(R_k=0,R_{K+k}=0)
&=\frac{5}{12}\left(\frac{9}{10}\right)^2
 +\frac{7}{12}\left(\frac{3}{10}\right)^2
 =\frac{39}{100}\\
&=\frac{7}{12}\left(\frac{4}{5}\right)^2
 +\frac{5}{12}\left(\frac{1}{5}\right)^2
 =\bar{\P}(R_k=0,R_{K+k}=0),
\end{aligned}
\]
\[
\begin{aligned}
\P(R_k=0,R_{K+k}=1)
&=\P(R_k=1,R_{K+k}=0)\\
&=\frac{5}{12}\frac{9}{10}\frac{1}{10}
 +\frac{7}{12}\frac{3}{10}\frac{7}{10}
 =\frac{4}{25}\\
&=\frac{7}{12}\frac{4}{5}\frac{1}{5}
 +\frac{5}{12}\frac{1}{5}\frac{4}{5}\\
&=\bar{\P}(R_k=0,R_{K+k}=1)
 =\bar{\P}(R_k=1,R_{K+k}=0),
\end{aligned}
\]
and
\[
\begin{aligned}
\P(R_k=1,R_{K+k}=1)
&=\frac{5}{12}\left(\frac{1}{10}\right)^2
 +\frac{7}{12}\left(\frac{7}{10}\right)^2
 =\frac{29}{100}\\
&=\frac{7}{12}\left(\frac{1}{5}\right)^2
 +\frac{5}{12}\left(\frac{4}{5}\right)^2
 =\bar{\P}(R_k=1,R_{K+k}=1).
\end{aligned}
\]
The independence of the attributes and local independence imply
that the \(K\) response pairs are mutually independent under
each specification. Hence, for every \(\rr\in\{0,1\}^{2K}\),
\[
\begin{aligned}
\P(\RR=\rr)
&=\prod_{k=1}^K\P(R_k=r_k,R_{K+k}=r_{K+k})\\
&=\prod_{k=1}^K\bar{\P}(R_k=r_k,R_{K+k}=r_{K+k})
 =\bar{\P}(\RR=\rr).
\end{aligned}
\]
Since \((s,g)\neq(\bar s,\bar g)\), no relabeling of the
latent variables makes these two parameter specifications
equivalent. Thus \((\QQ,\TTT)\) is not jointly identifiable
at the specified parameters. In particular, joint identifiability
does not hold for every \(\TTT\) with positive grouped
proportions. This completes the proof.
\end{proof}

\subsection{Proofs of results in Section~\ref{app:prob_multi}}
\label{app:proofs_multi}

\subsubsection{Proof of Theorem~\ref{thm:multi-general}}
\label{app:proof-multi-general}

\begin{lemma}\label{lem:multi-block-identifiability}
Consider the multi-parameter Boolean factor model, and let
\(I\subseteq[J]\). Suppose that every class in
\(\mathcal E(\QQ_I)\) contains a class in
\(\mathcal E(\QQ_{[J]\setminus I})\), and that
\(p_{[\aaa]}>0\) for every \([\aaa]\in\mathcal E(\QQ)\).
If \eqref{eq:equal-response-distributions} holds, then
\[
\operatorname{Cols}\{\PPhi(\QQ_I)\}
\subseteq\operatorname{Cols}\{\PPhi(\bar{\QQ}_I)\}.
\]
In particular, if \(\QQ_I\) satisfies the algebraic condition,
then \(\bar{\QQ}_I\sim\QQ_I\).
\end{lemma}

\begin{proof}
Write \(I^c=[J]\setminus I\). For \(j\in[J]\), let
\[
M_j=\begin{pmatrix}1-g_j&s_j\\g_j&1-s_j\end{pmatrix},
\qquad
\bar M_j=\begin{pmatrix}
1-\bar g_j&\bar s_j\\\bar g_j&1-\bar s_j
\end{pmatrix}.
\]
For \(H\subseteq[J]\), write
\(M_H=\bigotimes_{j\in H}M_j\) and
\(\bar M_H=\bigotimes_{j\in H}\bar M_j\).
These matrices are invertible because
\[
\det(M_j)=1-s_j-g_j>0,
\qquad
\det(\bar M_j)=1-\bar s_j-\bar g_j>0.
\]

For \(\vv\in\{0,1\}^{|I|}\) and
\(\mathbf w\in\{0,1\}^{|I^c|}\), let
\[
\mu_{\vv,\mathbf w}
=\sum_{\substack{\aaa:\,
\PPhi_{\Cdot,\aaa}(\QQ_I)=\vv,\\
\PPhi_{\Cdot,\aaa}(\QQ_{I^c})=\mathbf w}}p_{\aaa},
\qquad
\boldsymbol\mu=(\mu_{\vv,\mathbf w})_{\vv,\mathbf w},
\]
and let
\[
\bar\mu_{\vv,\mathbf w}
=\sum_{\substack{\aaa:\,
\PPhi_{\Cdot,\aaa}(\bar{\QQ}_I)=\vv,\\
\PPhi_{\Cdot,\aaa}(\bar{\QQ}_{I^c})=\mathbf w}}\bar p_{\aaa},
\qquad
\bar{\boldsymbol\mu}=(\bar\mu_{\vv,\mathbf w})_{\vv,\mathbf w}.
\]
Coordinates of block vectors retain their original
item labels. Let
\[
P_{I,I^c}
=\bigl[\P(\RR_I=\rr,\RR_{I^c}=\rr')\bigr]_{\rr,\rr'}.
\]
By local independence and equality of the response distributions,
\begin{equation}\label{eq:multi-block-factorization}
P_{I,I^c}
=M_I\boldsymbol\mu M_{I^c}^{\top}
=\bar M_I\bar{\boldsymbol\mu}\bar M_{I^c}^{\top}.
\end{equation}

Fix \(\vv\in\operatorname{Cols}\{\PPhi(\QQ_I)\}\).
By the class-containment assumption, there exists \(\bbb\) such that
\[
\PPhi_{\Cdot,\bbb}(\QQ_I)=\vv,
\qquad
[\bbb]_{\QQ_{I^c}}\subseteq[\bbb]_{\QQ_I}.
\]
Put \(\mathbf w=\PPhi_{\Cdot,\bbb}(\QQ_{I^c})\). Then
\[
[\bbb]_{\QQ_{I^c}}
=[\bbb]_{\QQ_I}\cap[\bbb]_{\QQ_{I^c}}
=[\bbb]_{\QQ},
\qquad
\boldsymbol\mu_{\Cdot,\mathbf w}
=p_{[\bbb]}\ee_{\vv},
\]
where \(\ee_{\vv}\) is the coordinate vector indexed by \(\vv\).
Since \(p_{[\bbb]}>0\), \eqref{eq:multi-block-factorization} gives
\[
M_I\ee_{\vv}
=\frac1{p_{[\bbb]}}
P_{I,I^c}M_{I^c}^{-\top}\ee_{\mathbf w}
\in\operatorname{col}(P_{I,I^c}).
\]
Here \(\operatorname{col}\) denotes the linear column space.
Every nonzero row of \(\bar{\boldsymbol\mu}\) is indexed by
a column of \(\PPhi(\bar{\QQ}_I)\). Therefore,
\[
M_I\ee_{\vv}
\in\operatorname{col}(P_{I,I^c})
\subseteq\operatorname{span}
\bigl\{\bar M_I\ee_{\mathbf u}:
\mathbf u\in\operatorname{Cols}\{\PPhi(\bar{\QQ}_I)\}\bigr\}.
\]

For \(j\in I\), let
\[
B_j=\bar M_j^{-1}M_j
=\frac1{1-\bar s_j-\bar g_j}
\begin{pmatrix}
1-\bar s_j-g_j&s_j-\bar s_j\\
g_j-\bar g_j&1-s_j-\bar g_j
\end{pmatrix}.
\]
Multiplying the preceding inclusion by \(\bar M_I^{-1}\) gives
\[
\left(\bigotimes_{j\in I}B_j\right)\ee_{\vv}
\in\operatorname{span}
\bigl\{\ee_{\mathbf u}:
\mathbf u\in\operatorname{Cols}\{\PPhi(\bar{\QQ}_I)\}\bigr\}.
\]
Thus, for every
\(\mathbf u\notin\operatorname{Cols}\{\PPhi(\bar{\QQ}_I)\}\),
\begin{equation}\label{eq:multi-missing-pattern}
0
=\left[\left(\bigotimes_{j\in I}B_j\right)
\ee_{\vv}\right]_{\mathbf u}
=\prod_{j\in I}(B_j)_{u_j,v_j}.
\end{equation}

By the nondegeneracy assumption and
Lemma~\ref{lem:one-item-separation},
\[
1-\bar s_j>g_j,
\qquad
1-s_j>\bar g_j,
\qquad j\in I.
\]
Both diagonal entries of every \(B_j\) are positive, so
\[
\left[\left(\bigotimes_{j\in I}B_j\right)
\ee_{\vv}\right]_{\vv}
=\prod_{j\in I}(B_j)_{v_j,v_j}>0.
\]
Equation~\eqref{eq:multi-missing-pattern} implies
\(\vv\in\operatorname{Cols}\{\PPhi(\bar{\QQ}_I)\}\).
Since \(\vv\) was arbitrary, the desired inclusion follows.
The last assertion follows from the algebraic condition.
\end{proof}

\begin{lemma}\label{lem:multi-Q-only}
Under conditions \textup{(i)}--\textup{(ii)} of
Theorem~\ref{thm:multi-general}, \(\QQ\) is identifiable
in the multi-parameter Boolean factor model~\eqref{eq:multi-model}.
\end{lemma}

\begin{proof}
Fix \(\TTT\) with \(p_{[\aaa]}>0\) for every
\([\aaa]\in\mathcal E(\QQ)\), and suppose that
\eqref{eq:equal-response-distributions} holds.
By condition \textup{(i)} and
Lemma~\ref{lem:multi-block-identifiability}, for every \(b\in[B]\)
there is a permutation matrix \(\Pi_b\) such that
\[
\bar{\QQ}_{I_b}=\QQ_{I_b}\Pi_b.
\]
Order the blocks as in \textup{(ii)}. Suppose inductively that
\[
\bar{\QQ}_{I_1\cup\cdots\cup I_{b-1}}
=\QQ_{I_1\cup\cdots\cup I_{b-1}}\Pi_1.
\]
On the overlap,
\[
\QQ_{I_b\cap(I_1\cup\cdots\cup I_{b-1})}\Pi_b
=\QQ_{I_b\cap(I_1\cup\cdots\cup I_{b-1})}\Pi_1.
\]
The overlap matrix has pairwise distinct columns, so
\(\Pi_b=\Pi_1\). Induction and \(\bigcup_{b=1}^B I_b=[J]\)
give \(\bar{\QQ}=\QQ\Pi_1\). Since \(\TTT\) was arbitrary,
\(\QQ\) is identifiable.
\end{proof}

\begin{proof}[Proof of Theorem~\ref{thm:multi-general}]
The first assertion follows from Lemma~\ref{lem:multi-Q-only}.
Suppose now that \textup{(iii)} also holds. Fix \(\TTT\) with
positive grouped proportions, and suppose that
\eqref{eq:equal-response-distributions} holds. By Lemma~\ref{lem:multi-Q-only},
\(\bar{\QQ}\sim\QQ\). After a common relabeling of the latent
variables, assume that \(\bar{\QQ}=\QQ\).

Fix \(j\in[J]\). By \textup{(iii)}, there exist \(I_b\ni j\)
and \(\aaa\) such that
\[
[\aaa]_{\QQ_{I_b\setminus\{j\}}}\subseteq\mca_j^+(\QQ).
\]
Let \(\vv=\PPhi_{\Cdot,\aaa}(\QQ_{I_b})\), and let
\(\mathbf u\) agree with \(\vv\) except that \(u_j=0\).
Then
\[
v_j=1,
\qquad
\vv\in\operatorname{Cols}\{\PPhi(\QQ_{I_b})\},
\qquad
\mathbf u\notin\operatorname{Cols}\{\PPhi(\QQ_{I_b})\}.
\]
Use \(M_h\) and \(B_h\) as defined in the proof of
Lemma~\ref{lem:multi-block-identifiability}.
Since \(\bar{\QQ}=\QQ\), \eqref{eq:multi-missing-pattern} gives
\[
0
=\prod_{h\in I_b}(B_h)_{u_h,v_h}
=\frac{s_j-\bar s_j}{1-\bar s_j-\bar g_j}
\prod_{h\in I_b\setminus\{j\}}(B_h)_{v_h,v_h}.
\]
The diagonal factors are positive, so \(s_j=\bar s_j\).

For the incapable class, \textup{(iii)} gives \(I_{b'}\ni j\)
and \(\aaa\) such that
\[
[\aaa]_{\QQ_{I_{b'}\setminus\{j\}}}\subseteq\mca_j^-(\QQ).
\]
Let \(\vv=\PPhi_{\Cdot,\aaa}(\QQ_{I_{b'}})\), and let
\(\mathbf u\) agree with \(\vv\) except that \(u_j=1\).
Now
\[
v_j=0,
\qquad
\vv\in\operatorname{Cols}\{\PPhi(\QQ_{I_{b'}})\},
\qquad
\mathbf u\notin\operatorname{Cols}\{\PPhi(\QQ_{I_{b'}})\}.
\]
Equation~\eqref{eq:multi-missing-pattern} gives
\[
0
=\prod_{h\in I_{b'}}(B_h)_{u_h,v_h}
=\frac{g_j-\bar g_j}{1-\bar s_j-\bar g_j}
\prod_{h\in I_{b'}\setminus\{j\}}(B_h)_{v_h,v_h}.
\]
Hence \(g_j=\bar g_j\). Since \(j\) was arbitrary,
\(\mathbf s=\bar{\mathbf s}\) and
\(\mathbf g=\bar{\mathbf g}\).

For \(\vv\in\{0,1\}^J\), let
\[
\mu_{\vv}
=\sum_{\aaa:\,\PPhi_{\Cdot,\aaa}(\QQ)=\vv}p_{\aaa},
\qquad
\bar\mu_{\vv}
=\sum_{\aaa:\,\PPhi_{\Cdot,\aaa}(\QQ)=\vv}\bar p_{\aaa}.
\]
Equality of the response distributions now gives
\[
\left(\bigotimes_{j=1}^J M_j\right)
(\mu_{\vv})_{\vv}
=\left(\bigotimes_{j=1}^J M_j\right)
(\bar\mu_{\vv})_{\vv}.
\]
Each \(M_j\) is invertible, so \(\mu_{\vv}=\bar\mu_{\vv}\)
for every \(\vv\). Therefore,
\[
p_{[\aaa]}
=\mu_{\PPhi_{\Cdot,\aaa}(\QQ)}
=\bar\mu_{\PPhi_{\Cdot,\aaa}(\QQ)}
=\bar p_{[\aaa]},
\qquad [\aaa]\in\mathcal E(\QQ).
\]
This proves joint identifiability.
\end{proof}

\subsubsection{Proof of Theorem~\ref{thm:multi-conj}}
\label{app:proof-multi-conj}

\begin{lemma}\label{lem:multi-anchor-extraction}
Suppose that \eqref{eq:equal-response-distributions} holds and that
\(p_{[\aaa]}>0\) for every \([\aaa]\in\mathcal E(\QQ)\).
Let \(I\subseteq[J]\) and suppose that, for some \(\aaa\),
\[
[\aaa]_{\QQ_{[J]\setminus I}}\subseteq[\aaa]_{\QQ_I}.
\]
With \(B_j\) defined in the proof of
Lemma~\ref{lem:multi-block-identifiability}, put
\(\vv=\PPhi_{\Cdot,\aaa}(\QQ_I)\). Then, for every
\(\mathbf u\notin\operatorname{Cols}\{\PPhi(\bar{\QQ}_I)\}\),
\begin{equation}\label{eq:multi-anchor-missing-pattern}
\prod_{j\in I}(B_j)_{u_j,v_j}=0.
\end{equation}
In particular, \(\vv\in\operatorname{Cols}\{\PPhi(\bar{\QQ}_I)\}\).
\end{lemma}

\begin{proof}
Write \(I^c=[J]\setminus I\), and use the matrices
\(M_I,M_{I^c},\bar M_I,\bar M_{I^c},\boldsymbol\mu,
\bar{\boldsymbol\mu}\) from the proof of
Lemma~\ref{lem:multi-block-identifiability}.
Let \(\mathbf w=\PPhi_{\Cdot,\aaa}(\QQ_{I^c})\).
The assumed containment gives
\([\aaa]_{\QQ_{I^c}}=[\aaa]_{\QQ}\), so
\(\boldsymbol\mu_{\Cdot,\mathbf w}=p_{[\aaa]}\ee_{\vv}\).
By \eqref{eq:multi-block-factorization},
\[
\left(\bigotimes_{j\in I}B_j\right)\ee_{\vv}
=\frac{1}{p_{[\aaa]}}
\bar M_I^{-1}P_{I,I^c}M_{I^c}^{-\top}\ee_{\mathbf w}
=\frac{1}{p_{[\aaa]}}
\bar{\boldsymbol\mu}\bar M_{I^c}^{\top}
M_{I^c}^{-\top}\ee_{\mathbf w}.
\]
The \(\mathbf u\)-th row of \(\bar{\boldsymbol\mu}\) is zero
whenever
\(\mathbf u\notin\operatorname{Cols}\{\PPhi(\bar{\QQ}_I)\}\).
Taking the \(\mathbf u\)-th coordinate proves
\eqref{eq:multi-anchor-missing-pattern}.
The diagonal entries of \(B_j\) are positive by
Lemma~\ref{lem:one-item-separation}, as established in the proof of
Lemma~\ref{lem:multi-block-identifiability}. Taking \(\mathbf u=\vv\)
therefore gives the last assertion.
\end{proof}

\begin{lemma}\label{lem:multi-conj-two-column}
Under the conjunctive operator, let \(K\geq2\).
If \(\QQ\) satisfies the algebraic condition, then it satisfies
the two-column condition.
Moreover, if \(\QQ_H\) satisfies the two-column condition, then,
for every \(k\in[K]\),
\begin{equation}\label{eq:multi-conj-zero-k-cover}
\bigvee_{h\in H:\,q_{hk}=0}\qq_h=\one-\ee_k,
\qquad
\sum_{h\in H}q_{hk}\geq1,
\end{equation}
and \([\one-\ee_k]_{\QQ_H}=\{\one-\ee_k\}\).
\end{lemma}

\begin{proof}
First suppose that column \(k\) of \(\QQ\) is zero. Add a one
to this column to obtain \(\bar{\QQ}\). Every true ideal pattern
is attained under \(\bar{\QQ}\) by setting \(\alpha_k=1\), so
\[
\operatorname{Cols}\{\PPhi(\QQ)\}
\subseteq\operatorname{Cols}\{\PPhi(\bar{\QQ})\}.
\]
But \(\bar{\QQ}\) has more ones than \(\QQ\), contradicting
the algebraic condition. Thus every column is nonzero.

If the two-column condition fails, there are distinct \(k,\ell\)
with \(\QQ_{\Cdot,k}\preceq\QQ_{\Cdot,\ell}\).
Let \(\bar{\QQ}\) agree with \(\QQ\) except that
\[
\bar{\QQ}_{\Cdot,\ell}
=\QQ_{\Cdot,\ell}-\QQ_{\Cdot,k}.
\]
Every changed row still has a one in column \(k\), so
no new zero rows are created. Also,
\[
\QQ_{\Cdot,\ell}
=\bar{\QQ}_{\Cdot,k}\vee\bar{\QQ}_{\Cdot,\ell}.
\]
Every union of columns of \(\QQ\) is therefore a union of
columns of \(\bar{\QQ}\). Since
\[
\PPhi_{\Cdot,\aaa}(\QQ)
=\one-\bigvee_{k:\,\alpha_k=0}\QQ_{\Cdot,k},
\]
we obtain
\(\operatorname{Cols}\{\PPhi(\QQ)\}
\subseteq\operatorname{Cols}\{\PPhi(\bar{\QQ})\}\).
However, \(\bar{\QQ}\) has fewer ones than \(\QQ\), so
\(\bar{\QQ}\nsim\QQ\). This again contradicts the algebraic
condition.

Now suppose that \(\QQ_H\) satisfies the two-column condition.
For every \(\ell\ne k\), some \(h\in H\) has
\(q_{hk}=0\) and \(q_{h\ell}=1\), proving the first equality
in \eqref{eq:multi-conj-zero-k-cover}. The second follows from
the existence of a row with \(q_{hk}=1\).
If \(\bbb\in[\one-\ee_k]_{\QQ_H}\), all rows with
\(q_{hk}=0\) are capable under \(\bbb\), so
\(\bbb\succeq\one-\ee_k\). A row with \(q_{hk}=1\) is
incapable, ruling out \(\bbb=\one\). Hence
\(\bbb=\one-\ee_k\).
\end{proof}

\begin{lemma}\label{lem:multi-conj-block-inclusion}
Consider the multi-parameter Boolean factor model under the
conjunctive operator, with \(K\geq2\) and
\(p_{[\aaa]}>0\) for every \([\aaa]\in\mathcal E(\QQ)\).
Suppose that \eqref{eq:equal-response-distributions} holds.
If \(\QQ_{[J]\setminus I}\) satisfies the two-column condition,
then
\[
\operatorname{Cols}\{\PPhi(\QQ_I)\}
\subseteq\operatorname{Cols}\{\PPhi(\bar{\QQ}_I)\}.
\]
\end{lemma}

\begin{proof}
For every \(k\in[K]\), Lemma~\ref{lem:multi-conj-two-column}
gives
\[
[\one-\ee_k]_{\QQ_{[J]\setminus I}}
=\{\one-\ee_k\}\subseteq[\one-\ee_k]_{\QQ_I}.
\]
Lemma~\ref{lem:multi-anchor-extraction} implies
\[
\PPhi_{\Cdot,\one-\ee_k}(\QQ_I)
\in\operatorname{Cols}\{\PPhi(\bar{\QQ}_I)\},
\qquad k\in[K].
\]
Under conjunction, ideal-response column sets contain \(\one\)
and are closed under coordinatewise intersections. Consequently,
for every \(\aaa\in\{0,1\}^K\),
\[
\PPhi_{\Cdot,\aaa}(\QQ_I)
=\bigwedge_{k:\,\alpha_k=0}
\PPhi_{\Cdot,\one-\ee_k}(\QQ_I)
\in\operatorname{Cols}\{\PPhi(\bar{\QQ}_I)\},
\]
where the empty intersection is \(\one\).
\end{proof}

\begin{lemma}\label{lem:multi-conj-Q-only}
Consider the multi-parameter Boolean factor model~\eqref{eq:multi-model}
under the conjunctive operator, with \(K\geq2\) and
\(p_{[\aaa]}>0\) for every \([\aaa]\in\mathcal E(\QQ)\).
Suppose that \(I_1,I_2\) partition \([J]\), where
\(\QQ_{I_1}\) satisfies the algebraic
condition~\eqref{eq:a_condition} and \(\QQ_{I_2}\) satisfies
the two-column condition. Then \(\QQ\) is identifiable.
\end{lemma}

\begin{proof}
Fix \(\TTT\) with \(p_{[\aaa]}>0\) for every
\([\aaa]\in\mathcal E(\QQ)\), and suppose that
\eqref{eq:equal-response-distributions} holds.
Since \([J]\setminus I_1=I_2\),
Lemma~\ref{lem:multi-conj-block-inclusion} gives
\[
\operatorname{Cols}\{\PPhi(\QQ_{I_1})\}
\subseteq\operatorname{Cols}\{\PPhi(\bar{\QQ}_{I_1})\}.
\]
The algebraic condition gives
\(\bar{\QQ}_{I_1}\sim\QQ_{I_1}\).
After a common relabeling of the latent variables, assume that
\(\bar{\QQ}_{I_1}=\QQ_{I_1}\).
By Lemma~\ref{lem:multi-conj-two-column},
\(\QQ_{I_1}\) also satisfies the two-column condition.

Fix \(h\in I_2\) and \(k\in[K]\) with \(q_{hk}=1\).
Choose \(j\in I_1\) with \(q_{jk}=1\), and let
\[
I=(I_1\setminus\{j\})\cup\{h\},
\qquad
I^c=(I_2\setminus\{h\})\cup\{j\}.
\]
Every \(r\in I_2\) with \(q_{rk}=0\) remains in \(I^c\),
and the added row \(j\) has \(q_{jk}=1\).
Equation~\eqref{eq:multi-conj-zero-k-cover} therefore gives
\[
[\one-\ee_k]_{\QQ_{I^c}}=\{\one-\ee_k\}.
\]
Put \(\vv=\PPhi_{\Cdot,\one-\ee_k}(\QQ_I)\).
Lemma~\ref{lem:multi-anchor-extraction} gives
\(\vv\in\operatorname{Cols}\{\PPhi(\bar{\QQ}_I)\}\).
Choose \(\bbb\) such that
\[
\PPhi_{\Cdot,\bbb}(\bar{\QQ}_I)=\vv.
\]
Every \(r\in I_1\) with \(q_{rk}=0\) remains in \(I\),
and \(\bar{\qq}_r=\qq_r\). These items are capable under
\(\bbb\), so \eqref{eq:multi-conj-zero-k-cover} implies
\(\bbb\succeq\one-\ee_k\).
If \(\bar q_{hk}=0\), item \(h\) would therefore be capable
under \(\bbb\), contradicting \(v_h=0\).
Thus \(\bar q_{hk}=1\). Since \(h,k\) were arbitrary,
\begin{equation}\label{eq:multi-conj-row-domination}
\bar{\QQ}_{I_2}\succeq\QQ_{I_2}.
\end{equation}

Since \([J]\setminus I_2=I_1\) and \(\QQ_{I_1}\) satisfies
the two-column condition,
Lemma~\ref{lem:multi-conj-block-inclusion} also gives
\[
\operatorname{Cols}\{\PPhi(\QQ_{I_2})\}
\subseteq\operatorname{Cols}\{\PPhi(\bar{\QQ}_{I_2})\}.
\]
For each \(k\in[K]\), choose \(\bbb\) such that
\[
\PPhi_{\Cdot,\bbb}(\bar{\QQ}_{I_2})
=\PPhi_{\Cdot,\one-\ee_k}(\QQ_{I_2}).
\]
Under the conjunctive operator, this equality and
\eqref{eq:multi-conj-row-domination} give
\[
\QQ_{I_2,k}
=\bigvee_{\ell:\,\beta_\ell=0}\bar{\QQ}_{I_2,\ell}
\succeq\bigvee_{\ell:\,\beta_\ell=0}\QQ_{I_2,\ell}.
\]
If \(\beta_\ell=0\), then
\(\QQ_{I_2,\ell}\preceq\QQ_{I_2,k}\).
The two-column condition rules this out for \(\ell\ne k\).
Because \(\QQ_{I_2,k}\ne\zero\), at least one
\(\beta_\ell\) is zero. Hence \(\bbb=\one-\ee_k\), and
\[
\bar{\QQ}_{I_2,k}=\QQ_{I_2,k}.
\]
This holds for every \(k\), proving \(\bar{\QQ}=\QQ\)
after aligning their columns.
Since \(\TTT\) was arbitrary subject to positive grouped
proportions, \(\QQ\) is identifiable.
\end{proof}

\begin{lemma}\label{lem:multi-conj-item-parameters}
Consider the multi-parameter Boolean factor model under the
conjunctive operator, with \(K\geq2\) and
\(p_{[\aaa]}>0\) for every \([\aaa]\in\mathcal E(\QQ)\).
Suppose that there exist disjoint row sets \(I_1,I_2\) such that
\(\QQ_{I_1}\) and \(\QQ_{I_2}\) satisfy the two-column condition,
and that every column of \(\QQ\) contains at least three ones.
If \eqref{eq:equal-response-distributions} holds with
\(\bar{\QQ}=\QQ\), then
\[
\bar{\mathbf g}=\mathbf g,\qquad
\bar{\mathbf s}=\mathbf s,\qquad
\bar p_{[\aaa]}=p_{[\aaa]},
\quad [\aaa]\in\mathcal E(\QQ).
\]
\end{lemma}

\begin{proof}
Fix \(j\in[J]\), and choose \(k\) with \(q_{jk}=1\).
The column-count assumption gives distinct items
\(h,\ell\in[J]\setminus\{j\}\) such that
\(q_{hk}=q_{\ell k}=1\). Let
\[
I=\{j,h\}\cup\{r\in I_1:q_{rk}=0\},
\qquad I^c=[J]\setminus I.
\]
Then \(I\setminus\{j\}\) contains \(h\) and all the rows of
\(I_1\) with \(q_{rk}=0\); \(I^c\) contains \(\ell\) and all
 the rows of \(I_2\) with \(q_{rk}=0\).
By Lemma~\ref{lem:multi-conj-two-column},
\[
\bigvee_{r\in I_1:\,q_{rk}=0}\qq_r
=\bigvee_{r\in I_2:\,q_{rk}=0}\qq_r
=\one-\ee_k.
\]
Consequently, both \(\QQ_{I\setminus\{j\}}\) and \(\QQ_{I^c}\)
have no zero column, and their rows with \(q_{rk}=0\) cover
all attributes other than \(k\). Therefore,
\begin{equation}\label{eq:multi-conj-joint-private-classes}
\begin{aligned}
{}[\one]_{\QQ_{I^c}}
&=[\one]_{\QQ_{I\setminus\{j\}}}=\{\one\},\\
[\one-\ee_k]_{\QQ_{I^c}}
&=[\one-\ee_k]_{\QQ_{I\setminus\{j\}}}
=\{\one-\ee_k\}.
\end{aligned}
\end{equation}
For the second equality, the capable rows with \(q_{rk}=0\)
force the latent vector to contain \(\one-\ee_k\); the
incapable item \(h\), or \(\ell\), rules out \(\one\).

First let \(\vv=\PPhi_{\Cdot,\one}(\QQ_I)=\one\), and let
\(\mathbf u\) agree with \(\vv\) except that \(u_j=0\).
The first line of \eqref{eq:multi-conj-joint-private-classes}
implies \(\mathbf u\notin\operatorname{Cols}\{\PPhi(\QQ_I)\}\).
By Lemma~\ref{lem:multi-anchor-extraction} and
\(\bar{\QQ}=\QQ\),
\[
0
=\prod_{r\in I}(B_r)_{u_r,v_r}
=\frac{s_j-\bar s_j}{1-\bar s_j-\bar g_j}
\prod_{r\in I\setminus\{j\}}(B_r)_{1,1}.
\]
Every diagonal factor is positive by
Lemma~\ref{lem:one-item-separation}, as in the proof of
Lemma~\ref{lem:multi-block-identifiability}.
Hence \(s_j=\bar s_j\).

Next let \(\vv=\PPhi_{\Cdot,\one-\ee_k}(\QQ_I)\), and let
\(\mathbf u\) agree with \(\vv\) except that \(u_j=1\).
Now \(v_j=0\), and the second line of
\eqref{eq:multi-conj-joint-private-classes} gives
\(\mathbf u\notin\operatorname{Cols}\{\PPhi(\QQ_I)\}\).
Lemma~\ref{lem:multi-anchor-extraction} gives
\[
0
=\prod_{r\in I}(B_r)_{u_r,v_r}
=\frac{g_j-\bar g_j}{1-\bar s_j-\bar g_j}
\prod_{r\in I\setminus\{j\}}(B_r)_{v_r,v_r}.
\]
The diagonal factors are again positive, so
\(g_j=\bar g_j\). Since \(j\) was arbitrary,
\(\mathbf s=\bar{\mathbf s}\) and
\(\mathbf g=\bar{\mathbf g}\).
Finally, the matrix inversion in the proof of
Theorem~\ref{thm:multi-general} gives
\(p_{[\aaa]}=\bar p_{[\aaa]}\) for every
\([\aaa]\in\mathcal E(\QQ)\).
\end{proof}

\begin{proof}[Proof of Theorem~\ref{thm:multi-conj}]
Lemma~\ref{lem:multi-conj-Q-only} proves identifiability of \(\QQ\).
Suppose now that every column contains at least three ones, and
that \eqref{eq:equal-response-distributions} holds. By the first
assertion, a common relabeling of the latent variables gives
\(\bar{\QQ}=\QQ\). The algebraic condition on \(\QQ_{I_1}\)
implies its two-column condition by
Lemma~\ref{lem:multi-conj-two-column}. Applying
Lemma~\ref{lem:multi-conj-item-parameters} to the disjoint blocks
\(I_1,I_2\) proves joint identifiability.
\end{proof}

\subsubsection{Proof of Proposition~\ref{prop:multi-known-g-conjunctive}}
\label{app:proof-multi-known-g}

\begin{lemma}\label{lem:multi-known-g-general}
Consider the multi-parameter Boolean factor model~\eqref{eq:multi-model}
with \(\mathbf g\) known and \(p_{[\aaa]}>0\) for every
\([\aaa]\in\mathcal E(\QQ)\). Suppose that \(\QQ\) satisfies
the algebraic condition. If, for every \(j\in[J]\), there exists
\(S_j\subseteq[J]\setminus\{j\}\) such that
\[
\varnothing\neq\bigcap_{h\in S_j}\mca_h^+(\QQ)
\subseteq\mca_j^+(\QQ),
\]
then \((\QQ,\mathbf s,p_{[\Cdot]})\) is jointly identifiable.
An intersection over an empty set is understood to be \(\{0,1\}^K\).
\end{lemma}

\begin{proof}
Suppose that
\[
\P(\RR=\rr\mid\QQ,\mathbf s,\mathbf g,p_{[\Cdot]})
=\P(\RR=\rr\mid\bar{\QQ},\bar{\mathbf s},\mathbf g,
\bar p_{[\Cdot]}),
\qquad \rr\in\{0,1\}^J.
\]
For \(j\in[J]\), let \(H_j=\prod_{h\in S_j}(R_h-g_h)\).
By local independence,
\[
\mathbb E(H_j)
=\prod_{h\in S_j}(1-s_h-g_h)
\sum_{\aaa\in\bigcap_{h\in S_j}\mca_h^+(\QQ)}p_{\aaa}>0.
\]
The strict inequality follows because the intersection is
nonempty and is a union of equivalence classes of \(\QQ\).
The assumed containment gives
\[
\mathbb E\{(R_j-g_j)H_j\}
=(1-s_j-g_j)\mathbb E(H_j).
\]
Under the alternative parameters, the same expectations satisfy
\[
\begin{aligned}
1-s_j-g_j
&=\frac{\mathbb E\{(R_j-g_j)H_j\}}{\mathbb E(H_j)}\\
&=(1-\bar s_j-g_j)
\frac{\displaystyle
\sum_{\aaa\in\mca_j^+(\bar{\QQ})\cap
\bigcap_{h\in S_j}\mca_h^+(\bar{\QQ})}\bar p_{\aaa}}
{\displaystyle
\sum_{\aaa\in\bigcap_{h\in S_j}\mca_h^+(\bar{\QQ})}
\bar p_{\aaa}}\\
&\leq1-\bar s_j-g_j.
\end{aligned}
\]
The denominator is positive because \(\mathbb E(H_j)>0\).
Thus \(s_j\geq\bar s_j\) for every \(j\).

For \(\vv\in\{0,1\}^J\), let
\[
\mu_{\vv}
=\sum_{\aaa:\,\PPhi_{\Cdot,\aaa}(\QQ)=\vv}p_{\aaa},
\qquad
\bar\mu_{\vv}
=\sum_{\aaa:\,\PPhi_{\Cdot,\aaa}(\bar{\QQ})=\vv}\bar p_{\aaa},
\]
and write \(\boldsymbol\mu=(\mu_{\vv})_{\vv}\) and
\(\bar{\boldsymbol\mu}=(\bar\mu_{\vv})_{\vv}\).
Use \(M_j,\bar M_j,B_j\) from the proof of
Lemma~\ref{lem:multi-block-identifiability}, with
\(\bar g_j=g_j\). Equality of the response distributions gives
\[
\left(\bigotimes_{j=1}^J M_j\right)\boldsymbol\mu
=\left(\bigotimes_{j=1}^J\bar M_j\right)
\bar{\boldsymbol\mu}.
\]
Each \(\bar M_j\) is invertible, so
\begin{equation}\label{eq:multi-known-g-ideal-transform}
\bar{\boldsymbol\mu}
=\left(\bigotimes_{j=1}^J B_j\right)\boldsymbol\mu,
\qquad
B_j=\begin{pmatrix}
1&\dfrac{s_j-\bar s_j}{1-\bar s_j-g_j}\\[1ex]
0&\dfrac{1-s_j-g_j}{1-\bar s_j-g_j}
\end{pmatrix}.
\end{equation}
Since \(s_j\geq\bar s_j\), each \(B_j\) is nonnegative with
positive diagonal entries. Therefore, for every
\(\vv\in\operatorname{Cols}\{\PPhi(\QQ)\}\),
\[
\bar\mu_{\vv}
\geq\mu_{\vv}\prod_{j=1}^J(B_j)_{v_j,v_j}>0.
\]
It follows that
\[
\operatorname{Cols}\{\PPhi(\QQ)\}
\subseteq\operatorname{Cols}\{\PPhi(\bar{\QQ})\}.
\]
The algebraic condition gives \(\bar{\QQ}\sim\QQ\).
After a common relabeling of the latent variables, assume that
\(\bar{\QQ}=\QQ\). The class containment defining \(S_j\)
then also holds under the alternative model, and the displayed
ratio becomes
\[
1-s_j-g_j
=\frac{\mathbb E\{(R_j-g_j)H_j\}}{\mathbb E(H_j)}
=1-\bar s_j-g_j.
\]
Hence \(\mathbf s=\bar{\mathbf s}\).
The matrix inversion in the proof of
Theorem~\ref{thm:multi-general} now gives
\(p_{[\aaa]}=\bar p_{[\aaa]}\) for every
\([\aaa]\in\mathcal E(\QQ)\).
\end{proof}

\begin{proof}[Proof of Proposition~\ref{prop:multi-known-g-conjunctive}]
We first prove sufficiency. If every column of \(\QQ\) contains
at least two ones, then, for every \(j\),
\[
\bigvee_{h\ne j}\qq_h=\one,
\qquad
\bigcap_{h\ne j}\mca_h^+(\QQ)
=\{\one\}\subseteq\mca_j^+(\QQ).
\]
Lemma~\ref{lem:multi-known-g-general}, with
\(S_j=[J]\setminus\{j\}\), proves joint identifiability.
In particular, the ratio in the proof of
Lemma~\ref{lem:multi-known-g-general} gives
\begin{equation}\label{eq:known-g-slip-ratio}
1-s_j
=g_j+
\frac{\displaystyle\mathbb E\!\left\{\prod_{h=1}^J(R_h-g_h)\right\}}
{\displaystyle\mathbb E\!\left\{\prod_{h\ne j}(R_h-g_h)\right\}},
\qquad j\in[J].
\end{equation}
The denominator equals
\(p_{[\one]}\prod_{h\ne j}(1-s_h-g_h)>0\).
Since the sufficient conditions do not depend on the unknown
parameters, they also imply that \(\QQ\) is identifiable.

We next prove necessity for joint identifiability. If the algebraic
condition fails, there
exists \(\bar{\QQ}\nsim\QQ\) such that
\[
\operatorname{Cols}\{\PPhi(\QQ)\}
\subseteq\operatorname{Cols}\{\PPhi(\bar{\QQ})\}.
\]
Let \(\mu_{\vv}\) be defined as in the proof of
Lemma~\ref{lem:multi-known-g-general}. Keep
\(\bar{\mathbf s}=\mathbf s\), and set
\[
\bar p_{[\bbb]}
=\mu_{\PPhi_{\Cdot,\bbb}(\bar{\QQ})},
\qquad [\bbb]\in\mathcal E(\bar{\QQ}).
\]
Distinct equivalence classes correspond to distinct ideal
patterns, so these proportions sum to one and give
\(\bar{\boldsymbol\mu}=\boldsymbol\mu\).
The channel factorization in the proof of
Lemma~\ref{lem:multi-known-g-general} gives equal response
distributions. Thus \(\QQ\) is not identifiable, and joint
identifiability fails.

The algebraic condition also excludes a zero column: adding a
one to such a column produces an inequivalent matrix whose ideal
column set contains the original one, by fixing that attribute
at one. It remains to consider a column \(k\) with its only one
in row \(j\). For \(a\in\{0,1\}\) and
\(\mathbf w\in\{0,1\}^{J-1}\), let
\[
\mu_{a,\mathbf w}
=\sum_{\substack{\aaa:\,
\PPhi_{j,\aaa}(\QQ)=a,\\
\PPhi_{\Cdot,\aaa}(\QQ_{[J]\setminus\{j\}})=\mathbf w}}
p_{\aaa}.
\]
If item \(j\) is capable, changing \(\alpha_k\) from one to
zero makes it incapable without changing any other ideal
response. Consequently,
\[
(1,\mathbf w)\in\operatorname{Cols}\{\PPhi(\QQ)\}
\quad\Longrightarrow\quad
(0,\mathbf w)\in\operatorname{Cols}\{\PPhi(\QQ)\},
\]
where the first coordinate denotes item \(j\).
Full grouped positivity gives
\(\mu_{1,\mathbf w}>0\Rightarrow\mu_{0,\mathbf w}>0\).

Keep \(\bar{\QQ}=\QQ\) and \(\bar s_h=s_h\) for
\(h\ne j\). Choose \(t>1\) such that
\[
0<t-1<\min_{\mathbf w:\,\mu_{1,\mathbf w}>0}
\frac{\mu_{0,\mathbf w}}{\mu_{1,\mathbf w}},
\]
and set
\[
1-\bar s_j-g_j=\frac{1-s_j-g_j}{t},
\qquad
\bar\mu_{1,\mathbf w}=t\mu_{1,\mathbf w},
\qquad
\bar\mu_{0,\mathbf w}
=\mu_{0,\mathbf w}-(t-1)\mu_{1,\mathbf w}.
\]
Since \(g_j<1-\bar s_j<1-s_j\leq1\), the perturbed item
parameter satisfies the model constraints. The probabilities
remain positive on exactly the same ideal patterns and sum to
one, so they define grouped proportions for \(\QQ\).
With the channel matrices from the proof of
Lemma~\ref{lem:multi-block-identifiability},
\[
M_j
=\bar M_j\begin{pmatrix}1&1-t\\0&t\end{pmatrix},
\qquad
\begin{pmatrix}\bar\mu_{0,\mathbf w}\\\bar\mu_{1,\mathbf w}\end{pmatrix}
=\begin{pmatrix}1&1-t\\0&t\end{pmatrix}
\begin{pmatrix}\mu_{0,\mathbf w}\\\mu_{1,\mathbf w}\end{pmatrix}.
\]
All other item channels are unchanged. Hence the response
distributions coincide, although \(\bar s_j\ne s_j\).
This proves necessity for joint identifiability.

It remains to prove necessity for identifiability of \(\QQ\)
when \(K\geq2\). Failure of the algebraic condition has already
been considered. Suppose that \(\QQ\) satisfies that condition,
but column \(k\) has its only one in row \(j\).
By Lemma~\ref{lem:multi-conj-two-column}, \(\QQ\) satisfies
the two-column condition. For every \(\ell\ne k\), the row
\((1,0)\) in columns \(k,\ell\) must therefore be row \(j\),
so \(q_{j\ell}=0\). Hence \(\qq_j=\ee_k\).
The two-column condition also gives a row \(h\ne j\), with
\(q_{hk}=0\).

Choose the true and alternative latent proportions as
\[
p_{\aaa}=2^{-K},
\qquad
\bar p_{\aaa}=
\begin{cases}
2^{-(K-1)},&\alpha_k=1,\\
0,&\alpha_k=0.
\end{cases}
\]
Thus the true attributes are mutually independent
\(\operatorname{Bernoulli}(1/2)\), whereas the alternative
fixes attribute \(k\) at one and preserves the distribution of
the remaining attributes. In particular, the true grouped
proportions are positive for every equivalence class.
Let \(\bar{\QQ}\) be obtained from \(\QQ\) by changing
\(q_{hk}\) from zero to one. Keep \(\mathbf g\) fixed and set
\[
1-\bar s_j-g_j=\frac{1-s_j-g_j}{2},
\qquad
\bar s_\ell=s_\ell,\quad \ell\ne j.
\]
These item parameters satisfy
\(g_j<1-\bar s_j<1-s_j\leq1\).

Write \(\P\) and \(\bar{\P}\) for probabilities under the
true and alternative specifications, respectively.
Since \(\qq_j=\bar{\qq}_j=\ee_k\),
\[
\P(R_j=1)
=g_j+\frac{1-s_j-g_j}{2}
=1-\bar s_j
=\bar{\P}(R_j=1).
\]
No other true item depends on attribute \(k\). Under the
alternative, its added requirement in row \(h\) is always
satisfied. The distribution of the remaining responses is
therefore unchanged:
\[
\P(\RR_{-j}=\rr_{-j})
=\bar{\P}(\RR_{-j}=\rr_{-j}),
\qquad \rr_{-j}\in\{0,1\}^{J-1}.
\]
Moreover, \(R_j\) is independent of \(\RR_{-j}\) under
both specifications. Consequently, for every
\(\rr\in\{0,1\}^J\),
\[
\begin{aligned}
\P(\RR=\rr)
&=\P(R_j=r_j)\P(\RR_{-j}=\rr_{-j})\\
&=\bar{\P}(R_j=r_j)\bar{\P}(\RR_{-j}=\rr_{-j})
=\bar{\P}(\RR=\rr).
\end{aligned}
\]
Since \(\bar{\QQ}\) has exactly one more entry equal to one
than \(\QQ\), we have \(\bar{\QQ}\nsim\QQ\).
Thus \(\QQ\) is not identifiable.

Finally, suppose that \(K\geq2\), \(\QQ\) satisfies the
algebraic condition, and \(\QQ\) has no pure rows.
By Lemma~\ref{lem:multi-conj-two-column}, \(\QQ\) satisfies
the two-column condition. If column \(k\) had its only one
in row \(j\), that condition would require \(q_{j\ell}=0\)
for every \(\ell\ne k\). This would give
\(\qq_j=\ee_k\), a pure row. Thus every column contains at
least two ones, as claimed in the main text.
\end{proof}

}

\end{document}